\documentclass[11pt]{article}
\usepackage{amsmath, amssymb, theorem, latexsym, epsfig, floatrow}
\PassOptionsToPackage{linktocpage}{hyperref}
\usepackage[hyperindex,breaklinks]{hyperref}
\usepackage[a4paper, total={6in, 8in}]{geometry}
 \usepackage{ifpdf}
 \usepackage[leftcaption]{sidecap}
\numberwithin{equation}{section}

\theoremstyle{plain}
\theorembodyfont{\itshape}
\newtheorem{theorem}{Theorem}[section]
\newtheorem{proposition}[theorem]{Proposition}
\newtheorem{lemma}[theorem]{Lemma}
\newtheorem{corollary}[theorem]{Corollary}
\newtheorem{problem}[theorem]{Problem}
\theorembodyfont{\rmfamily}
\newtheorem{definition}[theorem]{Definition}

\newtheorem{example}[theorem]{Example}
\newtheorem{remark}[theorem]{Remark}

\newenvironment{proof}{{\noindent \textbf{Proof}\,\,}}{\hspace*{\fill}$\Box$\medskip}

\def\sym{\operatorname{Sym}}
\def\cc{\mathbb C}
\def\mcf{\mathcal F}
\def\oc{\overline{\cc}}

\def\cp{\mathbb{CP}}
\def\wt#1{\widetilde#1}
\def\rr{\mathbb R}
\def\var{\varepsilon}

\def\Ga{\Gamma}

\def\hh{\mathbb H}
\def\nn{\mathbb N}

 \def\zz{\mathbb Z}
 \def\rp{\mathbb{RP}}

\def\La{\Lambda}
\def\ii{\mathbb I}
\def\mca{\mathcal A}
  \def\diag{\operatorname{diag}}
 \def\la{\lambda}
\def\mcc{\mathcal C}
\def\mcs{\mathcal S}

\def\mcm{\mathcal M}
\def\mcq{\mathcal Q}

\def\mcp{\mathcal P}
\def\mcn{\mathcal N}

\def\Xi{\mathcal Z}

\def\wh#1{\widehat#1}

\def\ii{\mathbb I}
\def\rrp{\mathbb P}

\title{On rationally integrable planar dual multibilliards and piecewise smooth projective billiards}
\author{Alexey Glutsyuk\thanks{CNRS, UMR 5669 (UMPA, ENS de Lyon), France. E-mail: 
aglutsyu@ens-lyon.fr} \thanks{HSE University, Moscow, Russia} \thanks{Kharkevich Institute for Information Transmission Problems (IITP, RAS), Moscow} \thanks{Supported by part by RFBR grant 20-01-00420}}
\begin{document}
\maketitle
\begin{abstract}  
 The billiard flow in a planar domain $\Omega$ 
acts on the  tangent bundle $T\rr^2|_\Omega$  
 as geodesic flow with reflections from the boundary. It 
 has the trivial first integral: squared modulus of the velocity. Bolotin's Conjecture, now a joint theorem of Bialy, Mironov and the author, deals with those billiards whose flow admits an additional integral that is polynomial in the velocity and whose restriction to the unit tangent bundle is non-constant. It states that 1) if the boundary of such a billiard is $C^2$-smooth, nonlinear and connected, then it is a conic; 2) if it is piecewise $C^2$-smooth and contains a nonlinear arc, then it consists of arcs of conics from a confocal pencil and segments of  "admissible lines" for the pencil; 3) the minimal degree of the additional integral is   either 2, or 4. In 1997 Sergei Tabachnikov introduced  {\it projective billiards:}  planar curves equipped with a transversal line field, which defines a  reflection acting on oriented lines and the projective billiard flow. They are common generalization of billiards on surfaces of constant curvature, but in general may have no canonical conserved quantity. 
In a previous paper the author  classified those $C^4$-smooth connected nonlinear planar projective billiards whose flow admits  a non-constant integral that is a rational $0$-homogeneous function of the velocity (with coefficients depending on the position): these billiards are called 
 {\it rationally $0$-homogeneously integrable.} It was shown  that: 1) the underlying curve is  a conic; 2) the  minimal degree of integral is equal to two, if the billiard is defined by a dual pencil of conics; 3) otherwise it can be arbitrary even number. In the present paper we classify {\it piecewise $C^4$-smooth} 
 rationally $0$-homogeneously integrable projective billiards.  Unexpectedly, we show that  such a billiard associated to a dual pencil of conics may have integral of minimal degree 2,  4, or {\bf 12.} For the proof of main results we prove dual results for  the  so-called dual multibilliards introduced in the present paper.
\end{abstract}
 \tableofcontents
\section{Introduction}
\subsection{Introduction, brief description of main results and plan of the paper}
Consider  a  planar billiard $\Omega\subset\rr^2$ bounded by a $C^2$-smooth strictly convex closed curve. 
Recall that its {\it caustic} is a curve $S\subset\rr^2$ 
 such that each tangent line to $S$ is reflected from the boundary $\partial\Omega$ 
  to a line tangent to $S$. A billiard is {\it Birkhoff integrable,} if an inner neighborhood of its boundary is foliated by 
  closed caustics, with boundary being a leaf of the foliation. This is the case in an elliptic billiard:  
   confocal ellipses form a foliation by closed caustics of a domain adjacent to the boundary ellipse. 
  The famous  open Birkhoff Conjecture states that {\it the only integrable billiards are ellipses.} See its brief survey in Subsection 1.7. 
  
  The billiard flow in a domain $\Omega$ acts on the tangent bundle $T\rr^2|_\Omega$ as follows. 
  A point $(Q,v)$, $Q\in\Omega$, $v\in T_Q\rr^2$, moves along a trajectory of geodesic flow 
  ($v$ remains constant as a vector in $\rr^2$, and $Q$ moves with constant speed $v$), until 
  $Q$ hits the boundary $\partial\Omega$. Then $v$ is reflected from the boundary to 
  the new vector $v^*\in T_Q\rr^2$ according to the standard reflection law: 
  the angle of incidence is equal to the angle 
  of reflection, and $||v^*||=||v||$. Afterwards the new point $(Q,v^*)$ moves 
  by the geodesic flow etc.  The billiard flow has trivial first integral $||v||^2$. 
  It is a well-known folklore fact that  Birkhoff integrability of a strictly convex bounded planar billiard is equivalent to the existence of a non-trivial first integral of the billiard flow independent with $||v||^2$ on a neighborhood of the unit tangent bundle to 
$\partial\Omega$ in $T\rr^2|_{\overline\Omega}$. 

A planar billiard is called {\it polynomially integrable,} if  its flow admits a first integral that is polynomial in the velocity (with coefficients depending on  $Q$) whose restriction to the unit tangent bundle is non-constant.
Sergei Bolotin suggested the following polynomial version of the Birkhoff Conjecture, which is now a theorem: a joint result of Mikhail Bialy, Andrei Mironov and the author, see \cite{bm, bm2, gl, gl2}. 
{\it A planar billiard with piecewise $C^2$-smooth boundary is polynomially integrable, if and only 
if  one of the following statements holds:

 1) if the boundary $\partial\Omega$ is $C^2$-smooth, connected and nonlinear, then it is a 
conic (or a connected component of a conic);

2) if $\partial\Omega$ is piecewise $C^2$-smooth, then it is a union of arcs of  conics from a 
confocal pencil
and maybe segments of so-called admissible lines for the pencil of conics in question. 

\noindent If a billiard has one of the above types 1) or 2), then the billiard flow has a nontrivial quadratic integral, unless the pencil in question consists of confocal parabolas and $\partial\Omega$ contains a segment of the line through the focus 
that is orthogonal to the common axis of the parabolas; in the latter case the minimal degree of 
integral is equal to four.}

The above parabolic example with degree four integral was discovered by A. Ramani,  A. Kalliterakis, B. Grammaticos, B. Dorizzi \cite{gram}.

  Sergei Tabachnikov suggested the generalization  of the Birkhoff Conjecture to projective billiards introduced by himself in 1997 in \cite{tabpr}. See the following definition and conjecture.

 \begin{definition} \cite{tabpr} A {\it projective billiard} is a smooth planar curve $C\subset\rr^2$ equipped with a transversal line field $\mcn$. 
 For every $Q\in C$ the {\it projective billiard reflection involution} at $Q$ acts on the space of lines through $Q$ as the affine involution 
 $\rr^2\to\rr^2$ that fixes the points of the tangent line to $C$ at $Q$, preserves the line $\mcn(Q)$ and acts on $\mcn(Q)$ as the central symmetry 
 with respect to the point\footnote{In other words, two lines $a$, $b$ through $Q$ are permuted by reflection at $Q$, if  and only if 
 the quadruple of lines $T_QC$, $\mcn(Q)$, $a$, $b$ is harmonic: there exists a projective involution of the space $\rp^1$ of lines through 
 $Q$ that fixes $T_QC$, $\mcn(Q)$ and permutes $a$, $b$.} $Q$. 
 In the case, when $C$ is a strictly convex closed curve,  the {\it projective billiard map} acts on the {\it phase cylinder:} 
 the space of oriented lines intersecting $C$. It sends an oriented line to its image under the above reflection involution at its last point 
 of intersection with $C$ in the sense of orientation. See Fig. 1.
 \end{definition} 
 \begin{figure}[ht]
\begin{center}
   \epsfig{file=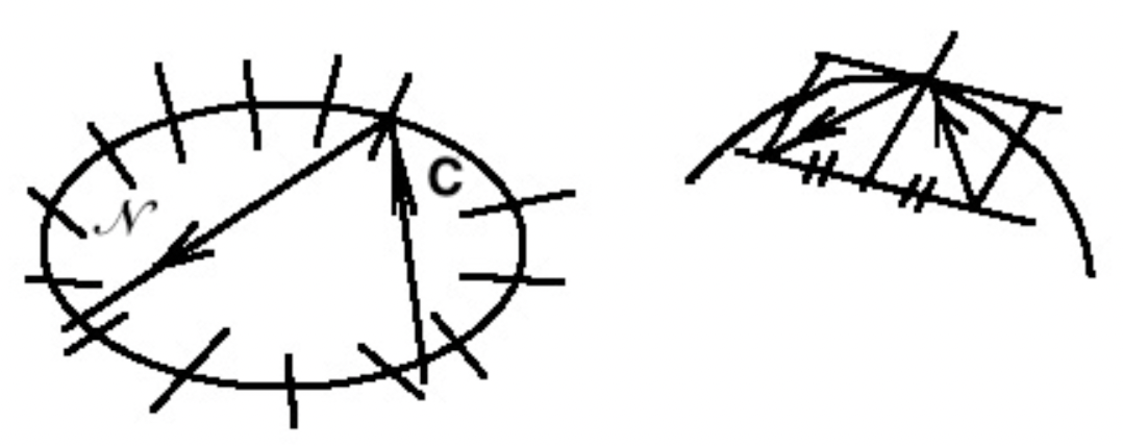, width=28em}
   \caption{The projective billiard reflection.}\label{fig1}
    \end{center} 
    \end{figure}
\begin{example} \label{exconst} 
A usual Euclidean planar billiard is a projective billiard with transversal line field being 
normal line field. Each billiard in a complete Riemannian surface $\Sigma$ of non-zero constant curvature 
(i.e., in sphere  $S^2$ and in  hyperbolic plane $\mathbb H^2$) also can be seen as a projective billiard, see \cite{tabpr}. 
Namely, consider $\Sigma=S^2$ as the unit sphere in the Euclidean space $\rr^3$, and  
$\Sigma=\mathbb H^2$  as the semi-pseudo-sphere $\{x_1^2+x_2^2-x_3^2=-1, \ x_3>0\}$ in the Minkovski space $\rr^3$ equipped with the form $dx_1^2+dx_2^2-dx_3^2$. The billiard in a
 domain $\Omega\subset\Sigma_+:=\Sigma\cap\{ x_3>0\}$  
is defined by reflection of geodesics from its boundary. The tautological projection $\pi:\rr^3\setminus\{0\}\to\rp^2$ sends $\Omega$ diffeomorphically 
onto a domain in the affine chart $\{ x_3=1\}$. It sends billiard orbits in $\Omega$ to orbits of the projective billiard on $C=\pi(\partial\Omega)$ with the transversal line field $\mcn$ on $C$ being  
the image of the normal line field to $\partial\Omega$ under the differential $d\pi$. 
\end{example}

The notion of caustic (integrability) of a projective billiard repeats the above similar notions for the usual billiards. 
{\bf Tabachnikov Conjecture}  states that {\it if a projective billiard is integrable, then the billiard boundary 
and the caustics are conics, whose dual conics form a pencil.}   It was stated in print in equivalent dual form (for dual billiards) in \cite{tab08}. 

The billiard flow of a projective billiard is defined analogously to the usual billiard flow, see \cite[pp. 958, 960]{tabpr}. But now as $Q$ hits the boundary (the underlying curve of the projective 
billiard), the initial velocity $v\in T_Q\rr^2$ is reflected to the new velocity $v^*$ by the 
projective billiard reflection. See Fig. \ref{figproj}. A projective billiard on a strictly convex 
closed planar curve is Birkhoff integrable, if and only if its flow admits a smooth first integral without critical points on a neighborhood of the unit tangent bundle to $\partial\Omega$ in $T\rr^2|_{\overline\Omega}$  that is a $0$-homogeneous function of the velocity. The proof of this 
equivalence statement repeats the proof of the similar above-mentioned folkolre fact for the usual billiards. 
\begin{figure}[ht]
\begin{center}
\epsfig{file=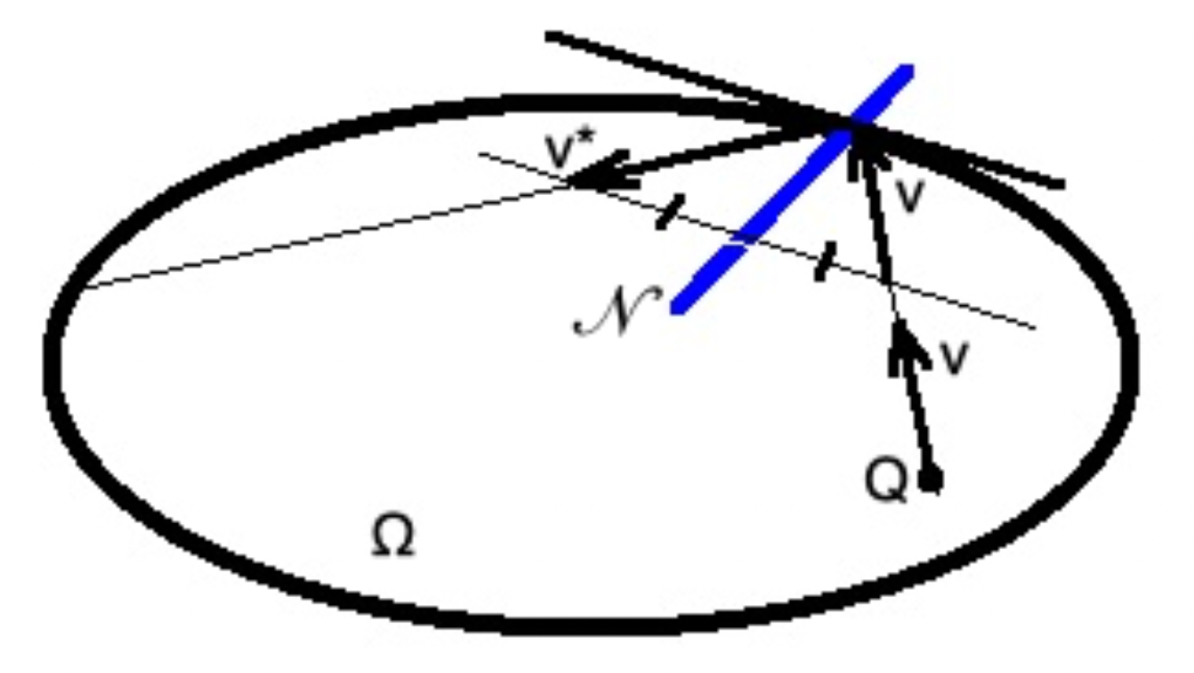, width=20em}
\caption{Projective billiard flow}
\label{figproj}
\end{center}
\end{figure}

\begin{remark}
 Each Euclidean planar billiard flow always has the trivial first integral $||v||^2$. 
But this is not true for   a generic projective billiard. 
\end{remark}
\begin{problem} \label{propol} Classify those projective billiards whose flow admits at least one nonconstant first integral that is polynomial in the velocity.
\end{problem}

In the present paper we classify rationally $0$-homogeneously integrable\footnote{The notion of polynomially integrable projective billiard, see Problem  \ref{propol}, is invariant under affine transformations. But a priori, it is not invariant under projective transformations, since in general, 
a projective 
transformation sending one projective billiard to another one does not conjugate the 
corresponding projective billiard flows. On the other hand, the notion of rationally $0$-homogeneously integrable billiard is invariant under projective transformations.  
The pullback of a rational $0$-homogeneous integral of a projective billiard under a projective transformation is an integral 
of its preimage.} 
 non-polygonal 
piecewise $C^4$-smooth projective billiards, see the following definition.

\begin{definition} A projective billiard is {\it rationally $0$-homogeneously integrable,} if its flow admits a non-constant first integral that 
depends on the velocity as a  rational $0$-homogeneous function, of degree uniformly bounded by 
some number $n$: 
a function $\Psi(Q,v)=\frac{P(v)}{T(v)}$, where $P$ and $T$ are homogeneous polynomials in $v$ of degree no greater than $n$ 
with coefficients depending on the position of the point $Q$. The maximal degree of the latter rational function 
through all $Q$ is called the {\it degree} of the rational integral.
\end{definition}

\begin{example} \label{exdual} Let $C, \Gamma\subset\rr^2_{x_1,x_2}$ be distinct  regular conics. 
The punctured conic $C\setminus\Gamma$ can be equipped with a unique projective billiard structure (transversal line field $\mcn$) so that the complexified 
 conic $\Gamma$ is its caustic. This means that for every $x\in C\setminus\Gamma$ 
 the two complex tangent lines to $\Gamma$ through $x$ and the lines $T_xC$, $\mcn(x)$ form 
 a harmonic quadruple. Thus constructed projective billiard on $C$ will be called a {\it dual pencil type projective billiard on a punctured conic $C$.} It is well-known (Tabachnikov's observation) that then every conic from the dual pencil containing $C$ and $\Gamma$ is a caustic in the above sense. The corresponding projective billiard flow  has a quadratic rational $0$-homogeneous integral. Namely, let $B,A\in GL_3(\rr)$ be such that in  homogeneous coordinates 
 $[x_1:x_2:x_3]$ in the ambient projective plane $\rp^2\supset\rr^2_{x_1,x_2}=\{ x_3=1\}$ 
 one has 
 $$C=\{<B^{-1}x,x>=0\}, \ \Gamma=\{<A^{-1}x,x>=0\}, \ <x,y>:=\sum_{j=1}^3x_jy_j.$$
 The  conics forming the dual pencil containing $C$ and $\Gamma$ are the conics 
  $$\mcc_\la:=\{<(B-\la A)^{-1}x,x>=0\}, \ \la\in\cc.$$
They are caustics  in the above sense. For every $(x_1,x_2)\in\rr^2$ and $(v_1,v_2)\in T_{(x_1,x_2)}\rr^2$ set  
 $$r:=(x_1,x_2,1), \ v:=(v_1,v_2,0),$$
 \begin{equation}\mcm=\mcm(r,v):=[r,v]=(-v_2,v_1,\Delta), \ \Delta:=x_1v_2-x_2v_1.\label{movect}
 \end{equation}
 Then the ratio 
 \begin{equation}\Psi(x,v):=\frac{<B\mcm,\mcm>}{<A\mcm,\mcm>}\label{dual-int}\end{equation}
 is a first integral of the projective billiard flow. This can be proved by 
 using Bolotin's argument \cite[p.22]{bolotin2} and will be exlained below, 
 in Remark \ref{removect}, via  duality.  
  \end{example} 
  \begin{remark} \label{remmom} A degree $n$ rational $0$-homogeneous integral  of a projective billiard is always a degree $n$ rational $0$-homogeneous function  in the moment 
 $\mcm=(-v_2, v_1,\Delta)$ with constant coefficients. See analogous statement for 
polynomial integrals of  planar billiards earlier proved by Bolotin in \cite{bolotin2}, and the general statement for projective billiards in \cite[proposition 1.23, statement 1)]{grat}. 
\end{remark}
 
In the previous paper  \cite[theorem 1.26]{grat}  the author classified rationally $0$-homogeneously 
integrable projective billiard structures on (germs of) nonlinear 
$C^4$-smooth connected planar curves. It was shown there that 

-  the underlying curve is always a conic;

-  besides Example \ref{exdual} of dual pencil type projective billiard, 
which has a quadratic integral, there are infinitely many other rationally $0$-homogeneously integrable projective billiard structures on a conic, called {\it exotic structures,} with minimal degree of integral being arbitrary even number greater than 3. 
These results of \cite{grat} are recalled  in Subsection 1.3. 

The main results of the present paper (Theorems \ref{thmd1pr}, \ref{thmd12pr}, \ref{thmd2pr} stated in Subsection 1.5) yield the classification of  rationally $0$-homogeneously integrable 
projective billiards {\it with piecewise $C^4$-smooth boundary} that  contains a nonlinear arc and maybe also straightline segments. They show that a piecewise $C^4$-smooth projective billiard is rationally $0$-homogeneously integrable, if and only if the following statements hold: 

1) The nonlinear part of the boundary lies either in one conic (equipped  with either a dual pencil, or exotic projective billiard structure), or   in a 
union of several conics lying the same dual pencil, with projective billiard structures defined by the 
same pencil.

2) If the projective billiard structure on the conical part of the boundary is defined by a dual pencil, then the projective billiard is of dual pencil type. This means that on each conic it is defined by the same dual pencil, and the collection of 
lines containing straightline segments of the boundary satisfies certain conditions. In particular, they belong to the list of so-called {\it admissible lines for the dual pencil} equipped with appropriate projective billiard structures defined by the pencil. 

3) If the conical part of the boundary lies in one conic equipped with an exotic projective billiard 
structure, then the ambient lines of boundary straightline segments (if any) belong to an explicit finite list of lines called {\it admissible for the exotic structure} in question. 

In the dual pencil case the projective billiard structure on each  admissible line is given by a line field that either is central (i.e., its lines pass through a given point), or consists of lines tangent to a given conic of the dual pencil. In the exotic case each admissible line is equipped with appropriate  central line field. 

Theorem \ref{thdegpr} gives a formula for the minimal degree of integral  of a dual pencil type billiard. It yields the following 
unexpected result: for a dual pencil type projective billiard defined by a generic pencil and containing straightline segments of appropriate collection of admissible lines, the minimal degree of integral is equal to {\bf 12.} 

Theorem \ref{tintpr} gives a formula for  degree 12 integral  in terms of just the dual pencil: it is an integral of every dual pencil type projective billiard defined by the given pencil. 

We also prove the following new result on  usual billiards on constant curvature surfaces.
     
     \begin{proposition} \label{proeucl} 
     A piecewise $C^2$-smooth non-polygonal billiard on a surface of constant curvature is rationally $0$-homogeneously integrable, if and only if it is polynomially integrable. In this case the
      minimal degree of its  non-trivial rational $0$-homogeneous integral is equal to either 
      2, or 4: to the minimal degree of a non-trivial nonlinear polynomial integral. Rational integral of minimal degree can be chosen to be $\frac{I(x,v)}{||v||^{2n}}$, where $I(x,v)$ is a nontrivial 
      homogeneous polynomial in $v$ integral of  degree $2n$, $n\in\{1,2\}$.
      \end{proposition}

    Thus, {\it rationally $0$-homogeneously integrable  projective 
     billiards of dual pencil type  with integrals of degree 12} presented and classified in this  
     paper {\it form an essentially  new class} of rationally 
     integrable projective billiards of dual pencil type, {\it not covered by the known list of polynomially 
     integrable billiards} on surfaces of constant curvature. 

For the proof of main results, we state (in Subsection 1.4) and prove their projective dual versions giving classification of the so-called rationally integrable  dual multibilliards defined below. 
Namely, consider the plane $\rr^2_{x_1,x_2}$ as the affine chart $\{ x_3=1\}$ in $\rp^2_{[x_1:x_2:x_3]}$. Consider the correspondence sending a two-dimensional vector subspace in the Euclidean space $\rr^3_{x_1,x_2,x_3}$ to its orthogonal complement, which is a one-dimensional subspace. The induced correspondence 
$\rp^{2*}\to\rp^2$ sending a projective line in $\rp^2$ (projectivized two-dimensional subspace) 
to a point in $\rp^2$ (the projectivized orthogonal) is a projective duality called the {\it orthogonal polarity.} For every smooth planar curve $C$ its dual $\gamma=C^*$ is the family of points  
dual to the projective  lines  tangent to $C$. For every $Q\in C$ let $L_Q$ denote the projective line tangent to $C$ at $Q$. The dual to the space $\rp^1$ of lines through $Q$ 
is the projective line $L_P$ tangent to $\gamma$ at $P=L_Q^*$. The projective billiard reflection acting on  $\rp^1$ is conjugated to a projective involution $\sigma_P:L_P\to L_P$.  This implies that  the dual to a projective billiard on a curve $C$ containing no straightline segments is  the dual billiard introduced by Tabachnikov, see the following definition.

\begin{definition} \label{defdual} \cite[definitions 1.6, 1.17]{grat}  A {\it real (complex) dual billiard} is a smooth (holomorphic) curve 
$\gamma\subset\rp^2 (\cp^2)$ where for each point $P\in\gamma$ 
the real (complex) projective line $L_P$ tangent to $\gamma$ at $P$ is equipped with a projective involution 
$\sigma_P:L_P\to L_P$ fixing $P$; 
the involution family (called  {\it dual billiard structure}) is parametrized by tangency points $P$. 
\end{definition}

It appears that the  projective dual to a  rationally $0$-homogeneously integrable projective billiard 
on a nonlinear curve is a rationally integrable dual billiard, see the next definition. 

\begin{definition} \cite[definition 1.12]{grat} A dual billiard is called {\it rationally integrable,} if there exists a non-constant rational function $R$ on the plane (called {\it integral}) whose restriction to each line tangent to the underlying curve is invariant under the corresponding involution:
 \begin{equation}(R\circ\sigma_P)|_{L_P}=R|_{L_P}.\label{ratint}\end{equation}
\end{definition}

The dual to a dual pencil type projective structure on a conic is a pencil type dual billiard, see the 
following example. 
\begin{example} \label{expencil} \cite[example 1.14]{grat} 
A dual billiard is said to be of {\it pencil type,} if the underlying curve $\gamma$ is a (punctured) conic and there 
exists a pencil of conics containing $\gamma$ such that for every $P\in\gamma$ the involution $\sigma_P$ 
permutes the intersection points of the line $L_P$ with each conic of the pencil (or fixes the intersection point, 
if it is unique). As was observed by S.Tabachnikov, conversely, 
for every conic $\gamma$ and every  pencil containing $\gamma$, for every $P$ in the conic $\gamma$  
punctured in at most 4 complex base points of the pencil there exists a projective involution $\sigma_P:L_P\to L_P$ 
satisfying the above  condition, and thus, a well-defined pencil type dual billiard  on the punctured conic $\gamma$. It is rationally integrable with a quadratic integral, which is the ratio of two quadratic polynomials vanishing on two arbitrary given conics of the pencil (Tabachnikov's observation, see \cite[example 1.14]{grat}). 
\end{example} 

\begin{remark}\label{removect}
For every nonzero vector $v\in T_{(x_1,x_2)}\rr^2$ the duality sends the projective line tangent to $v$ to the point $[\mcm_1:\mcm_2:\mcm_3]\in\rp^2$, where $\mcm$ is the moment 
vector, see (\ref{movect}). Thus, $[\mcm_1:\mcm_2:\mcm_3]$ can be treated as  natural homogeneous 
coordinates on the projective plane containing the dual billiard  to a given projective billiard on a curve $C$. This implies that {\it each rational integral of the dual billiard written as a rational $0$-homogeneous function of $\mcm=[r,v]$ is a rational $0$-homogeneous integral of the  projective billiard flow.} A more general statement for multibilliards with its converse will be proved later, in Subsection 3.1. 
 It is well-known that the duality given by the orthogonal 
polarity transforms each conic $\gamma=\{<A\mcm,\mcm>=0\}$, $A\in GL_3(\cc)$, to the conic given by 
the inverse matrix, $C=\{<A^{-1}x,x>=0\}$. This together with the above statement implies that the function (\ref{dual-int}) is indeed an integral of the dual pencil type 
projective billiard in Example \ref{exdual}. \end{remark}

\begin{example} \label{exdualint}  
If a dual billiard on a nonlinear curve has a polynomial integral, then it is an {\it outer billiard:} the  
projective involution of  each tangent line is its central symmetry with respect to the tangency point  
(see \cite[example 1.13]{grat}).  The dual billiard to a usual planar billiard is Bialy--Mironov angular billiard \cite{bm}.
\end{example} 

As shown in \cite[theorem 1.16]{grat}, a dual billiard on a nonlinear $C^4$-smooth germ $\gamma$ of planar curve billiard is rationally integrable, if and only if  $\gamma$ is  a conic, and the dual billiard structure either is of pencil type, or belongs to an explicit infinite list of 
rationally integrable dual billiards on conic given in loc. cit.   These results are recalled below, in Subsection 1.2. 
 
Let the underlying curve of a projective billiard contain a segment $I$ of a line $q$. 
The duality sends $q$ to a point $Q\in\rp^2$, and the family of  points in $I$ is sent to 
an open subset in the space $\rp^1$ of lines through $Q$. For every point $r\in I$ the duality conjugates the  projective billiard reflection at $r$ acting on the space $\rp^1$ of lines through $r$ 
to a projective involution of the dual line $\ell=r^*$. Thus, the projective billiard structure on $I$  
is transformed by duality to a  dual billiard structure at the point $Q$, and {\it the dual to 
a piecewise smooth projective billiard is a dual multibilliard,} see the following two definitions.

\begin{definition} \label{defdp} A {\it dual billiard structure} at a point
$Q\in\rp^2(\cp^2)$ is a family of  projective involutions 
$\sigma_{Q,\ell}:\ell\to\ell$ acting on real (complex) projective lines $\ell$ through $Q$;  $\sigma_{Q,\ell}$ 
are defined on an open subset  $U\subset\rp^1(\cp^1)$ of the space of lines through $Q$. No regularity of the family 
$\sigma_{Q,\ell}$ is assumed. 
\end{definition}
\begin{definition}  
A {\it real (complex) dual multibilliard} is a (may be infinite) collection of smooth (holomorphic) nonlinear connected 
curves $\gamma_j$ 
and points $Q_s$ in $\rp^2 (\cp^2)$ (called {\it vertices}), where 
each curve $\gamma_j$ and each point $Q_s$ 
are equipped with a dual billiard structure. (A  priori  a curve (vertex)  may appear in a multibilliard several times as the same curve (point) equipped with different dual billiard structures.)  
\end{definition}

We show that the dual to a rationally $0$-homogeneously integrable piecewise smooth 
projective billiard is a rationally integrable dual multibilliard, see the following definition.

 \begin{definition} A dual multibilliard is {\it rationally integrable,} if there exists a non-constant rational function 
 on $\rp^2(\cp^2)$ whose restriction to each tangent line to every curve $\gamma_j$ 
 is invariant under the corresponding involution, and the same statement holds for its restriction to each line 
 $\ell$ through any vertex $Q$, where the corresponding involution $\sigma_{Q,\ell}$ is defined.
 \end{definition}

 The main results on classification of rationally integrable real and complex  dual multibilliards 
 (not reduced to  just one curve without vertices) are 
 given by Theorems \ref{thmd1}, \ref{thmd12} and \ref{thmd2} in Subsection 1.4. In particular we show that in every rationally integrable multibilliard 
 the  dual billiard structure at each vertex  is either a global projective involution of the ambient 
 projective plane, or a birational involution of de Jonqui\`eres type whose fixed point locus is a conic. 
 The main results on classification of rationally $0$-homogeneously integrable piecewise $C^4$-smooth 
 projective billiards presented in Subsection 1.5 will be deduced from the above-mentioned results 
 on multibilliards via projective duality arguments. 

 Plan of proof of main results is presented in Subsection 1.6. 
A historical survey is given in Subsection 1.7. 
The main results  are proved in Sections 2 (for multibilliards) and  3 (for projective billiards). 
In Section 4 we prove formulas for degree 12 integrals (Theorems \ref{proformint}, \ref{tintpr} and Lemma \ref{lintpr}) and  present examples of  projective billiards with integrals of degrees 4 and 12. 
  
  \subsection{Previous results 1: classification  of real and complex 
  rationally integral dual billiards on one curve}
  
    \begin{theorem} \label{tgerm} \cite[theorem 1.16]{grat} Let $\gamma\subset\rr^2\subset\rp^2$ be a $C^4$-smooth connected 
    nonlinear  (germ of) curve  equipped with a rationally integrable dual   billiard structure. 
 Then $\gamma$ is a conic, and the dual   billiard structure has one of the 
 three following types (up to real-projective equivalence):
 
  1) The  dual billiard  is of conical pencil type, defined by a real pencil, and has a quadratic integral.
    
  2) There exists an affine chart $\rr^2_{z,w}\subset\rp^2$ 
  in which $\gamma=\{ w=z^2\}$ and such that for every $P=(z_0,w_0)\in\gamma$ the 
  involution $\sigma_P:L_P\to L_P$ is given by one of the following formulas: 
  
  a) In the coordinate 
  $$\zeta:=\frac z{z_0}$$
  $$\sigma_P:\zeta\mapsto\eta_\rho(\zeta):=\frac{(\rho-1)\zeta-(\rho-2)}{\rho\zeta-(\rho-1)},$$
\begin{equation} \rho=2-\frac 2{2N+1}, \ \text{ or } \ \rho=2-\frac1{N+1} \  \text{ for some 
  } N\in\nn.\label{rhoval}\end{equation} 
  
  b) In the coordinate 
  $$u:=z-z_0$$
  \begin{equation} \sigma_P: u\mapsto-\frac u{1+f(z_0)u},\label{sigmaef}\end{equation}
   \begin{equation} f=f_{b1}(z):=\frac{5z-3}{2z(z-1)} \text{ (type 2b1))}, \  \text{ or } \ 
   f=f_{b2}(z):=\frac{3z}{z^2+1} \text{ (type 2b2))}.\label{sigma2}\end{equation}

  c) In the above coordinate $u$ the involution $\sigma_P$ takes the form (\ref{sigmaef}) with 
   \begin{equation} f=f_{c1}(z):=\frac{4z^2}{z^3-1} \text{ (type 2c1))}, \  \text{ or } \ 
   f=f_{c2}(z):=\frac{8z-4}{3z(z-1)} \text{ (type 2c2))}.\label{sigma3}\end{equation}
   
   d) In the above coordinate $u$ the involution $\sigma_P$ takes the form (\ref{sigmaef}) with 
   \begin{equation} f=f_d(z)=\frac4{3z}+\frac1{z-1}=\frac{7z-4}{3z(z-1)} \ \ \text{ (type 2d)}.\end{equation}
  \end{theorem}

  \smallskip

  {\bf Addendum to Theorem \ref{tgerm}.} {\it Every dual   billiard structure 
  on $\gamma$ of type 2a) has a rational first integral $R(z,w)$ of the form} 
  \begin{equation}R(z,w)=\frac{(w-z^2)^{2N+1}}{\prod_{j=1}^N(w-c_jz^2)^2}, \ \ 
  c_j=-\frac{4j(2N+1-j)}{(2N+1-2j)^2}, \ \text{ for } \rho=2-\frac2{2N+1};\label{exot1}
  \end{equation}
  \begin{equation}R(z,w)=\frac{(w-z^2)^{N+1}}{z\prod_{j=1}^N(w-c_jz^2)}, \ \ 
  c_j=-\frac{j(2N+2-j)}{(N+1-j)^2}, \ \text{ for } \rho=2-\frac1{N+1}.\label{exot2}\end{equation}
   {\it The dual billiards of types 2b1) and 2b2) have respectively the integrals}
  \begin{equation}R_{b1}(z,w)=\frac{(w-z^2)^2}{(w+3z^2)(z-1)(z-w)}, 
  \label{exo2bnew} \end{equation}
   \begin{equation}R_{b2}(z,w)=\frac{(w-z^2)^2}{(z^2+w^2+w+1)(z^2+1)}.
  \label{exo2bnew2} \end{equation}
  {\it The dual billiards of types 2c1), 2c2) have respectively the integrals}
  \begin{equation}R_{c1}(z,w)=\frac{(w-z^2)^3}{(1+w^3-2zw)^2},\label{exo2b}\end{equation} 
  \begin{equation}R_{c2}(z,w)=\frac{(w-z^2)^3}{(8z^3-8z^2w-8z^2-w^2-w+10zw)^2}.\label{exoc2}\end{equation}
{\it The dual billiard  of type 2d) has the integral}
 \begin{equation}R_{d}(z,w)=\frac{(w-z^2)^3}{(w+8z^2)(z-1)(w+8z^2+4w^2+5wz^2-14zw-4z^3)}.\label{exodd}\end{equation}

 \begin{theorem} \label{tcompl} \cite[theorem 1.18 and its addendum]{grat}. 
 Every regular  (germ of) connected holomorphic curve in $\cp^2$ (different from a straight line) 
 equipped with a 
 rationally integrable  complex dual   billiard structure is a conic. 
 Up to complex-projective equivalence, the corresponding billiard structure has one of the types 1) (now defined by a complex pencil of conics), 2a), 2b1), 2c1), 2d) listed in Theorem \ref{tgerm}, with 
 a rational integral as in its addendum.  The billiards of  types 2b1), 2b2), see (\ref{sigma2}), are 
 complex-projectively equivalent, and so are  billiards  2c1), 2c2). 
  \end{theorem}

\subsection{Previous results 2: classification of rationally $0$-homogeneously integrable projective billiards on one curve}

The notion of rationally $0$-homogeneously integrable projective billiard also makes sense  
for a projective billiard structure on an arc of planar curve $C$ (or a germ of curve), with projective billiard flow 
defined in a (germ of) domain adjacent to $C$.  

\begin{remark} \label{remoside}
The property of  rational $0$-homogeneous 
integrability of a projective billiard on a 
 curve $C$  is independent on  the side from $C$  on which the billiard domain is chosen: 
an integral for one side is automatically an integral for the other side. See \cite[proposition 1.23, statement 2)]{grat}. 
\end{remark}

 \begin{theorem} \label{tgermpr} \cite[theorem 1.26]{grat} Let $C\subset\rr^2_{x_1,x_2}$ be a nonlinear $C^4$-smooth germ of 
 curve equipped with a transversal line field $\mcn$. 
  Let the corresponding germ of projective billiard be $0$-homogeneously rationally integrable. 
 Then $C$ is a conic;  the line field $\mcn$ extends to a global analytic transversal line field   on all of $C$  punctured in at most four points;  
the corresponding projective billiard has one of the following types up to projective equivalence.
 
 1) A dual pencil type projective billiard.

2)  $C=\{x_2=x_1^2\}\subset\rr^2_{x_1,x_2}\subset\rp^2$, and 
the  line field $\mcn$  is directed by one of the following vector fields at points of the conic $C$: 
 \smallskip
 
2a) \ \ \ \ \ $(\dot x_1,\dot x_2)=(\rho,2(\rho-2)x_1)$, 
$$\rho=2-\frac2{2N+1} \text{ (case 2a1), \ or }  \ \ \rho=2-\frac1{N+1} \text{ (case 2a2), }  \ \ N\in\mathbb N,$$
 the  vector field 2a) has quadratic first integral $\mcq_{\rho}(x_1,x_2):=\rho x_2-(\rho-2)x_1^2.$

2b1) \ $(\dot x_1, \dot x_2)=(5x_1+3,
2(x_2-x_1))$, \ \ \ 2b2)  \ $(\dot x_1,\dot x_2)=(3x_1, 2x_2-4)$, 
\smallskip

2c1) \ $(\dot x_1, \dot x_2)=(x_2, 
x_1x_2-1)$, \ \ \ \ \ \ 2c2) \ $(\dot x_1, \dot x_2)=(2x_1+1, 
x_2-x_1)$.

2d) \ $(\dot x_1, \dot x_2)=(7x_1+4, 2x_2-4x_1)$.
\end{theorem}
{\bf Addendum to Theorem \ref{tgermpr}.} {\it The projective billiards from Theorem \ref{tgermpr} have the 
following $0$-homogeneous rational integrals:

Case 1): A ratio of two homogeneous quadratic polynomials in $(v_1,v_2,\Delta)$, 
$$\Delta:=x_1v_2-x_2v_1.$$
Case 2a1), $\rho=2-\frac2{2N+1}$:
\begin{equation}\Psi_{2a1}(x_1,x_2,v_1,v_2):=\frac{(4v_1\Delta-v_2^2)^{2N+1}}{v_1^2\prod_{j=1}^N(4v_1\Delta-c_jv_2^2)^2}. \label{r2a1v}\end{equation} 

Case 2a2), $\rho=2-\frac1{N+1}$:
\begin{equation}\Psi_{2a2}(x_1,x_2,v_1,v_2)=\frac{(4v_1\Delta-v_2^2)^{N+1}}{v_1v_2\prod_{j=1}^N(4v_1\Delta-c_jv_2^2)}. \label{r2a2v}\end{equation}
The $c_j$ in (\ref{r2a1v}), (\ref{r2a2v}) are the same, as in (\ref{exot1}) and (\ref{exot2}) respectively.

Case 2b1): 
 \begin{equation}\Psi_{2b1}(x_1,x_2,v_1,v_2)=\frac{(4v_1\Delta-v_2^2)^2}{(4v_1\Delta+3v_2^2)(2v_1+v_2)(2\Delta+v_2)}.\label{r2b1v}\end{equation}
 
 Case 2b2): 
 \begin{equation}\Psi_{2b2}(x_1,x_2,v_1,v_2)=\frac{(4v_1\Delta-v_2^2)^2}{(v_2^2+4\Delta^2+
4v_1\Delta+4v_1^2)(v_2^2+4v_1^2)}.
 \label{r2b2v}\end{equation}
 
 \begin{equation}\text{Case 2c1): } \ \ \ \ \ \ \  \ \ \ \ \Psi_{2c1}(x_1,x_2,v_1,v_2)=\frac{(4v_1\Delta-v_2^2)^3}{(v_1^3+\Delta^3+
 v_1v_2\Delta)^2}.\label{r2c1v}\end{equation}
 
 Case 2c2): 
  \begin{equation}\Psi_{2c2}(x_1,x_2,v_1,v_2)=\frac{(4v_1\Delta-v_2^2)^3}{(v_2^3+2v_2^2v_1+(v_1^2+2v_2^2+5v_1v_2)\Delta+v_1\Delta^2)^2}.
 \label{r2c2v}\end{equation}
 
 Case 2d):}  $\Psi_{2d}(x_1,x_2,v_1,v_2)$
 \begin{equation}=\frac{(4v_1\Delta-v_2^2)^3}{(v_1\Delta+2v_2^2)(2v_1+v_2)(8v_1v_2^2+2v_2^3+(4v_1^2+5v_2^2+28v_1v_2)\Delta+16v_1\Delta^2)}.
 \label{r2dv}\end{equation}

 \subsection{Main results: classification of rationally integrable planar dual multibilliards with $C^4$-smooth curves} 
 Each curve of a rationally integrable dual multibilliard is a conic, being itself an integrable dual billiard, see 
 Theorem \ref{tgerm}. 
 The first results on classification of rationally integrable dual multibilliards presented below deal  with those multibilliards whose curves are conics lying in one pencil, 
 equipped with dual billiard structure defined by the same pencil. They state that its vertices should be  
  admissible for the pencil. To define admissible vertices, let us first introduce the following definition.

 \begin{definition} \label{extypes}  A {\it projective angular symmetry} centered at a point $A\in\cp^2$ 
 is a non-trivial projective involution  
 $\sigma_A:\cp^2\to\cp^2$ fixing $A$ and each line through $A$. It is known to have a fixed point line $\La\subset\cp^2$ disjoint from $A$. Its restrictions to lines throughs $A$  define  
 a dual billiard structure at $A$. 
 \end{definition}
 
\begin{example} \label{extypes2}
 Let now $A\in\cp^2$ and let $\mcs\subset\cp^2$ be a (may be singular) conic disjoint from $A$. 
 There exists a projective angular symmetry centered at $A$ and permuting the intersection points 
 with $\mcs$ of each line through $A$, called {\it $\mcs$-angular symmetry,} 
see \cite[definition 2.4]{gl2}.  
\end{example}

\begin{definition} \label{extypes3} 
 Let now $A\in\cp^2$, $\mcs\subset\cp^2$ be a regular conic through $A$, and  $L_A$ the projective 
tangent line to $\mcs$ at $A$. The {\it degenerate $\mcs$-angular symmetry} centered at $A$ is 
 the involution $\sigma_A=\sigma_A^{\mcs}$ acting on the complement $\cp^2\setminus(L_A\setminus\{ A\})$ 
 that fixes $A$, fixes each line $\ell\neq L_A$ through $A$ and whose restriction to $\ell$ is the projective involution 
fixing $A$ and the other point of the intersection $\ell\cap\mcs$. It is known to be a birational 
map\footnote{The projective angular symmetry and the degenerate $\mcs$-angular symmetry belong to the well-known class of 
birational involutions in the Cremona group, called {\it de Jonqui\`eres involutions.}  
Each of them is an involution fixing all but a finite number of lines   through a given point  
$A\in\cp^2$.  See, e.g.,  \cite[p. 422, example 3.1]{blanc}.} $\cp^2\to\cp^2$.  
\end{definition}
\begin{definition} A dual billiard structure at a point $A\in\cp^2$ is called {\it global (quasi-global)} if it is given by a projective 
angular symmetry (respectively, degenerate $\mcs$-angular symmetry) centered at $A$. 
\end{definition}
    
 \begin{definition} \label{multip} Consider a  complex pencil of conics in $\cp^2$. 
 A vertex, i.e., a point of the ambient plane equipped with a complex dual billiard structure, is called 
 {\it admissible} for the pencil, if it belongs 
  to the following list of vertices split in two types:   {\it standard}, or {\it skew}.
 
 Case a):  a pencil of conics  through 4 distinct  points $A$, $B$, $C$, $D$,  see 
 Fig. \ref{fig2},  which will be referred to as a {\it non-degenerate} pencil. 
 
 a1) The {\it standand vertices:} $M_1=AB\cap CD$, $M_2=AD\cap BC$, $M_3=AC\cap BD$ 
 equipped with the global dual billiard  structure given by the projective angular symmetry  
 $\sigma_{M_j}=\sigma_{M_j}^{M_iM_k}$,   $i,k\neq j$, $i\neq k$, centered at $M_j$ with fixed point line $M_iM_k$.

 a2) The {\it skew vertices} $K_{EL}$, numerated by unordered pairs of points
  $E,L\in\{ A,B,C,D\}$, $E\neq L$: $K_{EL}$ is the intersection point 
 of the line $EL$ with the line $M_iM_j$ such that $M_i,M_j\notin EL$.
   The involution $\sigma_{K_{EL}}$ is the projective angular symmetry $\rrp^2\to\rrp^2$ centered at $K_{EL}$ 
   with fixed point line     $ST$, $\{ S,T\}=\{ A,B,C,D\}\setminus\{ E,L\}$.
    \begin{SCfigure}
    \centering
    \caption{}\label{fig2}
   \epsfig{file=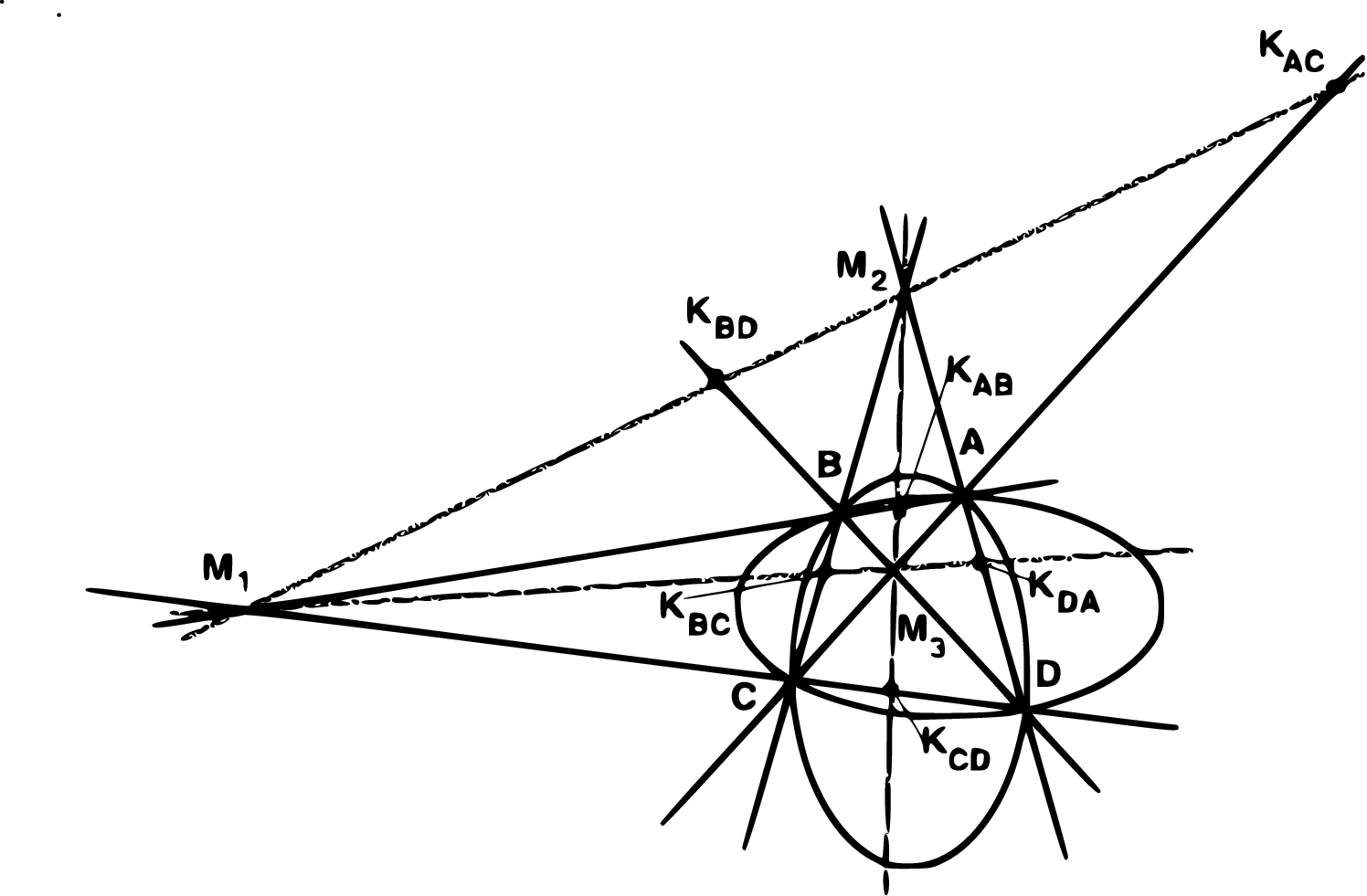, width=25em}\hspace{0.5cm}
\end{SCfigure}

Case b): a pencil of conics through 3 points $A$, $B$, $C$  tangent at the point $C$ to the same line $L$. See Fig. \ref{fig3}.

b1) One {\it standard vertex} $M=AB\cap L$, equipped with the projective angular symmetry 
$\sigma_M:\rrp^2\to\rrp^2$ centered at $M$ with fixed point line $CK_{AB}$. The point $K_{AB}$ is defined as follows. 

b2) The {\it skew vertex} $K_{AB}\in AB$ such that the projective involution 
$AB\to AB$ fixing $M$ and $K_{AB}$ permutes $A$ and $B$. That is, the points $M$, $K_{AB}$, $A$, $B$ 
form a harmonic quadruple. The dual billiard structure at  $K_{AB}$ is  
 given by the projective angular symmetry $\sigma_{K_{AB}}:\rrp^2\to\rrp^2$ centered at $K_{AB}$ with fixed point 
  line $L$.

b3) The {\it skew vertex} $C$ equipped with the projective angular symmetry $\sigma_C=\sigma_C^{AB}:\rrp^2\to\rrp^2$ 
centered at  $C$ with fixed point line $AB$. 
 
b4) The {\it skew vertex} $C$ equipped with a degenerate $\mcs$-angular symmetry  $\sigma_C=\sigma_C^\mcs$ 
centered at $C$, defined by arbitrary given regular conic $\mcs$ of the pencil; see Definition \ref{extypes3}. 
This yields a one-parametric family of  
quasi-global dual billiard structures at $C$. 
\begin{SCfigure}
\centering
\caption{}\label{fig3}
   \epsfig{file=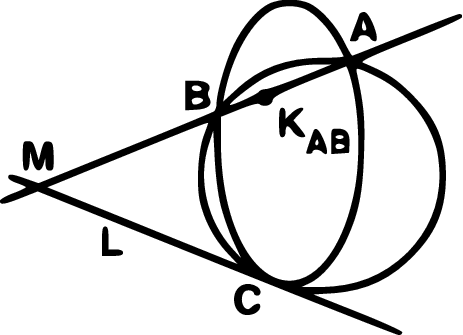, width=10em}
\end{SCfigure}

Case c): a pencil of conics through two given points $A$ and $C$ that are tangent at them to two given lines 
$L_A$ and $L_C$ respectively. See Fig. \ref{fig4}.

c1) {\it Standard vertices:} $M=L_A\cap L_C$ and any point $M'\in AC$, $M'\neq A,C$.
The vertex $M$ is equipped with the projective angular symmetry $\sigma_M$ centered at $M$ 
with fixed point line $AC$. The vertex  $M'$ is equipped with the  $(L_A\cup L_C)$-angular symmetry 
centered at $M'$,  which permutes the  intersection points of each line 
through $M'$  with the lines $L_A$ and $L_C$. 

c2) {\it Skew vertices equipped with global dual billiard structures:} the points $A$ and $C$. 
The dual billiard structure at $A$ ($C$) is the projective angular symmetry 
centered at $A$ ($C$) with fixed point line $L_C$ (respectively, $L_A$). 

c3) {\it Skew vertices} $A$ and $C$; $A$ ($C$) being equipped with a degenerate $\mcs_A$ ($\mcs_C$)-angular 
symmetry centered at $A$ ($C$), 
defined by any regular conic $\mcs_A$ ($\mcs_C$) of the pencil. This yields a one-parametric family of 
quasi-global dual billiard 
structures at each one of the vertices $A$, $C$, as in b4). 
\begin{SCfigure}
\centering
\caption{}\label{fig4}
   \epsfig{file=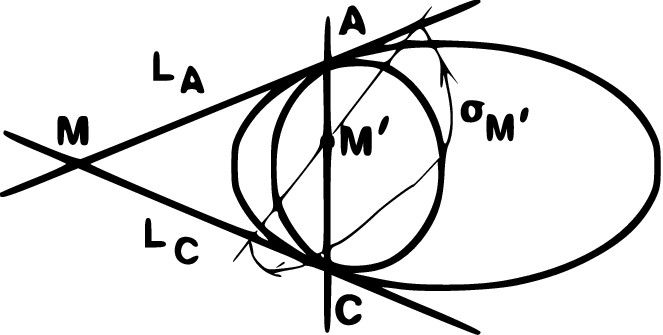, width=15em}
\end{SCfigure}

Case d): pencil of conics through two distinct points $A$ and $B$, 
tangent to each other at $A$ with contact of order three; let $L$ denote their common 
tangent line at $A$. See Fig. \ref{fig5}. 

d1) The {\it skew vertex} $A$,  equipped with a quasi-global dual billiard structure: 
a degenerate $\mcs$-angular symmetry $\sigma_A^\mcs$ centered at $A$  
defined by any   regular  conic $\mcs$ from the pencil.

d2) Any point $C\in L\setminus\{ A\}$, called a {\it skew vertex}, equipped with a 
projective angular symmetry $\sigma_C$ centered at $C$ with fixed point line $AB$. 
\begin{SCfigure}
\centering
\caption{}\label{fig5}
   \epsfig{file=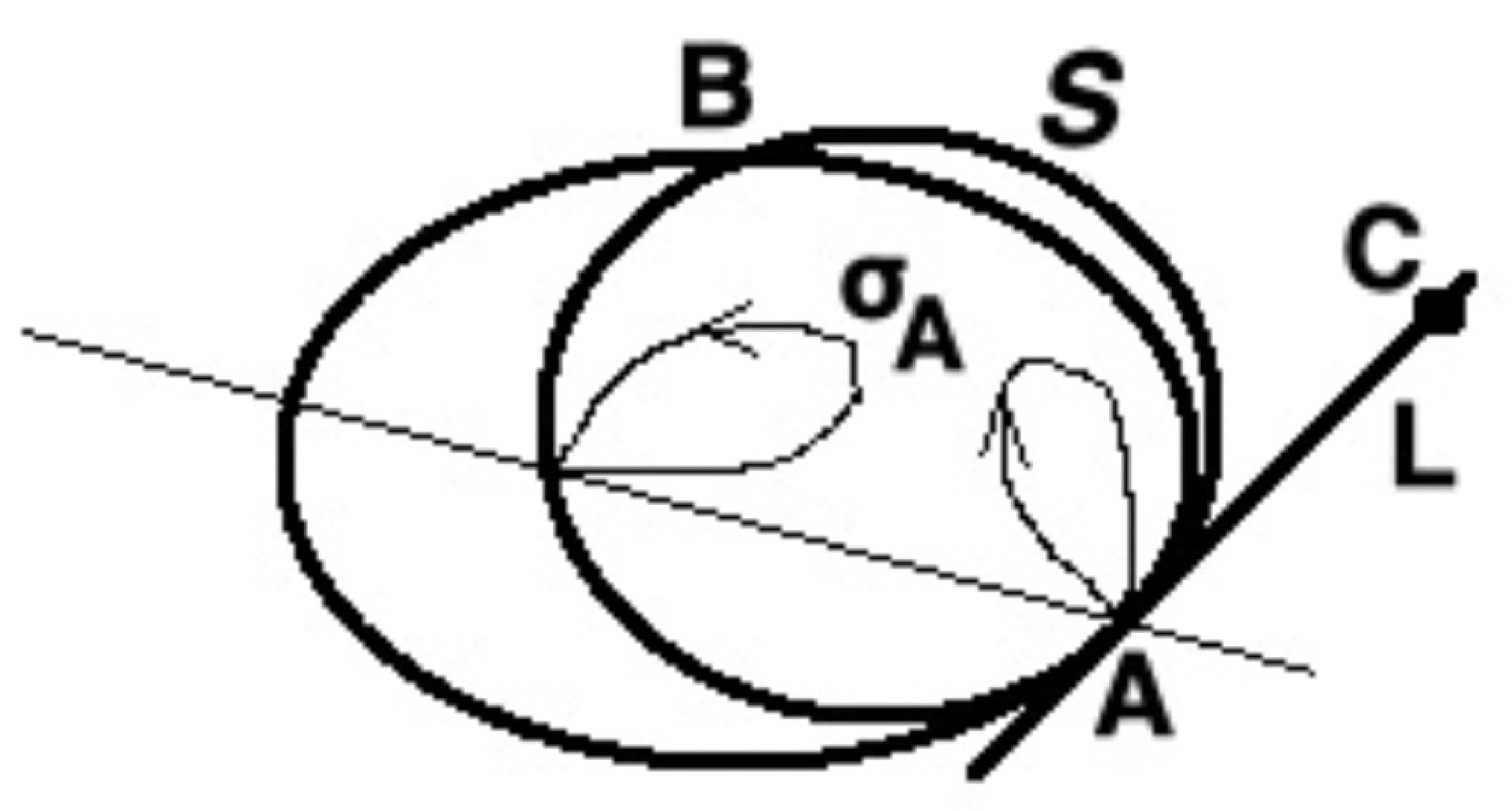, width=15em}
  \end{SCfigure}
\indent Case e): pencil of conics through one given point $A$, tangent to each other with 
contact of order four. See Fig. \ref{fig6}. Let $L$ denote their common tangent line at $A$. 

e1) The {\it skew vertex} $A$ equipped with a degenerate $\mcs$-angular symmetry $\sigma_A^\mcs$ centered at $A$ 
defined by any given regular conic $\mcs$ of the pencil. 

e2) Any point $C\in L\setminus\{ A\}$, called a {\it standard vertex,} equipped with a projective angular symmetry 
$\sigma_C$ centered at $C$. Its fixed point line is the set of those points $D\in\rrp^2$ for which the line $CD$ 
is tangent to the conic of the pencil through $D$ at $D$ (including $D=A$). 
   \begin{SCfigure}
   \centering
   \caption{} \label{fig6}
   \epsfig{file=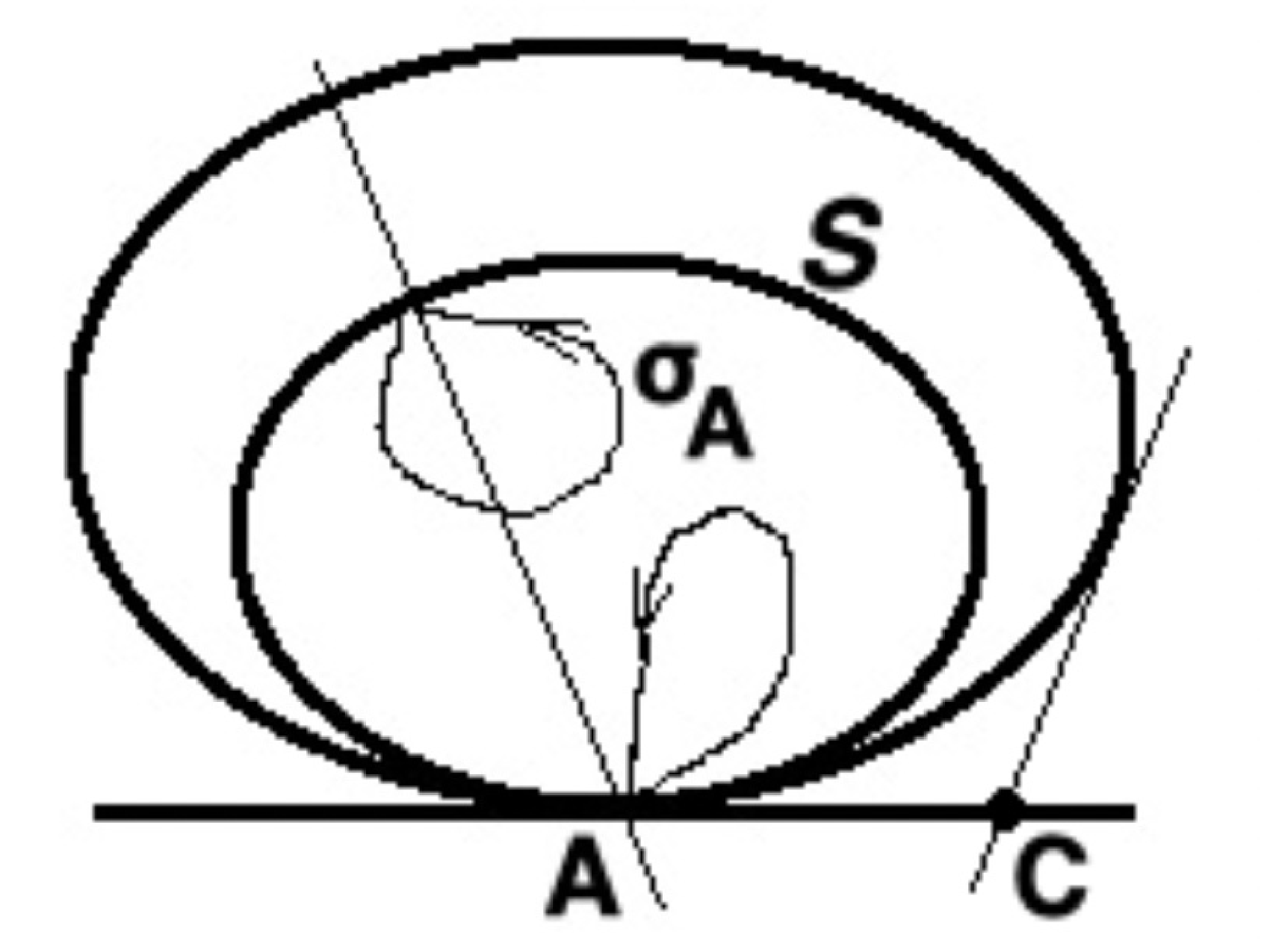, width=12em}  
\end{SCfigure}

\noindent The definition of real (standard or skew) admissible vertex for a real pencil of conics in $\rp^2$ is analogous. 
\end{definition}

\begin{definition} \label{ddist} Consider a dual multibilliard formed by some conics  and maybe by some vertices. 
Let the dual billiard structure at each vertex (if any) be either global, 
or quasi-global, and that on each conic be either of pencil type, or as in Theorem \ref{tgerm}, Case 2). 
We say that its two conics 
(vertices) are {\it distinct,} if  they either are geometrically distinct, or  coincide 
as conics (vertices) but have different dual billiard structures. 
\end{definition}

\begin{definition} \label{multipb} 
A (real or complex) dual multibilliard is said to be {\it of pencil type,} if the following conditions hold. 
 
 1) All its curves are conics lying in one  pencil, and their dual billiard structures 
 are defined by the same pencil. (Case of one conic 
  equipped with a dual billiard structure of pencil type is possible.)  
 
 2) All its vertices are admissible for the pencil.

3) If the multibilliard contains a skew vertex equipped with a quasi-global dual billiard structure 
(which may happen only for a degenerate pencil), then 
it contains no other skew vertex, with the following exceptions: 

- in Case c) the skew vertex collection is allowed to be the pair of vertices $A$ and $C$ 
equipped with quasi-global structures defined by one and the same (but arbitrary) 
regular conic $\mcs=\mcs_A=\mcs_C$ of the pencil; 

- in Case d) the skew vertex collection is allowed to be a pair of vertices $A$ and $C$ defined by 
 any  given regular conic $\mcs$ of the pencil: 
the vertex $A$ is equipped with the quasi-global $\mcs$-dual billiard structure; 
 the vertex $C$ is  the intersection point 
of the line $L$  with the line tangent to $\mcs$ at $B$, equipped with the projective angular symmetry with fixed 
point line $AB$ (it coincides with the $\mcs$-angular symmetry centered at $C$).

4) In Case d) the multibilliard may contain at most one vertex $C\in L\setminus\{ A\}$.
 
5) Each skew admissible vertex that a priori admits several possible dual billiard structures listed in Definition \ref{multip}  
 is allowed to be included in the multibilliard with no more than one dual billiard structure. 
\end{definition}

Well-definedness of the above notion of admissible vertex and pencil type dual multibilliard  in the real case 
is implied by the following proposition. 
\begin{proposition} \label{preal} Consider a real pencil of conics in $\rp^2$ with {\bf complex} base points. 

1) An admissible vertex equipped with a global projective involution is real, if and only if  its   involution is real. A quasi-global involution associated to a skew vertex is real, if and only if the vertex and the fixed point conic $\mcs$  are both real (i.e., invariant under complex conjugation). 

2) Case ot type a). At least one standard vertex $M_j$  is real. (Hovewer in general  $M_j$  
  ($K_{EL}$) are  not necessarily all real.)

3) Case of  type b). 
All  admissible vertices $C$, $M$, $K_{AB}$ are always real. 

4) Case of type c). The standard vertex $M$ is real. But  
the vertices $A$, $C$, $M'$ are not necessarily real.  

5) Case of type d) or e). The  admissible 
vertex $A$ is always real.  
\end{proposition}

 \begin{theorem} \label{thmd1} Let a (real or complex) dual multibilliard on a collection of 
  real $C^4$-smooth (or holomorphic) nonlinear connected  curves $\gamma_j$ and some vertices 
 be rationally integrable. Then the following statements hold. 
 
 1) Each curve $\gamma_j$ is a conic equipped with a dual billiard structure either of pencil type, or as in Theorem 
 \ref{tgerm}, Case 2). 
 
 2) If the multibilliard contains at least two distinct conics (in the sense of  
 Definition \ref{ddist}), then all the conics $\gamma_j$ lie in the same pencil, and the dual billiard structures on them are defined by the same pencil.
 \end{theorem}
 
 \begin{theorem} \label{thmd12}
 Let in a dual multibilliard all the curves be conics lying in the same pencil. Let they be 
 equipped with the dual billiard structure defined by the same pencil. (Case of 
 one conic equipped with a pencil type dual billiard structure is possible.) 
 Then the multibilliard is rationally integrable, if and only if it is of pencil type, see 
  Definition \ref{multipb}.
 \end{theorem}
 \begin{theorem} \label{thdeg} 
 The minimal degree of a rational integral of a pencil type 
 multibilliard is  
 
 (i) degree two, if it contains no skew vertices;
 
 {\bf (ii)  degree 12,} if the pencil has type a) and the multibilliard contains a pair of  
{\bf neighbor} skew vertices: namely, a pair of vertices of type $K_{EL}$ and $K_{ES}$ 
 for some three distinct $E, L, S\in\{ A, B, C, D\}$.
 
 (iii) degree four in any other case.
  \end{theorem} 
  The next theorem yields a formula for integral of degree 12 of 
  pencil type multibilliards for pencils of type a). To state it, let 
  us introduce the following notations. Let $\rp^2_{[y_1:y_2:y_3]}$ 
  denote the ambient projective plane of the multibilliard, 
  considered as the projectivization of the space $\rr^3_{y_1,y_2,y_3}$. 
  For every projective line $X$ 
  let $\pi^{-1}(X)\subset\rr^3$ denote the 
  corresponding two-dimensional subspace. Let $\xi_{X}(Y)$ denote a non-zero linear functional
  vanishing on $\pi^{-1}(X)$. It is well-defined up to constant factor.
  
  \begin{theorem} \label{proformint} Consider a pencil of conics through four distinct base points 
  $A$, $B$,  $C$, $D$. Set $M_1=AB\cap CD$, $M_2=BC\cap AD$, $M_3=AC\cap BD$. 
  
  1) The functionals 
  $\xi_{EL}$ corresponding to the lines $EL$ through  distinct  points $E, L\in\{ A, B, C, D\}$ can be normalized 
  by constant factors so that 
  \begin{equation}\xi_{AB}\xi_{CD}+\xi_{BC}\xi_{AD}+\xi_{AD}\xi_{BC}=0.
  \label{normfunct}\end{equation}
  
  2) If (\ref{normfunct}) holds, then for every $\mu\in\cc\setminus0$ the degree 12 rational function 
  \begin{equation}\prod_{\{ EL; FN\}\neq\{ E'L'; F'N'\}}\left(\frac{\xi_{EL}\xi_{FN}}{\xi_{E'L'}\xi_{F'N'}}(Y)+\mu)\right)
  \label{integr12}\end{equation}
  is a first integral of every pencil type multibilliard defined by the given pencil. Here the product is taken over ordered pairs $(\{ EL; FN\}, \{ E'L'; F'N'\})$  
  of two-line sets  with $\{ E,L,F,N\}=\{ E',L',F',N'\}=\{ A, B, C, D\}$. In  Theorem \ref{thdeg}, Case (ii)  this is a minimal degree integral. 
  \end{theorem}

 \begin{definition} \label{defex} A rationally integrable real (complex) 
 dual billiard structure on conic that is not 
 of pencil type, see Theorem \ref{tgerm}, Case 2), 
 will be called {\it exotic.} The singular points of the dual billiard structure 
 (which are exactly the indeterminacy points of the corresponding integral $R$ from the Addendum to 
 Theorem \ref{tgerm}) will be called the {\it base points.}
 \end{definition}
 
 \begin{corollary} \label{corex} Let a rationally integrable real (complex) 
 multibilliard be not of pencil type. 
 Then it contains only one curve, namely, a conic equipped with an exotic 
 dual billiard structure,  
 and maybe some vertices. 
 \end{corollary}
 
 \begin{theorem} \label{thmd2} A (real or complex) 
 multibilliard consisting of one  conic $\gamma$ equipped with an exotic dual billiard structure from Theorem \ref{tgerm}, Case 2), and maybe some vertices is rationally integrable, if and only if the collection of vertices 
  either is empty, or consists of the so-called {\bf admissible vertices} $Q$ defined below,   being 
equipped with the $\gamma$-angular symmetry $\sigma_Q$:
 
(i) Case of type 2a) dual billiard on $\gamma$. The   {\bf unique  admissible vertex} is 
the intersection point $Q=[1:0:0]$  of the 
$z$-axis and the infinity line; one has $\sigma_Q(z,w)=(-z,w)$ in the chart $(z,w)$. See Fig. \ref{figa}. 
 In Subcase 2a1), 
 when $\rho=2-\frac2{2N+1}$, the  
 function $R(z,w)$ from (\ref{exot1}) is a rational integral of the multibilliard $(\gamma, (Q,\sigma_Q))$ 
 of minimal degree: $\deg R=4N+2$. In Subcase 2a2), when $\rho=2-\frac1{N+1}$, 
  the function $R^2(z,w)$ with $R$ the same, as in (\ref{exot2}), is a  
 rational integral of  $(\gamma, (Q,\sigma_Q))$ of minimal 
 degree: $\deg R^2=4N+4$.

 (ii)  Case of  type 2b1) or 2b2). There are three base points. 
 One of them, denoted $X$, is the intersection point of two lines contained in the polar locus $R=\infty$. The {\bf unique 
 admissible vertex} $Q$ is the intersection point of two lines: 
  the tangent line to $\gamma$ at $X$ and the line through the two other base points. 
  In Case 2b1) one has $Q=(0,-1)$. In Case 2b2)  one has $Q=[1:0:0]$, $\sigma_Q(z,w)=(-z,w)$. 
 The  corresponding 
  rational function $R$, see (\ref{exo2bnew}), (\ref{exo2bnew2}) is a rational integral of the 
 multibilliard $(\gamma,(Q,\sigma_Q))$ of minimal degree: $\deg R=4$. See Fig. \ref{figbc}.

 (iii) Case of  type 2c1) or 2c2).  There are  three complex base 
 points.  There are {\bf three admissible vertices.} Each of them 
 is the intersection point of a line through two base points and the tangent line to $\gamma$ 
 at the other one.  In Case 2c1) the point $(0,-1)$ is the unique real admissible vertex. In Case 2c2) all the admissible   vertices are real: they are $(0,-1)$, $(1,0)$, $[1:1:0]$, see  Fig. \ref{figbc}.
 The function $R$ is a degree 6 rational integral of the multibilliard formed by the conic $\gamma$ and arbitrary 
  admissible vertex collection. 
  
  (iv) Case of type 2d). No admissible vertices.
 \end{theorem}
 \begin{figure}
 \begin{center}
   \epsfig{file=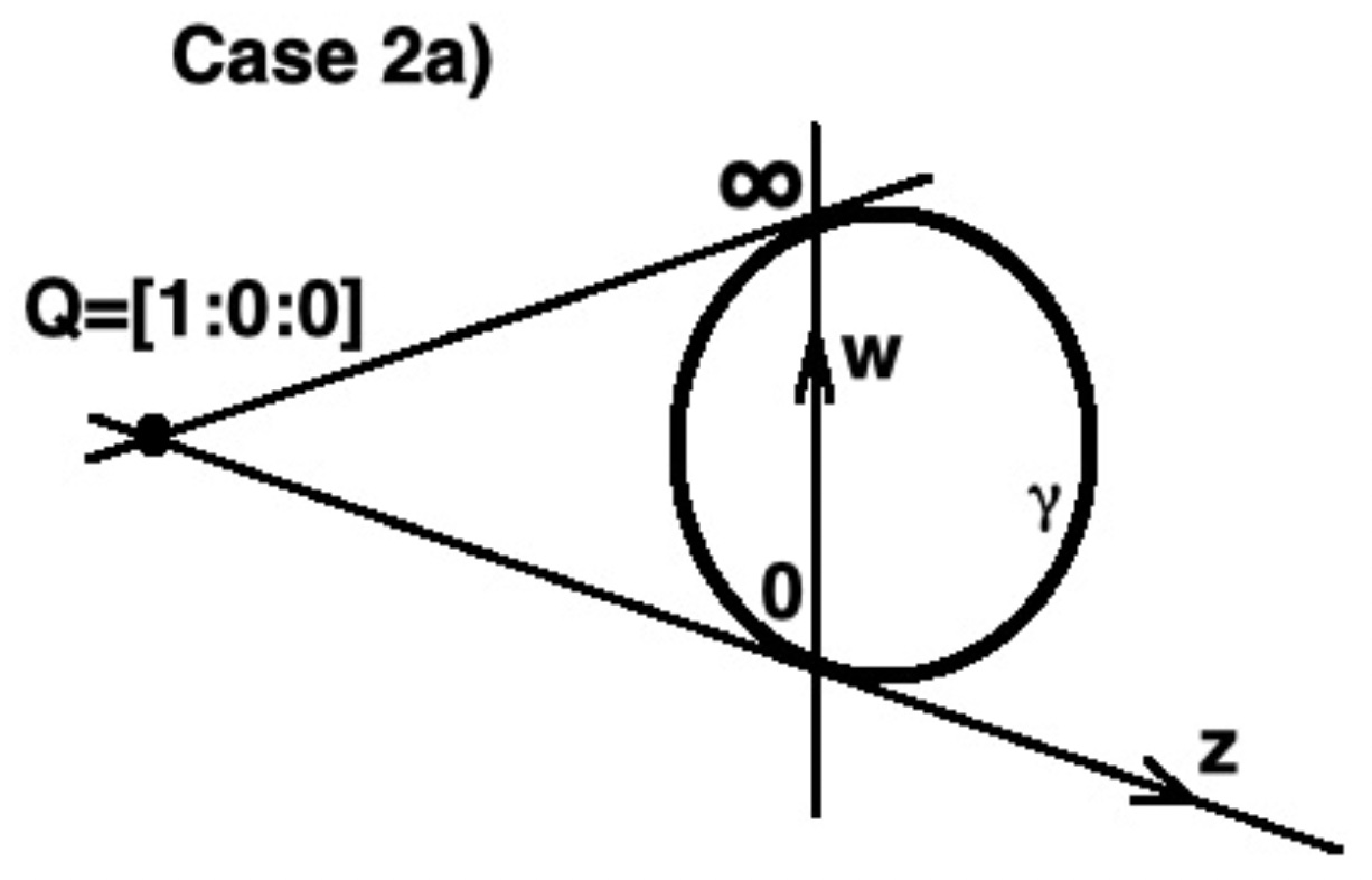, width=15em}
   \caption{The only admissible vertex in Case 2a) is the point $Q=[1:0:0]$.}\label{figa}
   \end{center}
\end{figure}
\begin{figure}
  \begin{center}
   \epsfig{file=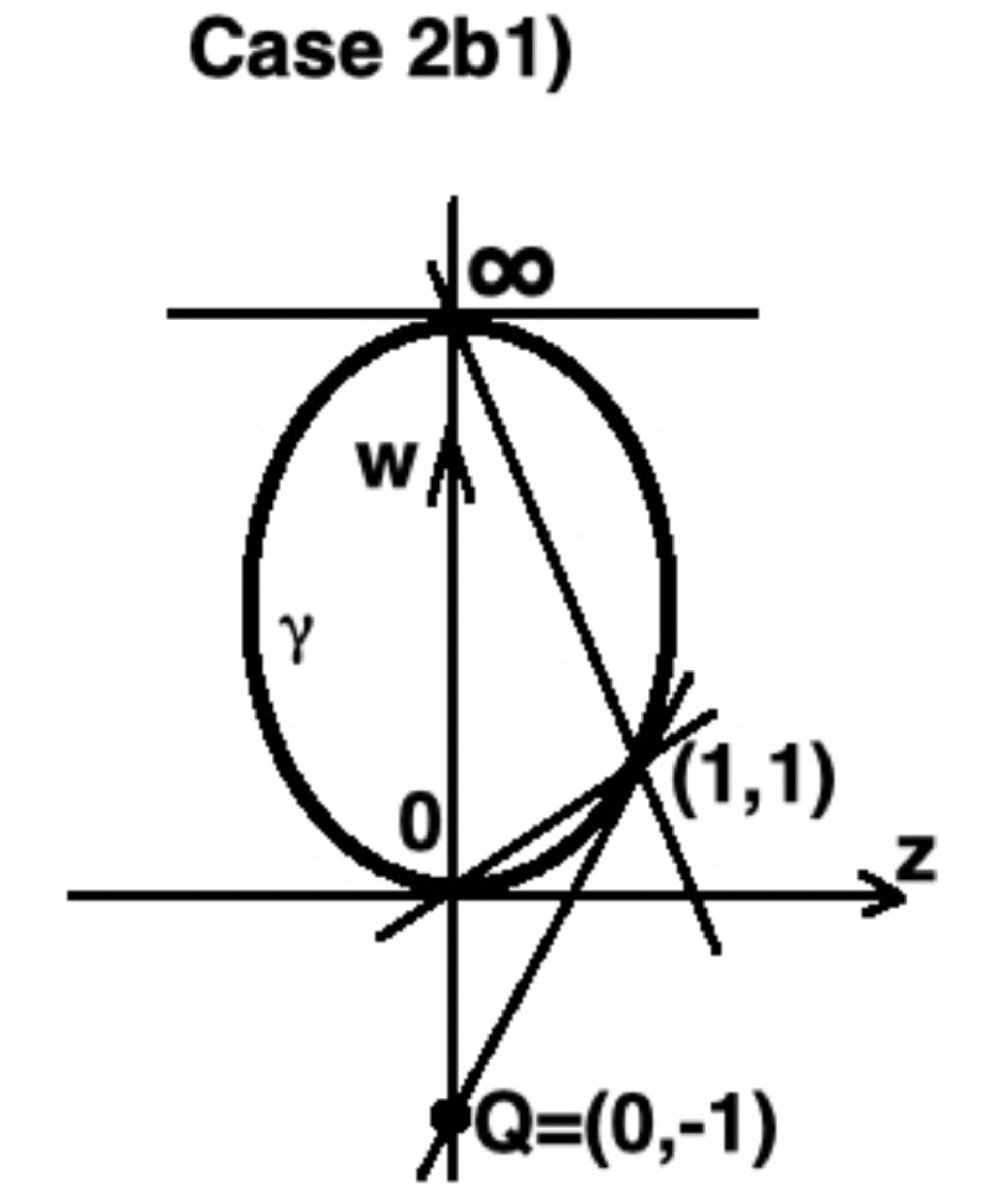, width=10em}
   \hspace{4em}
   \epsfig{file=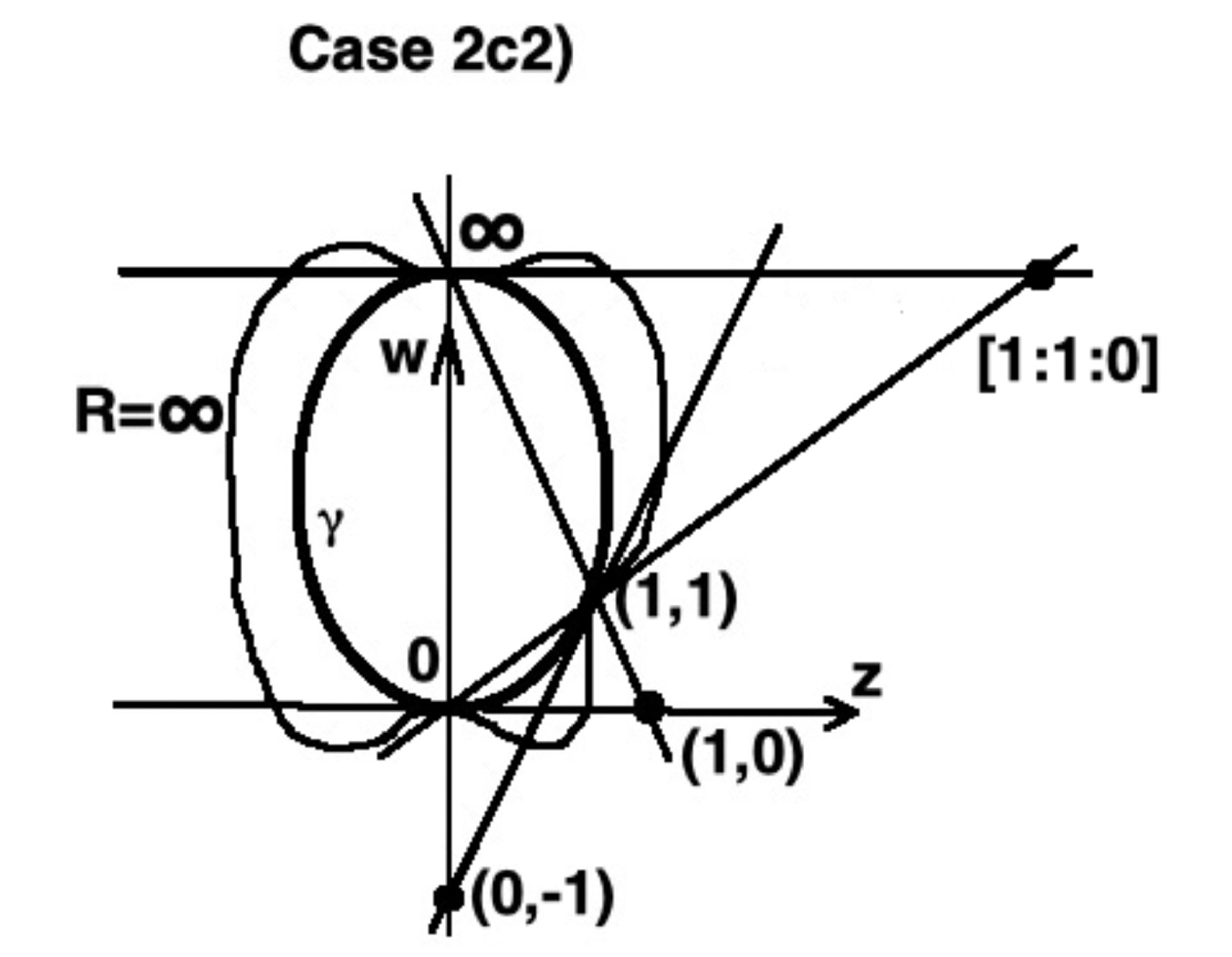, width=16em}
   \end{center}
   \caption{The admissible vertices in Cases 2b1) and 2c2) are marked in bold.}\label{figbc}
\end{figure}

 \begin{proposition} \label{pthmd2}
 1) The two  multibilliards of type (ii) (Cases 2b1), 2b2)) with one admissible vertex 
 are complex-projectively equivalent. 2) Two multibilliards  of 
 type (iii) (either of different subtypes 2c1), 2c2), or of the same subtype) 
 are complex-projectively equivalent, if and only if they have the same number of vertices. 3) Two real  multibilliards of type (iii) with at least one vertex are real-projectively equivalent, if and only if  they have either both subtype 2c1) and one real admissible vertex, 
 or both subtype 2c2) and the same number (arbitrary, from 1 to 3) of real admissible vertices.
 \end{proposition}
 
 Proposition \ref{pthmd2} follows from the last statement of Theorem \ref{tcompl} and the fact that 
 the dual  billiard of type 2c2) has real order three projective symmetry cyclically permuting admissible points.
 
 \subsection{Application: classification of rationally $0$-homogeneously integrable piecewise $C^4$-smooth non-polygonal projective billiards}
 
 Let us recall the following definition.
 
 \begin{definition} \cite[p. 962]{tabpr} The {\it central} projective billiard structure on a  curve $C\subset\rr^2$ with center $O\in\rr^2$ is 
 the field of lines  through $O$ on $C$. 
 \end{definition}
 
 Everywhere below for two projective lines $e$ and $f$ by $ef$ we will denote their intersection point.
 
 \begin{definition} \label{multipr} Consider a complex dual pencil of conics: a family of conics whose dual form a 
 pencil. Let $\ell$ be 
 a real or complex line equipped with a projective billiard structure; the corresponding (real or complex) line field is well-defined either on all of $\ell$, or 
on $\ell$ punctured at one point called singular. The line $\ell$ is said to be {\it admissible} for the dual pencil, if it 
belongs to the following list of lines equipped with projective billiard structures, called either standard, or skew.

 Case a): dual pencil of conics tangent to four distinct complex lines $a$, $b$, $c$, $d$. 
 
 a1) The three {\it standard admissible lines} are the lines  $m_1$, $m_2$, $m_3$ 
 through the points $ab$ and $cd$, 
 the points $bc$ and $ad$, the points $ac$ and $bd$ respectively. The line $m_1$ is equipped with projective 
 billiard structure centered at $m_2m_3$, and the projective billiard structures on $m_2$, $m_3$ are defined 
 analogously. 
 
 a2) Let $k_{bc}$  denote the line through 
 the points $m_1m_3$ and $bc$, equipped with the projective billiard structure centered at $ad$: the  field of lines
  through $ad$.  
 Let $k_{ad}$ be the line through $m_1m_3$ and $ad$,  equipped with the projective billiard structure centered at $bc$.  
 The other lines $k_{ef}$, $e,f\in\{ a,b,c,d\}$, $e\neq f$, equipped with central projective billiard structures are defined analogously. We identify $k_{ef}$ with $k_{fe}$. All the six lines $k_{ef}$ thus constructed 
  are called {\it skew admissible lines.} See Fig. \ref{fig-a-d}. 
  
  Case b): dual pencil of conics tangent to three distinct lines $a$, $b$, $c$ and having common tangency point $C$ with $c$. See Fig. \ref{fig-b-d}. 
  
  b1) The {\it skew line} $c$ equipped with the field of lines through  $ab$. 
  
  b2) The {\it skew line} $k$ such that the quadruple of lines $a$, $b$, $m$, $k$ through the point $ab$  
  is harmonic. It is equipped with the field of lines through $C$. Here $m$ is the line through $C$ and $ab$. 
  
  b3) The {\it standard line} $m$  with the field of lines through the  point $ck$. 
  
  b4) For arbitrary given regular conic $\mcs$ of the dual pencil the line $c$ equipped with the field of lines tangent 
  to $\mcs$ is a {\it skew line}.
  
  Case c): dual pencil of conics tangent to each other at two points $A$ and $B$. Led $a$ and $b$ denote the 
  corresponding tangent lines. See Fig. \ref{fig-c-d}. 
  
  c1) The {\it standard line} $m=AB$  with the field of lines through  $ab$. 
  
  c2) The {\it skew lines} $a$ and $b$ equipped with the fields of lines through $B$ and $A$ respectively. 
  
  c3) Fix arbitrary line $c\neq a,b$ through $ab$. Let $Z\in m$ denote the point such that the quadruple of points 
  $cm, Z, B, A\in m$ is harmonic. The line $c$ equipped with the field of lines through $Z$ is called a {\it skew line.} 
  
  c4) Fix regular conics $\mcs_A$, $\mcs_B$ from the dual pencil. The lines $a$ and $b$, each being equipped with the field 
  of lines tangent to $\mcs_A$ (respectively $\mcs_B$) at points distinct from $A$ and $B$ respectively are called {\it skew lines.}
  
  Case d): dual pencil of conics tangent to a given line $a$ at a given point  $A$, having triple contact 
  between each other at $A$, and tangent to another given line $b\neq a$. See Fig. \ref{fig-d-d}. 
  
  d1) The {\it skew line} $a$ equipped with the line field tangent to a given (arbitrary) regular conic $\mcs$ from the pencil.
  
  d2) Any line $c\neq a$ through $A$ called {\it skew}, equipped with the field of lines through the point $ab$.
  
Case e): dual pencil of conics tangent to each other at a point $A$ with order 4 contact. Let $a$ denote their 
common tangent line at $A$. See Fig. \ref{fig-e-d}.

e1) The {\it skew line} $a$ equipped with the line field tangent to a given (arbitrary) regular conic $\mcs$ from the pencil.

e2) Any line $b$ through $A$ called {\it skew}, equipped with the field of lines tangent to the conics of the pencil 
at points of the line $b$: these tangent lines pass through the same point $C=C(b)\in a$. 
  \end{definition}
   \begin{figure}
 \begin{center}
 \vspace{-1cm}
   \epsfig{file=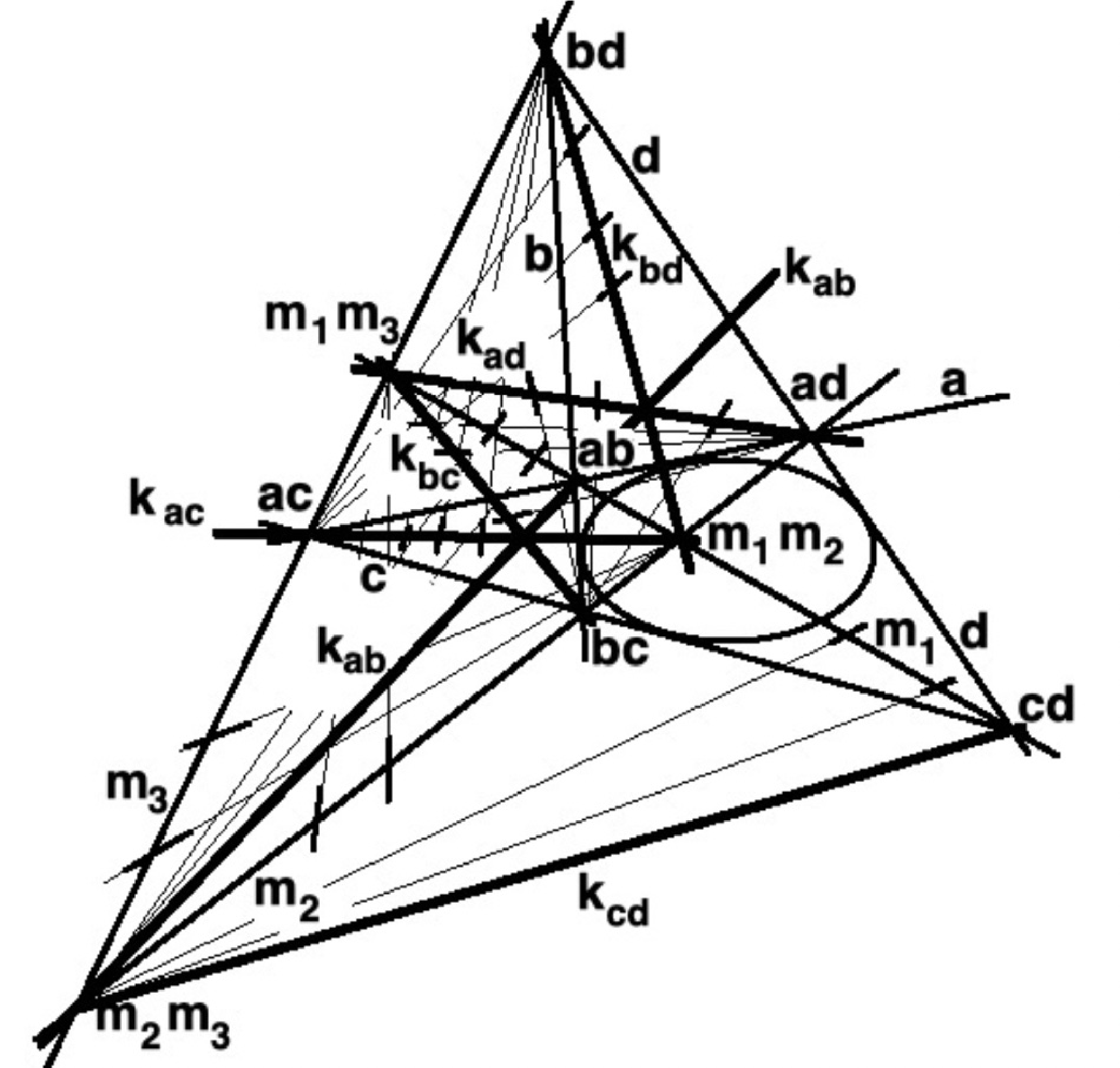, width=20em}
   \caption{Dual pencil of type a). The standard admissible lines are $m_1$, $m_2$, $m_3$. 
   The skew admissible lines $k_{ef}$ 
   are marked in bold.}\label{fig-a-d}
   \end{center}
\end{figure}
\begin{figure}
 \begin{center}
   \epsfig{file=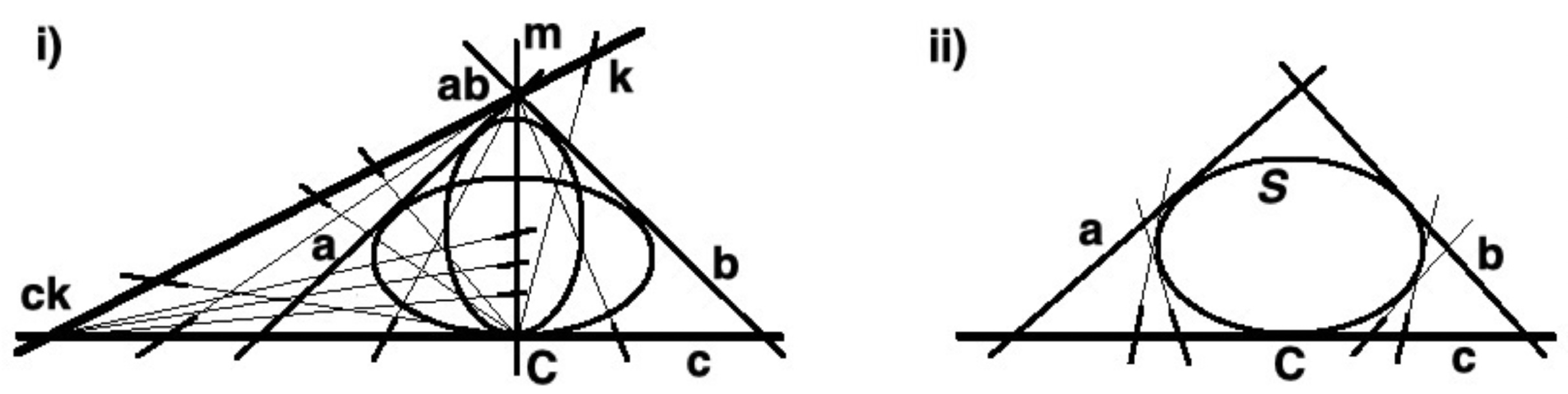, width=30em}
   \caption{Dual pencil of type b): i) one standard line $m$ and two skew lines $c$, $k$; ii) skew line $c$ with another line field, tangent to a given conic $\mcs$.}\label{fig-b-d}
   \end{center}
\end{figure}
\begin{figure}
 \begin{center}
   \epsfig{file=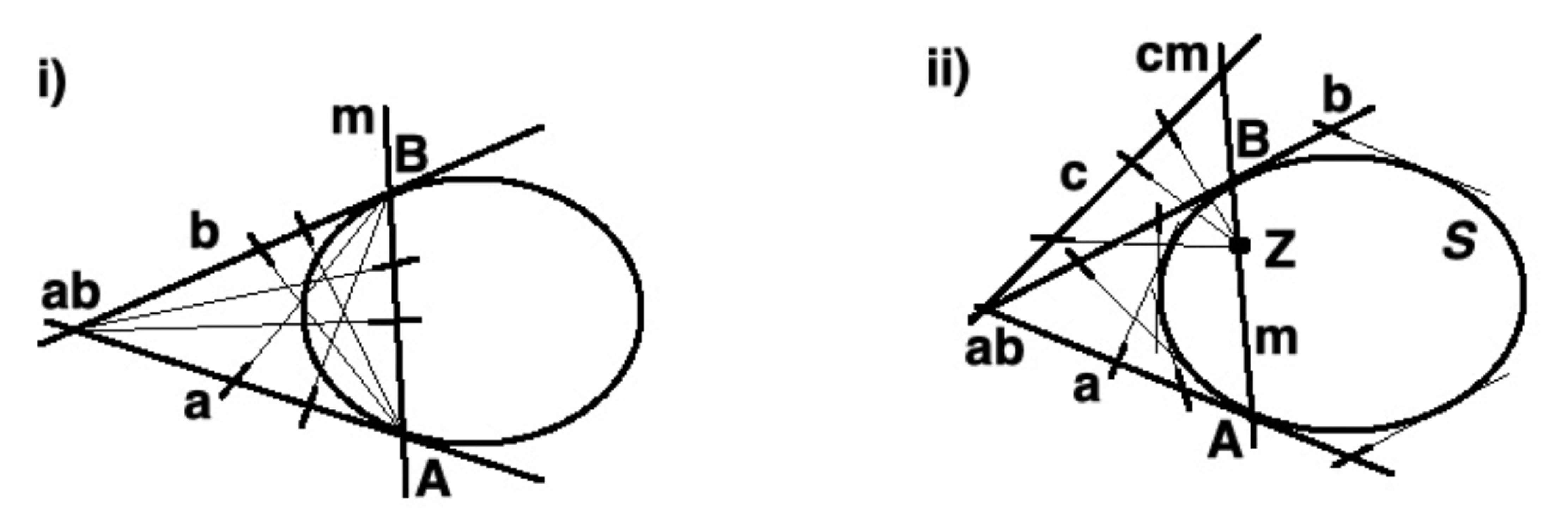, width=30em}
   \caption{Dual pencil of type c): i) one standard line $m$ and two skew lines $a$, $b$ equipped with central projective 
   billiard structures; ii) arbitrary  line $c\neq a,b$ through $ab$ (called skew)  with field of lines through $Z$,  
   and the skew lines $a$, $b$  with fields of lines tangent to a given conic $\mcs$ from the 
   pencil.}\label{fig-c-d}
   \end{center}
\end{figure}
\begin{figure}
 \begin{center}
   \epsfig{file=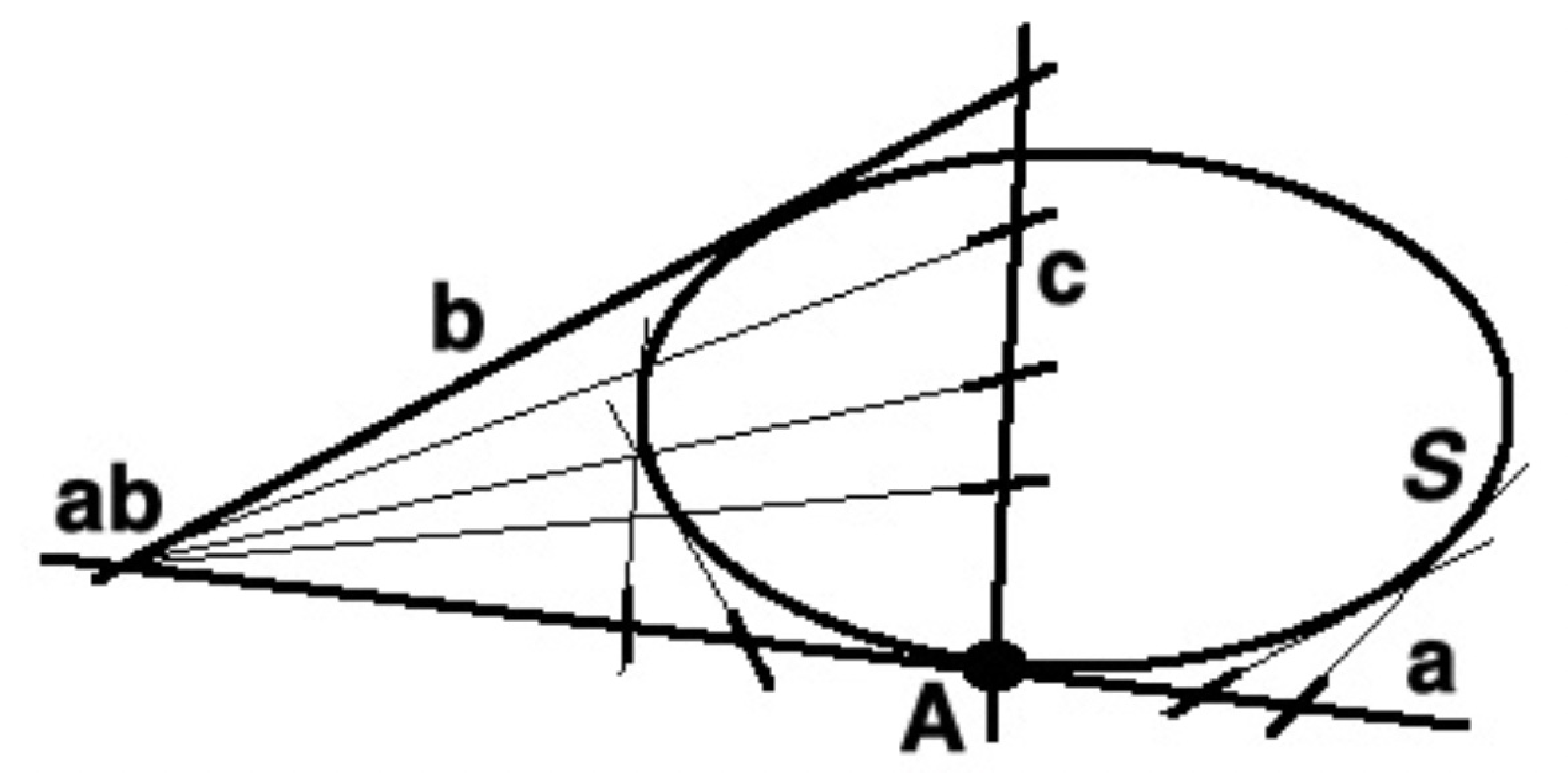, width=15em}
   \caption{Dual pencil of type d):  the skew  line $a$ and an arbitrary skew line $c\neq a$.}\label{fig-d-d}
    \end{center}
\end{figure}
\begin{figure}
 \begin{center}
   \epsfig{file=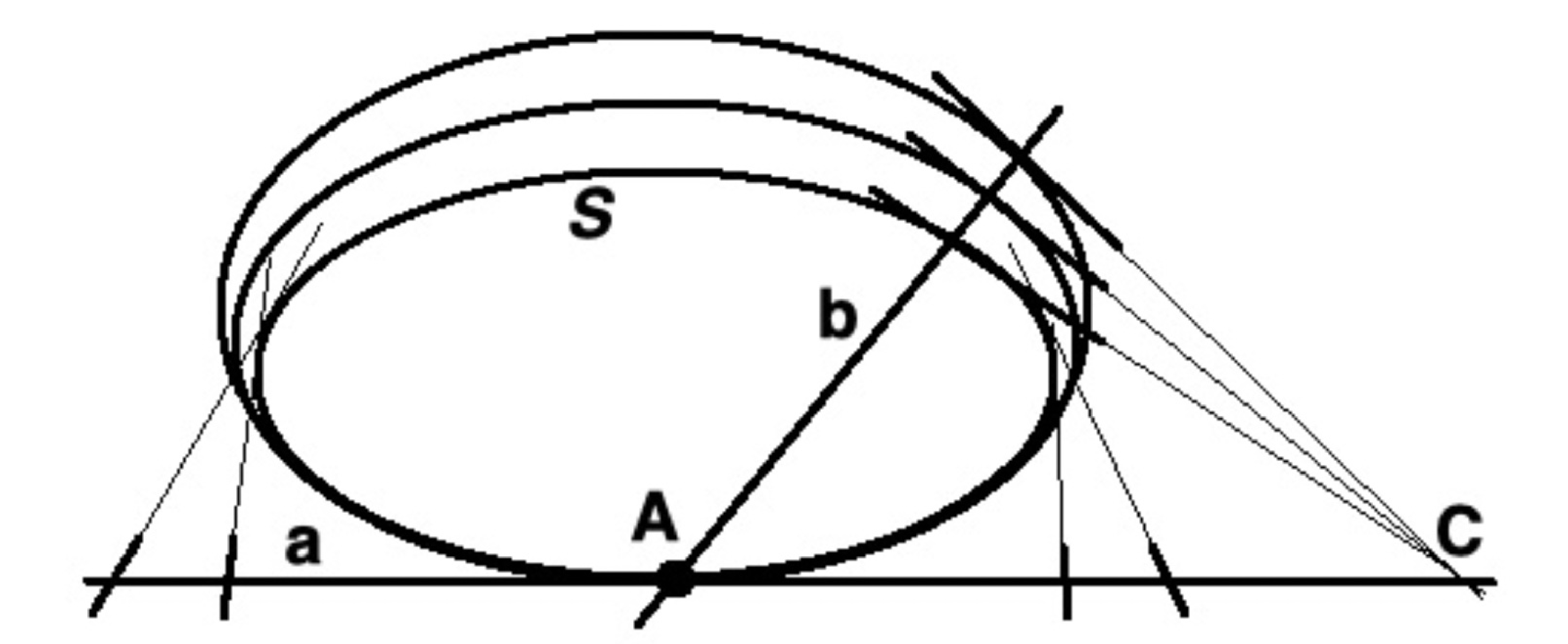, width=17em}
   \caption{Dual pencil of type e):  the skew  line $a$ and a standard line $b\neq a$.} \label{fig-e-d}
   \end{center}
\end{figure}
  \begin{proposition} In Case a) for every distinct $e,f,g\in\{ a,b,c,d\}$ the lines $k_{ef}$, $k_{fg}$, $k_{ge}$ 
  pass through the same point. In particular, the line $k_{ab}$ passes through the intersection points 
  $k_{bd}\cap k_{ad}$ and $k_{ac}\cap k_{bc}$, see Fig. \ref{fig-a-d}. 
    \end{proposition}
 \begin{proof} 
  The latter intersection points and the point $m_2m_3$ lie on one line, by the dual  
   Desargues Theorem applied to the triangles $(bd, ad, k_{bd}\cap k_{ad})$ and 
   $(ac, bc, k_{ac}\cap k_{bc})$: the intersection points $m_1m_2$, $cd$, $m_1m_3$ of pairs 
   of lines containing the corresponding pairs of their sides lie on the same line $m_1$; 
   hence, the lines  through the corresponding pairs of vertices $(ac, bd)$ (the line $m_3$), 
   $(bc, ad)$ (the line $m_2$), $(k_{bd}\cap k_{ad}, k_{ac}\cap k_{bc})$ 
   pass through the same point, namely, $m_2m_3$. 
   The point $ab$ lies on the same line, by Pappus Theorem 
   applied to the triples of points $bd$, $m_1m_3$, $ac$ and $ad$, $m_1m_2$, $bc$. 
   \end{proof}
    
 \begin{definition}
  A projective billiard in $\rr^2$ with {\it piecewise-$C^4$-smooth boundary} having at least one nonlinear smooth arc  
   is said to be of {\it 
 dual pencil type,} if it satisfies the following conditions: 
 
1) Each $C^4$-smooth arc of the boundary is either a conical arc, or a segment. 
All the conical arcs lie in the same dual 
pencil and are equipped with the projective billiard structure defined by the same pencil.

2) The  segments in the boundary are contained in lines admissible for the pencil and 
are equipped  with the projective billiard structures of the ambient admissible lines. 
 
 3) If the  boundary contains a skew line segment whose projective billiard structure is not 
 a central one (hence, it is given by a field of lines tangent to a conic of the pencil), then the 
 boundary contains no segments of other skew lines with the following exceptions of possible 
  ambient skew line collections:
 
 - in Case c) the skew line collection is allowed to be the pair of   lines $a$ and $b$ 
 equipped with the fields of lines tangent to one and the same (but arbitrary) conic of the pencil; 
 
 - in Case d) the skew line collection is allowed to be a pair of lines $a$ and $c$: $a$ being equipped with 
 the field of lines tangent to a given conic $\mcs$ from the pencil; $c$ is the line through the point $A$ and 
 the tangency point of the conic $\mcs$ with the line $b$, equipped with the field of lines throgh the point $ab$.
 
 4) In Case d) the boundary may contain a segment of at most one line $c\neq a$.
 
 5) Each ambient skew line of a boundary segment that a priori admits several possible projective structures listed 
 above is allowed to be included in the boundary with no more than one projective billiard structure.  
 \end{definition}

 \begin{remark} \label{remdu} The admissible standard (skew) vertices of a pencil are projective dual  to the admissible standard (skew) lines of its dual pencil. 
 \end{remark}
 
 \begin{theorem} \label{thmd1pr} Let  a planar projective billiard with piecewise $C^4$-smooth boundary containing 
 a nonlinear arc be 
 rationally integrable. Then the following statements hold.
 
 1) All the nonlinear arcs of the boundary are conical. Different arcs of the same conic are equipped with 
 the restriction to them of one and the same projective billiard structure on the ambient conic: either of 
 dual pencil type, or a one from Theorem \ref{tgermpr}, Case 2).
 
 2) If the boundary contains at least two arcs of two distinct  regular conics, then all the 
 ambient conics of nonlinear arcs lie in the same dual pencil and their projective billiard 
 structures are defined by the same dual pencil. 
 \end{theorem}
 \begin{theorem} \label{thmd12pr} 
 Let a planar projective billiard have piecewise $C^4$-smooth boundary whose all nonlinear 
 $C^4$-smooth pieces are conical arcs lying in the same dual pencil and equipped with projective billiard 
 structures defined by the same pencil. Then the billiard is $0$-homogeneously rationally integrable, if and only if 
 it is of dual pencil type. 
 \end{theorem}
 
 \begin{theorem} \label{thdegpr}
  The minimal degree of $0$-homogeneous rational integral of a dual pencil type projective billiard 
 is 
 
 (i) degree two, if its boundary  contains no skew line segment;
 
 {\bf (ii)) degree 12,} if the dual pencil has type a) and the billiard boundary contains segments of 
 some two so-called {\bf neighbor} skew admissible lines: namely, lines $k_{ef}$, $k_{fs}$ for  some three distinct $e,f,s\in\{ a,b,s,d\}$; 
 
 (iii) degree four in any other case.
 \end{theorem}
 
 \begin{proposition} \label{prorelines} A type a) dual pencil $\mcc^*$, 
 i.e., a family of conics tangent to 
 given four distinct complex lines, has a pair of {\bf real} neighbor skew admissible lines, if 
 and only if  all their four  common tangent lines  are real. Thus,  in this and only in this case 
$\mcc^*$ admits  a dual pencil type projective billiard with minimal degree of 
 rational integral being equal to 12, see (ii). 
 \end{proposition} 
 
 \begin{theorem} \label{tintpr} Consider a type a)  dual pencil of conics tangent to given four distinct real lines 
 $a$, $b$, $c$, $d$. Let us consider the ambient plane $\rr^2_{x_1,x_2}$ as the horizontal plane 
 $\{ x_3=1\}\subset\rr^3_{x_1,x_2,x_3}$. Set 
$$ r:=(x_1,x_2,1)\in\rr^3, \ v=(v_1,v_2,0) \text{ for every } (v_1,v_2)\in T_{(x_1,x_2)}\rr^2,$$
 \begin{equation} \mcm=\mcm(r,v):=[r,v]=(-v_2,v_1,\Delta), \ \Delta:=x_1v_2-x_2v_1.\label{moments}\end{equation}
In the above notations for intersection points $em$ of lines $e$ and $m$ set 
$$r(em)=(x_1(em), x_2(em), 1).$$
There exists a collection of three numbers $\chi_{em;fn}\in\rr$, $\{ e, m, f, n\}=\{ a, b, c, d\}$, 
numerated by {\it unordered pairs} of intersection points $em=e\cap m$, $fn=f\cap n$ ($(em; fn)=(fn;em)$, 
by definition) such that 
\begin{equation}\sum_{(em;fn)}\chi_{em;fn}<r(em),\mcm><r(fn),\mcm>=0\label{sum0pr}\end{equation}
as a quadratic form in $\mcm$. This collection is unique up to common constant factor. For every  
$\mu\in\rr\setminus\{0\}$ the corresponding expression
 \begin{equation}\prod_{(em; fn)\neq(e'm';f'n')}
\left(\frac{\chi_{em; fn}<r(em),\mcm><r(fn),\mcm>}{\chi_{e'm';f'n'}<r(e'm'),\mcm><r(f'n'),\mcm>}+\mu\right)\label{intpr12}
\end{equation}
is a degree 12 first integral of every projective billiard of dual pencil type defined by the pencil in question. 
Here the product is taken over ordered "big" pairs 
$((em; fn); (e'm'; f'n'))$.
\end{theorem}
\begin{lemma} \label{lintpr} Consider the following  segment lengths, see Fig. \ref{fig-a-int-form}:
 \begin{equation} \rho:=|bc - ab|, \ t:= |ab - bd|, \ \tau:=|ad-ab|, \ s:=|ab-ac|.\label{distances}\end{equation}
Here the lengths are oriented:  lengths of two adjacent aligned segments 
(say, $s$ and $\tau$), are taken  with the same sign, 
if their common end separates them ($ab$ lies between the points $ac$ and $ad$), and with  opposite signs 
 otherwise. Relation (\ref{sum0pr}) (and hence, 
 the statements of Theorem \ref{tintpr}) hold for 
 \begin{equation}(\chi_{ab;cd},\ \chi_{bc; ad},\  \chi_{ac;bd})=\left(\frac{st}{\rho\tau}-1,\ -\frac{st}{\rho\tau},1\right).\label{chi-pr}\end{equation}
 \end{lemma}
 \begin{figure}
 \begin{center}
   \epsfig{file=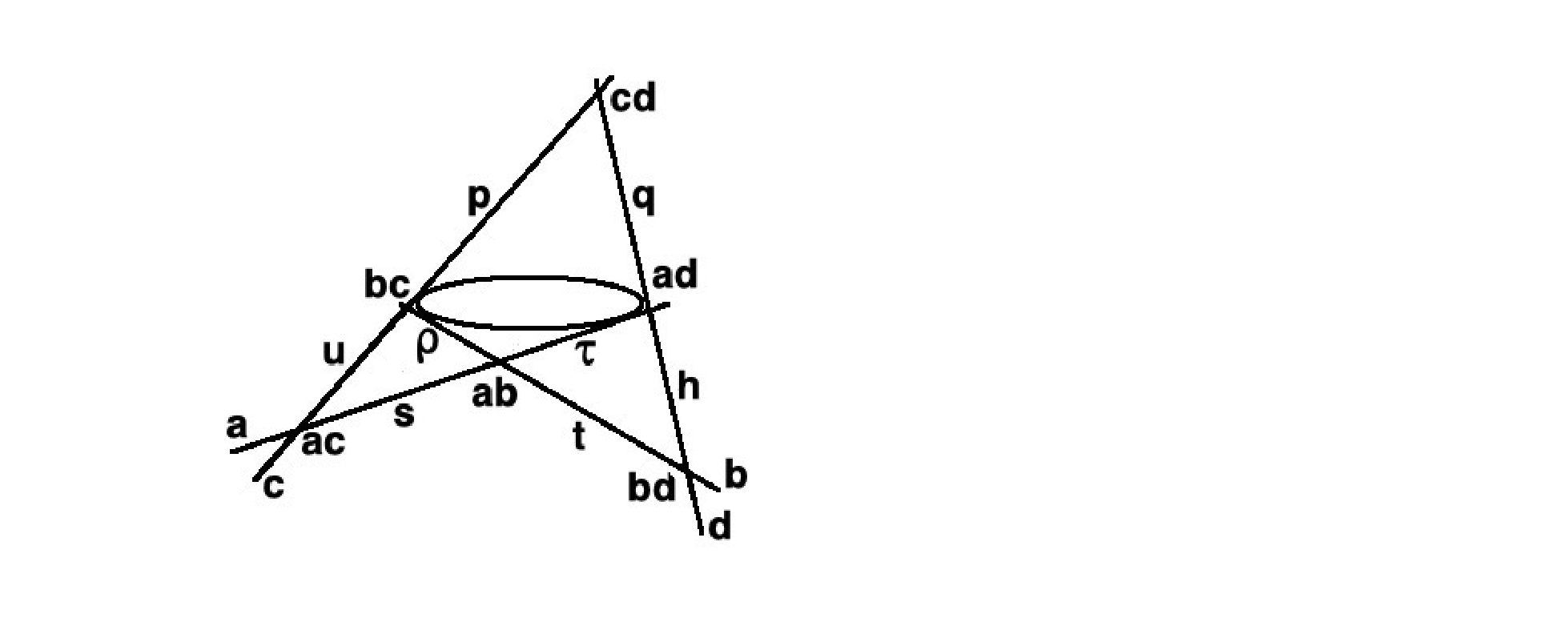, width=30em}
   \caption{Pencil of type a). Distances between intersection points.} \label{fig-a-int-form}
   \end{center}
\end{figure}

\begin{definition} \label{defexpr} A rationally $0$-homogeneously integrable  
projective billiard structure on a conic $\gamma$ that is not 
 of dual pencil type (i.e., any projective billiard structure from Theorem \ref{tgermpr}, Case 2)) 
 will be called {\it exotic.} Its singular points  (which are the points, 
 where the corresponding line field is either undefined, or tangent to $\gamma$) will be called the {\it base points.}
 \end{definition}
 
 \begin{corollary} \label{corexpr} Let a rationally integrable projective billiard with piecewise $C^4$-smooth boundary 
be not of dual pencil type. Let its boundary contain at least one nonlinear arc. 
Then all its nonlinear arcs  lie in one conic, equipped with an exotic 
 projective billiard structure from Theorem \ref{tgermpr}, Case 2). 
 \end{corollary}
 
 \begin{theorem} \label{thmd2pr} Let a projective billiard has piecewise smooth boundary consisting of arcs   
 of one and the same  conic $\gamma$ equipped with an exotic projective billiard structure and maybe some straightline segments. The billiard 
  is rationally integrable, if and only if the collection of ambient lines of the boundary segments  
  either is empty, or consists of some of the following so-called {\bf admissible lines} equipped with 
   central projective billiard structures:
 
(i) Case of type 2a) projective billiard structure on $\gamma$; $\rho=2-\frac2m$, $m\in\nn$, $m\geq3$. 
The    unique  admissible line is 
the vertical $x_2$-axis, equipped with the normal (i.e., horizontal) line field. See Fig. \ref{figapr}. 
The projective billiard  bounded by  a half of $\gamma$ and  the $x_2$-axis has a 
rational $0$-homogeneous integral of minimal degree $2m$: the  
 function $\Psi_{2a1}$ from (\ref{r2a1v}) for odd $m$;  
  the function $\Psi_{2a2}^2$ with $\Psi_{2a2}$ the same, as in (\ref{r2a2v}), for even $m$. See Fig. \ref{figapr}. 
 
 (ii)  Case of  type 2b1). The unique admissible line is the line $\{ x_2=1\}$ equipped with the field of lines 
 through the point $(0,-1)$.
 
 (iii) Case of type 2b2).  The unique admissible line is the $x_2$-axis equipped with the normal (horizontal) 
 line field. See Fig. \ref{figbpr}. In both cases 2b1), 2b2) the  functions $\Psi_{2b1}$, $\Psi_{2b2}$ 
 from (\ref{r2b1v}) and (\ref{r2b2v})  
 are  integrals of minimal degree $4$ for each billiard bounded by $\gamma$ and the admissible line.

 (iv) Case of  type 2c1). The unique admissible line is the line $\{ x_2=1\}$ equipped with the field of lines 
 through the point $(0,-1)$.
 
 (v) Case of  type 2c2). There are three admissible lines:
 
 - the line $\{ x_2=1\}$, with the field of lines through the point $(0,-1)$;
 
 - the line $\{ x_1=-\frac12\}$, with the line field  parallel to the vector $(-1,1)$;
 
 - the line $\{ x_2=-2x_1\}$, with the field of lines through the point $(-1,0)$. 
 
 See Fig. \ref{figcpr}. In both  cases 2c1), 2c2) the corresponding functions $\Psi_{2c1}$, $\Psi_{2c2}$ 
 from (\ref{r2c1v}) and (\ref{r2c2v})  
 are  integrals of minimal degree $6$ for each billiard bounded by $\gamma$ and segments of 
 admissible lines. 
  
  (vi) Case of type 2d). No admissible lines.
 \end{theorem}
 \begin{figure}
 \begin{center}
   \epsfig{file=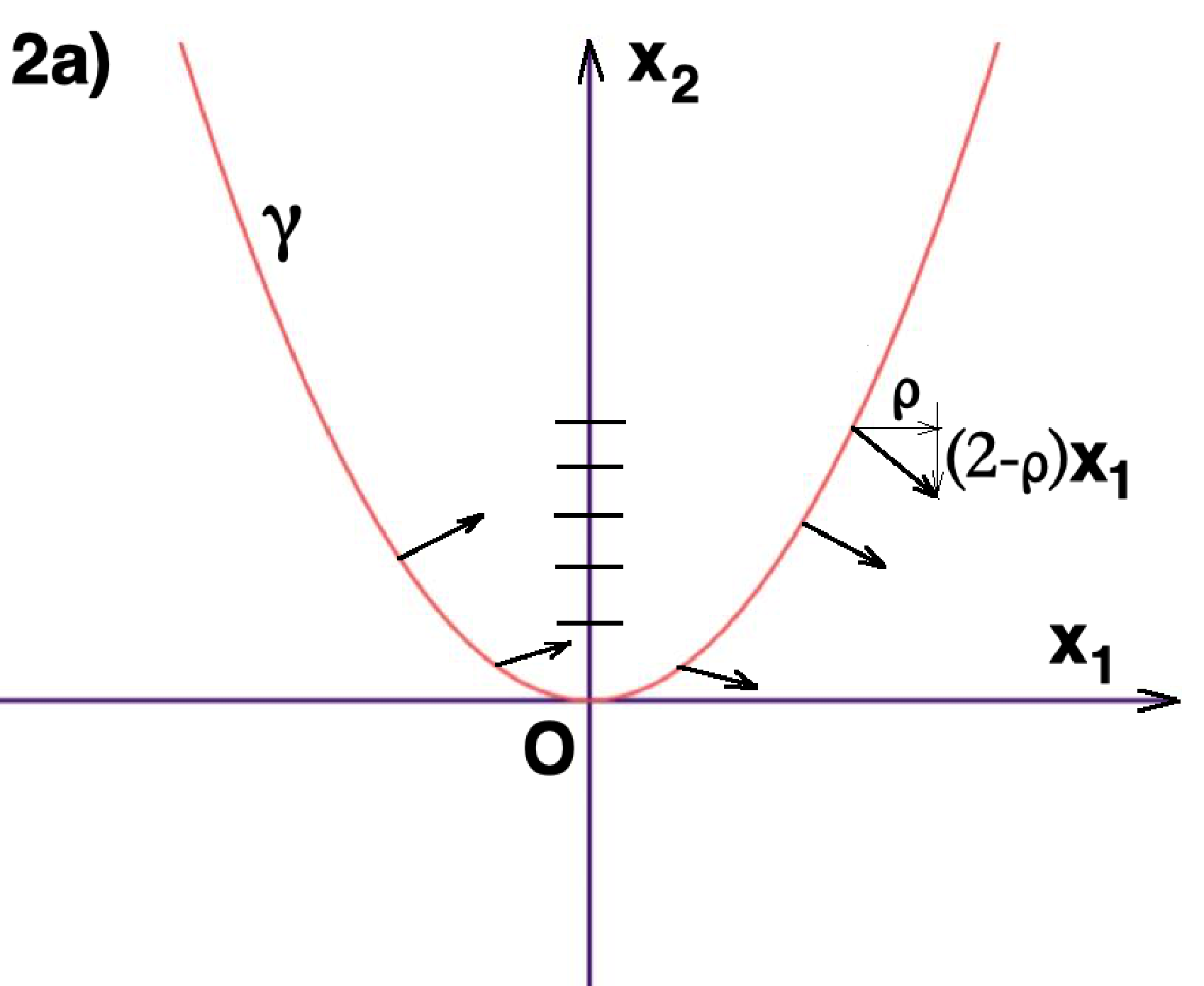, width=14em}
   \caption{The only admissible line in Case 2a) is the $x_2$-axis.}\label{figapr}
   \end{center}
\end{figure}
 \begin{figure}
 \begin{center}
   \epsfig{file=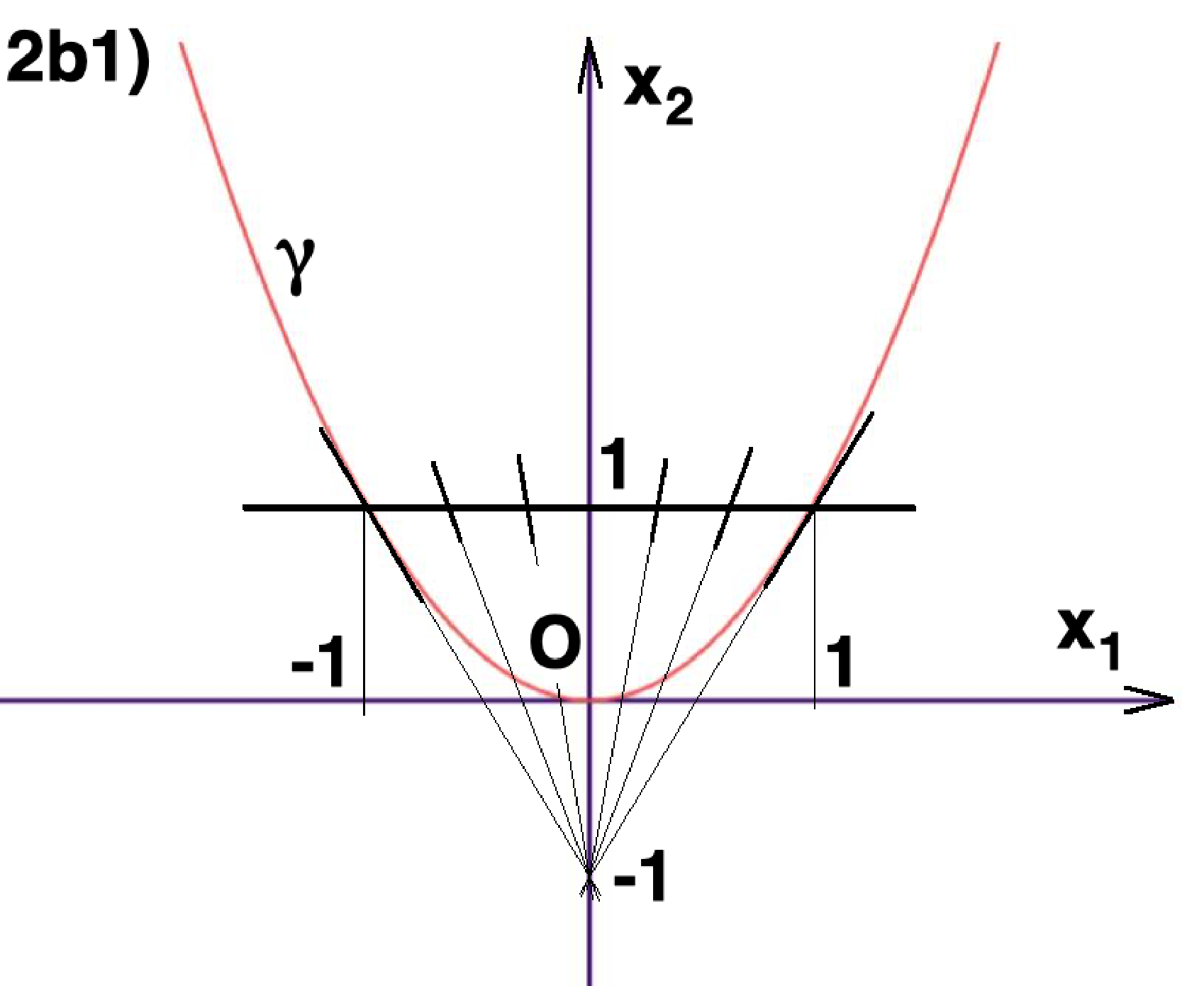, width=14em}
   \hspace{1cm}
   \epsfig{file=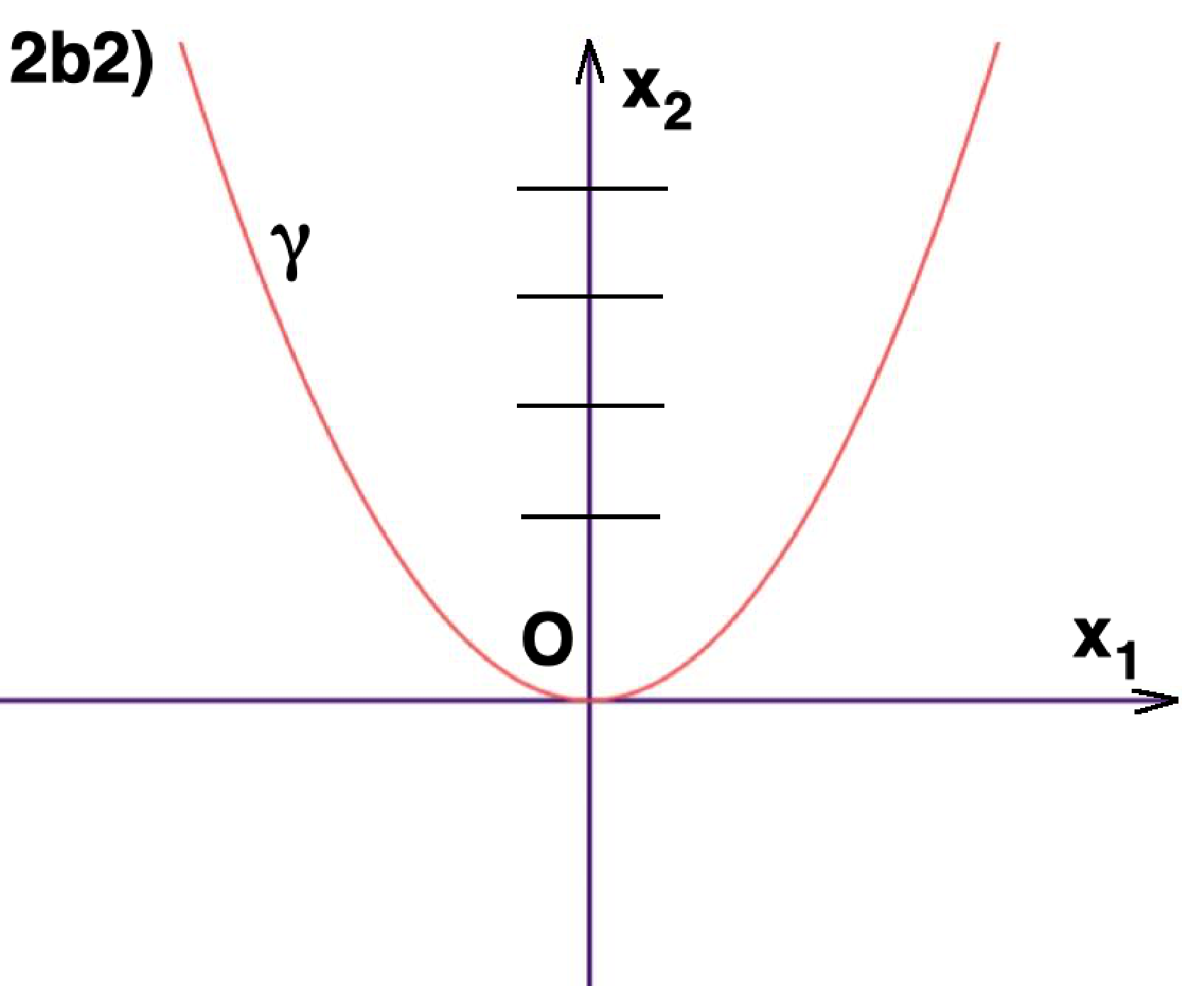, width=14em}
   \caption{The only admissible line in Case 2b1) is the line $\{ x_2=1\}$. 
   The only admissible line in Case 2b2) is the $x_2$-axis.}\label{figbpr}
   \end{center}
\end{figure}
 \begin{figure}
 \begin{center}
   \epsfig{file=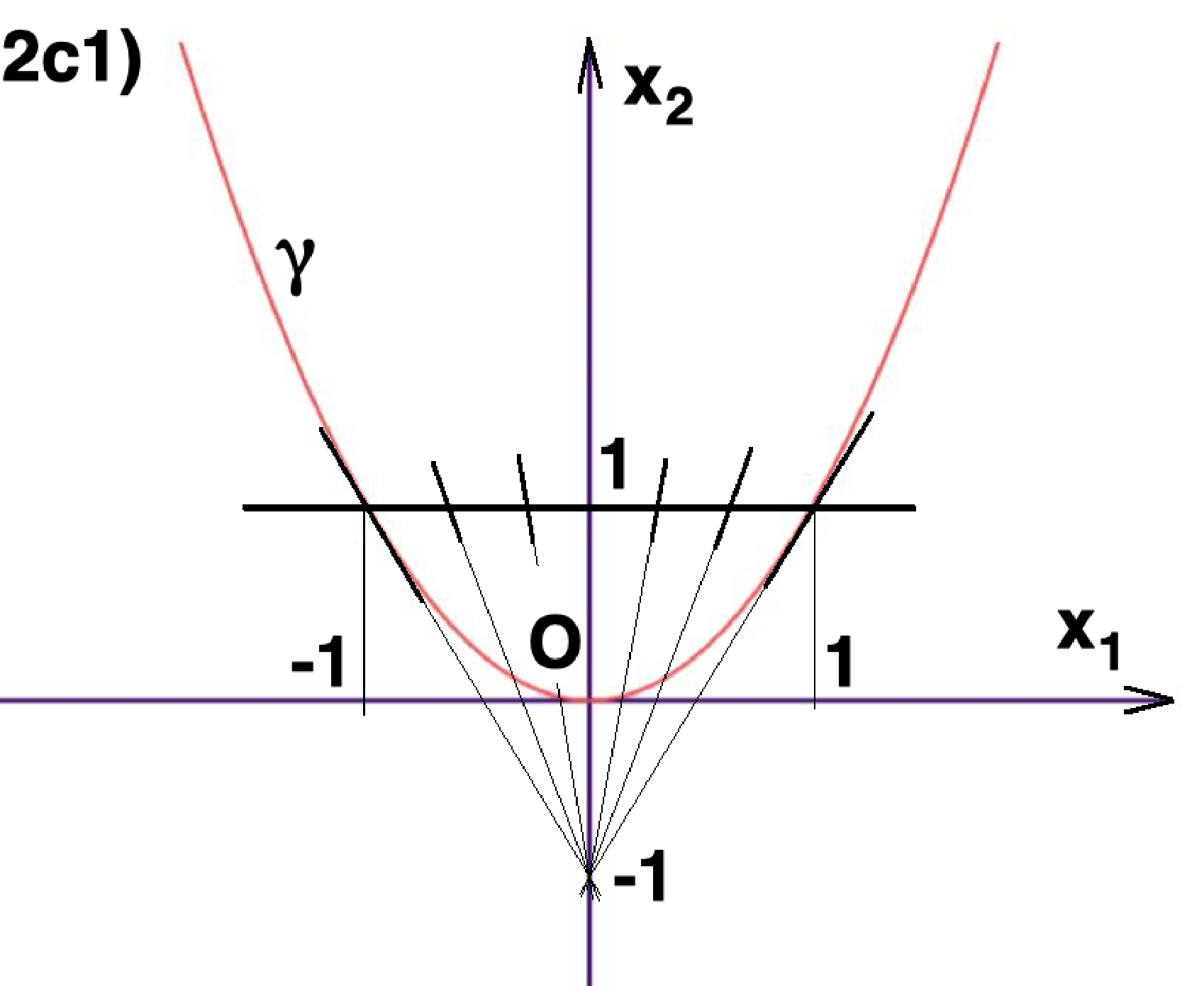, width=14em}
   \hspace{1cm}
   \epsfig{file=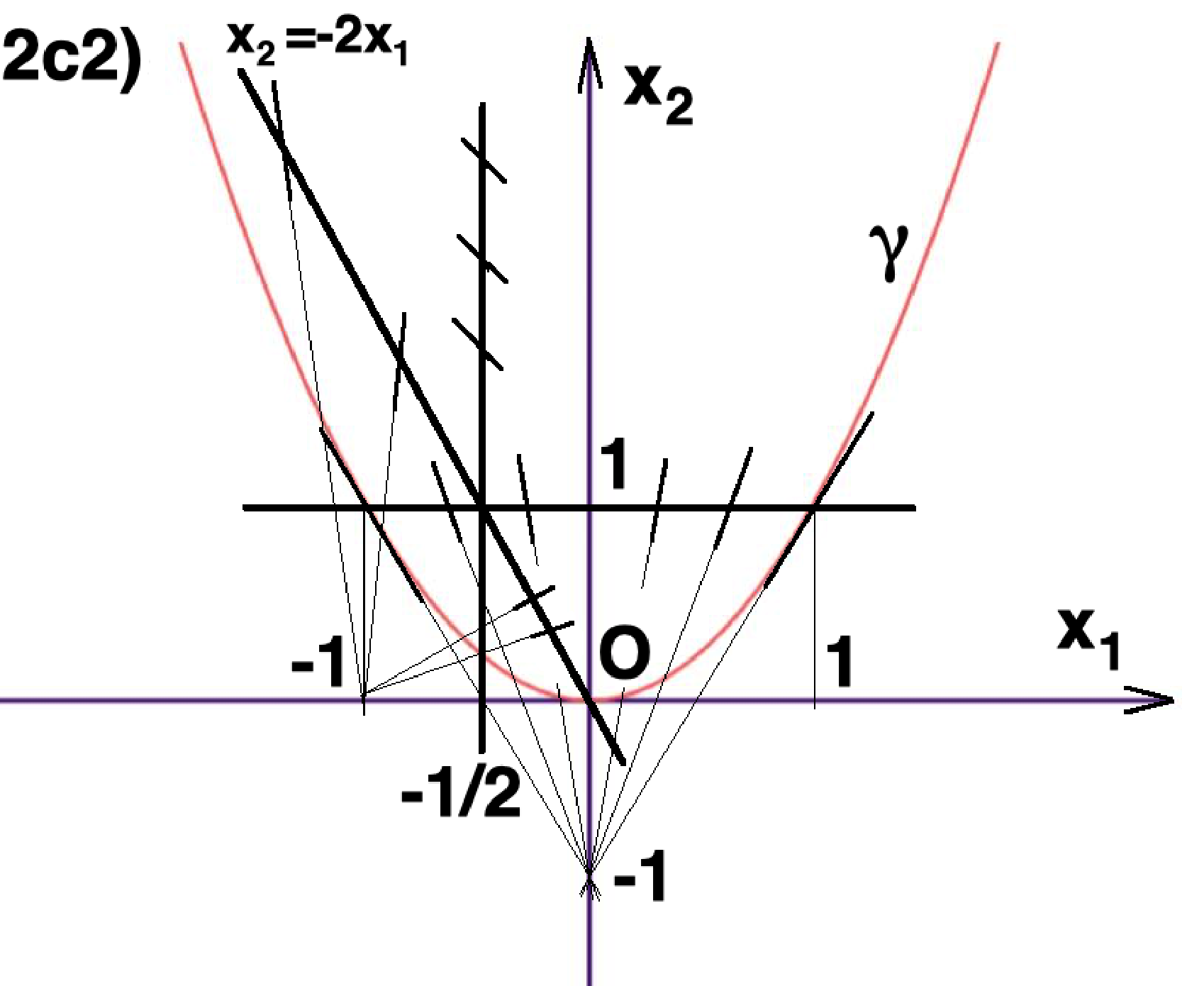, width=14em}
   \caption{The only admissible line in Case 2c1) is the line $\{ x_2=1\}$. 
   In Case 2c2) there are three admissible lines:  $\{ x_2=1\}$; $\{ x_1=-\frac12\}$, $\{ x_2=-2x_1\}$.}\label{figcpr}
   \end{center}
\end{figure}

 \subsection{Plan of proofs of main results}
 
 Step 1.  In Subsection 2.1 we prove  rational integrability 
 of pencil type complex multibilliard. (This implies analogous result in the real case.) 
 To do this, we show that for every pencil all the involutions associated to all the corresponding  
 admissible vertices preserve the pencil and act on its parameter space $\oc$ by 
 conformal involutions (Proposition \ref{glopro}). We fix an arbitrary collection of admissible vertices and 
 consider  the subgroup $G\subset Aut(\oc)=PSL_2(\cc)$ 
generated by the corresponding conformal involutions. We show that finiteness of the group  
 $G$ is equivalent to the system of Conditions 3)--5) 
 of Definition \ref{multipb} of pencil type multibilliard (Proposition \ref{progroup}), and in this case 
 $G$ is either trivial, or isomorphic to either $\zz_2$, or $S_3$. 
 We then deduce rational integrability of every pencil type multibilliard with integral of degree $2|G|\in\{ 2, 4, 12\}$.

To classify rationally integrable dual multibilliards, in what follows we consider 
an arbitrary  dual multibilliard with a rational integral $\Psi$. Each its curve is already known to be a conic 
equipped with either a pencil type dual billiard structure, or an exotic billiard structure from Theorem \ref{tgerm}, 
Case 2). We fix some its conic $\mcs$ and consider the canonical integral $R$ of its dual billiard structure: 
either a quadratic integral in the case of pencil, or the corresponding integral from the  Addendum to Theorem \ref{tgerm}.

Step 2. In Subsection 2.2 we 
show that the singular foliations $\Psi=const$ and $R=const$ on $\cp^2$ coincide. 
 We show that  a generic level curve of the integral $R$ is irreducible, 
of the same degree $d=\deg R$, 
and thus, $R$ is a rational first integral of minimal degree for the above foliation. 
In the exotic case we also show that the conic $\mcs$ is its unique level curve of multiplicity $d$, which means that the irreducible 
level curves of the function $R$ accumulating to $\mcs$ converge to $\frac d2[\mcs]$ as divisors: 
the intersection of a small cross-section to $\mcs$ with a level curve close to $\mcs$ consists of $\frac d2$ points. 

Step 3. We then deduce (in Subsection 2.2) that if on some conic of the multibilliard the dual billiard structure is defined by 
a pencil (or if the multibilliard contains at least two distinct conics), 
then the above foliation coincides with the pencil and the dual billiard structures on all the other conics 
are defined by the same pencil. This will prove Theorem \ref{thmd1}. Results of Step 1 together with constance 
of integral on the conics of the pencil (given by Step 2)  imply Theorem \ref{thdeg}.

Step 4. In Subsection 2.3 we study vertices of the multibilliard. First we show that the family of 
involutions $\sigma_{A,\ell}:\ell\to\ell$ associated to each vertex $A$ is given by the restrictions to the lines 
$\ell$ through $A$ of a  birational involution $\sigma_A:\cp^2\to\cp^2$ preserving the foliation $\Psi=const$. 
Then we deduce that each $\sigma_A$ is 
either a projective angular symmetry, or a degenerate $\mcs$-angular symmetry defined by a regular conic $\mcs$ 
through $A$. We show that  in the latter case the foliation $\Psi=const$ is a pencil of conics 
containing $\mcs$.

Step 5. In Subsection 2.4 we prove Theorem \ref{thmd12}. It deals with the case, 
when the foliation $\Psi=const$ is a pencil of conics. We  
 show that  each vertex of the multibilliard is admissible.  Each level curve of the function $\Psi$ is a collection of at most $\frac{\deg\Psi}2$ conics of the pencil, and it is invariant under the involutions defining the dual billiard structures at
  the vertices. This implies finiteness of the group $G$ generated by the conformal involutions corresponding to the vertices. Together with the results of Step 1 (Subsection 2.1), this implies that  the multibilliard  is 
of pencil type.

In Subsection 2.5 we prove Theorem \ref{thmd2} on classification of rationally integrable multibilliards consisting of  
a conic $\mcs$ with an exotic dual billiard structure and (may be) some vertices. We describe the corresponding 
admissible vertices using the result of Step 4 stating that the corresponding involutions are projective angular symmetries.   

 In Subsection 2.6 we  prove Propositions \ref{preal} and \ref{prorelines}. 
 
 In Section 3 we prove the main results on classification of rationally $0$-homogeneously integrable piecewise $C^4$-smooth projective billiards 
 (Theorems \ref{thmd1pr}, \ref{thmd12pr}, \ref{thdegpr}, \ref{thmd2pr}). 
 We reduce them to the main results on dual multibilliards via  the  projective duality given by orthogonal polarity. We prove bijectivity of the correspondence (given in Remark \ref{removect})  between rational integrals of a 
multibilliard and rational $0$-homogeneous integrals of the flow of its dual projective billiard. 
  This together with the results from 
  \cite{grat} on duality between exotic dual billiards from Theorem \ref{tgerm} and exotic projective billiards from 
  Theorem \ref{tgermpr} and the results of the present paper on dual multibilliards will imply the main results on 
  projective billiards. 
   
 \subsection{Historical remarks}
 
 Existence of a continuum of closed caustics in every strictly convex bounded planar billiard 
 with sufficiently smooth boundary was proved by V.F.Lazutkin \cite{laz}. Existence of continuum of foliations by 
 (non-closed) caustics in open billiards was proved by the author \cite{glcaust}. H.Poritsky \cite{poritsky} (and later E.Amiran \cite{amiran}) proved the Birkhoff Conjecture under the additional assumption that for every two caustics the smaller one 
 is a caustic for the bigger one.  M.Bialy \cite{bialy} proved that if the phase cylinder is foliated by non-contractible 
 invariant curves for the billiard map, then the billiard table is a disk. See also  \cite{wojt}, where 
 another proof of Bialy's result was given, and Bialy's papers \cite{bialy2, bialy1} for similar results on billiards on constant curvature surfaces and on magnetic billiards on these surfaces. D.V.Treschev conjectured existence of billiards where 
 the squared billiard map has fixed point where its germ is analytically conjugated to rotation and confirmed this by 
 numerical experiments: in two dimensions \cite{tres1, tres2} and in higher dimensions \cite{tres3}. V.Kaloshin and 
 A.Sorrentino \cite{kalsor} proved that {\it any integrable deformation of an ellipse is an ellipse} 
 (see  \cite{kavila} for the case of ellipses with small excentricities). 
 A stronger version of their result for almost every ellipse was very recently proved by Illya Koval 
 \cite{koval}. 
 M.Bialy and A.B.Mironov proved the Birkhoff Conjecture for centrally-symmetric billiards admitting a continuous 
 family of caustics extending up to a caustic of 4-periodic orbits \cite{bm6}. For a dynamical entropic version of the Birkhoff Conjecture and related results see \cite{marco}. For a survey on the Birkhoff Conjecture and  results on it see \cite{kalsor, KS18, bm6} and references therein. 

 A.P.Veselov proved a series of complete integrability results for billiards bounded by confocal quadrics 
 in space forms of any dimension and described billiard orbits there in terms of 
 a shift of the Jacobi variety corresponding to an appropriate hyperelliptic curve \cite{veselov, veselov2}. 
Dynamics in (not necessarily convex) billiards of this type was also studied in  \cite{drag, dr2, dr3, dr4, dr5}. 

 The Polynomial Birkhoff Conjecture together with its generalization to piecewise smooth billiards 
 on surfaces of constant curvature was  
 stated by S.V.Bolotin and partially studied by himself, see  \cite{bolotin}, \cite[section 4]{bolotin2}, 
 and  by M.Bialy and A.E.Mironov  \cite{bm3}. Its complete  solution 
 is a joint result of M.Bialy, A.E.Mironov and the author given in  the series of papers 
  \cite{bm, bm2, gl, gl2}. It implies that if a polynomial integral  of a piecewise smooth billiard exists, 
  then its minimal degree is equal to either two, or four.
   For a survey of Bolotin's Polynomial Birkhoff Conjecture and of its version
   for magnetic billiards (an open conjecture, with a substantial progress made in \cite{bm4, bm2.5}) and related results 
 see \cite{bm4, bm, bm2,  bm2.5, bm5, KS18, kozlov} and references therein. 

The generalization of the Birkhoff Conjecture to  dual billiards was stated by S.Tabachnikov 
in \cite{tab08}. Its rationally integrable version was solved by the author of the present paper in \cite{grat}. 
Its polynomially integrable version for outer billiards was stated and partially studied in \cite{tab08} and solved 
completely in \cite{gs}. Projective billiards were introduced by S.Tabachnikov \cite{tabpr}. He had shown there  
 that  if a projective billiard on circle has an invariant area form smooth up to the boundary of the 
phase cylinder, then it is integrable. 

A series of results on the analogue of Ivrii Conjecture on periodic orbits in billiard (stating that their Lebesgue measure 
is zero) for projective billiards was obtained by C.Fierobe \cite{fierobe-th, fierobe-proj, fierobe-proj-refl}.
 
 \section{Rationally integrable dual multibilliards. Proofs of Theorems \ref{thmd1}, \ref{thmd12}, \ref{thmd2}, \ref{thdeg}}
 
 \subsection{Rational integrability of pencil type multibilliards}

\begin{proposition} \label{glopro} Consider a complex pencil of conics  and the corresponding admissible 
vertices. For every  standard vertex the corresponding involution leaves invariant 
  each conic of the pencil. For every skew vertex the corresponding involution 
  permutes conics of the pencil non-trivially: it acts as a conformal involution of the  
  parameter space $\oc$ of the complex pencil. 
\end{proposition}
\begin{proof}  Case a): pencil of conics through 
 four distinct basic points $A$, $B$, $C$, $D$, see Fig. \ref{fig2}. It is well-known that in this case no three of them lie on the same line. This implies that the three vertices $M_j$ are well-defined, 
  distinct, do not lie on the same line
   and different from the basic points, and so are the vertices $K_{EL}$, and 
   the latter are distinct from the vertices $M_j$. Set  
   \begin{equation}\Gamma_1:=AB\cup CD, \ \Gamma_2:= BC\cup AD, \ \Gamma_3=AC\cup BD.\label{gammy}
   \end{equation}
   Let $\sigma_{M_1}:\cp^2\to\cp^2$ be 
   the  $\Gamma_2$-angular symmetry centered at $M_1$: the projective involution fixing 
   each line through $M_1$ and permuting its intersection points with the lines $AD$ and $BC$. 
   It permutes the points $A$ and $B$, $C$ and $D$. Hence, it preserves the pencil. 
   It fixes the line $M_2M_3$, which passes through the points $AD\cap BC$ and $AC\cap BD$. Hence, 
   it fixes each its point $X$, since it fixes the line $M_1X$. The pencil is parametrized by a parameter 
   $\lambda\in\oc$, and $\sigma_{M_1}$ acts on $\oc_\la$ by 
   conformal automorphism.  
  Let $\la_1,\la_2,\la_3\in\oc$ denote the parameter values corresponding to the singular conics 
 $\Gamma_1$, $\Gamma_2$, $\Gamma_3$ respectively. Each $\Gamma_j$ is $\sigma_{M_1}$-invariant, by 
 construction. Therefore, the conformal automorphism $\oc\to\oc$ induced by 
   $\sigma_{M_1}$ fixes three distinct points $\la_1$, $\la_2$, $\la_3$. Hence, 
   it is identity, and $\sigma_{M_1}$ preserves each conic of the pencil. This proof 
   is valid for the other vertices $M_j$.
   
   The involution $\sigma_{K_{BC}}$ is the projective angular symmetry centered at $K_{BC}$ 
   with fixed point line $AD$. Hence, it fixes $M_2$. 
   
   {\bf Claim 1.} {\it The involution $\sigma_{K_{BC}}$ permutes $B$ and $C$. Or equivalently, 
   the quadruple of points $K_{BC}$, $M_2$, $B$, $C$ on the line $BC$ is harmonic.}
   
   \begin{proof} The restriction of the involution $\sigma_{K_{BC}}$ to the line $BC$ coincides with the involution 
   $\sigma_{M_2}$, since both of them are non-trivial projective involutions of the line $BC$ fixing $K_{BC}$ and 
   $M_2$. The involution $\sigma_{M_2}$ permutes $B$ and $C$, as in the above discussion on $\sigma_{M_1}$. 
   Hence, so does $\sigma_{K_{BC}}$.
   \end{proof}
   
   \begin{corollary} \label{permut1} Each one of the involutions $\sigma_{K_{BC}}$, 
   $\sigma_{K_{AD}}$  fixes  
   $\Gamma_2$ and permutes $\Gamma_1$, $\Gamma_3$. Hence, it yields a 
   non-trivial conformal involution $\oc_\la\to\oc_\la$ of the parameter  space 
   of the pencil, fixing $\la_2$ and permuting $\la_1$, $\la_3$.
   \end{corollary}
   \begin{proof} The involution $\sigma_{K_{BC}}$ fixes $A$, $D$ and permutes 
   $B$, $C$. Similarly, the involution $\sigma_{K_{AD}}$ fixes $B$, $C$ and permutes $A$, $D$. 
   \end{proof}

   Case b): pencil of conics through three distinct points $A$, $B$, $C$, tangent at the point $C$ to the same line $L$, see Fig. \ref{fig3}. The involution $\sigma_M$ 
   fixes the points $C$, $K_{AB}$, the line $L$ and permutes $A$ and $B$, by definition 
   and harmonicity of the quadruple 
   $M$, $K_{AB}$, $A$, $B$. Therefore, it preserves the pencil. Similarly, the involution $\sigma_{K_{AB}}$ 
   preserves the pencil. And so does the involution $\sigma_C:\cp^2\to\cp^2$ defined to fix $C$ 
   and each point of the line $AB$. Now the pencil, 
   parametrized by a parameter $\la\in\oc$, contains just two singular conics: 
   \begin{equation}\Gamma_1:=AB\cup L, \ \Gamma_2:=AC\cup BC,\label{gammy2}\end{equation} 
   corresponding to some parameter values $\la_1$ and $\la_2$. 
   
   {\bf Claim 2.} {\it The involution $\sigma_M$ and the composition $\sigma_C\circ\sigma_{K_{AB}}$ 
    preserve each conic of the pencil. Each one of the involutions  $\sigma_C$, $\sigma_{K_{AB}}$
    fixes only the conics $\Gamma_1$, $\Gamma_2$ of the pencil.}
   
   \begin{proof} 
   The pencil in question is the limit 
   of a family of pencils of conics through $A$, $B$, $C_\mu$, $D_\mu$ 
   with basic points $C_\mu$, $D_\mu$ depending on small parameter $\mu$, 
   confluenting to $C$, as $\mu\to0$, so that the line $C_\mu D_\mu$ pass through 
   $M=M_1$ and tends to the tangent line $L$, as $\mu\to0$.  
   Then $M_2=M_2(\mu)\to C$, $M_3=M_3(\mu)\to C$, 
   and the involutions $\sigma_{M_1}=\sigma_{M_1(\mu)}$ corresponding to the 
   perturbed pencil, with $\mu\neq0$, converge to $\sigma_M$, as $\mu\to0$. 
   The involution $\sigma_{M_1(\mu)}$ preserves each conic of the 
   pencil for $\mu\neq0$. Hence, so does its limit $\sigma_M$. 
   The involutions at the vertices $K_{C_\mu D_\mu}$, 
   $K_{AB}$ converge to $\sigma_C$ and $\sigma_{K_{AB}}$, by construction. 
   They act on the perturbed pencil as non-trivial involutions, permuting conics in the 
   same way (Corollary \ref{permut1}). Hence, this statement remains valid for their 
   limits $\sigma_C$ and $\sigma_{K_{AB}}$. The claim is proved.
   \end{proof}
   
  \begin{proposition} \label{invpar} Consider a pencil of complex conics that are tangent to each other at a point $C$. 
  Let $\mcs$ be its regular conic, and let $C$ be equipped with the quasi-global  dual billiard structure defined by 
  $\mcs$. Then the corresponding involution $\sigma_C$  preserves the pencil 
   and induces a non-trivial conformal involution $\oc\to\oc$ of its parameter space.
   \end{proposition}
   
   \begin{proof} Let $L$ denote the common projective tangent line at $C$ to the regular conics 
   of the pencil. Let us take an affine 
chart $\cc_{z,w}=\cp^2\setminus L$ so that $C$ is the intersection point of the $w$-axis with the infinity line $L$.  
Then the conics of the pencil are  parabolas $\mcs_\la:=\{ w=(a_1z^2+b_1z+c_1)+\la(a_2z^2+b_2z+c_2)\}$. 
Let us normalize the parameter $\la$ so that $\mcs_0=\mcs$. Then in the affine chart $(z,w)$ one has  
$$\sigma_C(z,w)=(z,2(a_1z^2+b_1z+c_1)-y).$$
Hence, $\sigma_C(\mcs_\la)=\mcs_{-\la}$. The proposition is proved.
\end{proof}

  Case c): pencil of conics through two distinct points $A$ and $C$ tangent  to two given lines $L_A$ and $L_C$ through them; $L_A,L_C\neq AC$.  
  See Fig. \ref{fig4}. Fix an arbitrary point $M'\in AC\setminus\{ A, C\}$. 
  
  {\bf Claim 3.} {\it  The  projective angular symmetries
   $\sigma_A$, $\sigma_C$ with fixed point lines $L_C$ and $L_A$ respectively preserve the pencil. The involutions 
   $\sigma_M$, $\sigma_{M'}$ and  the composition 
  $\sigma_A\circ\sigma_C$ preserve each conic of the pencil.}
  
   The claim is proved analogously to 
  the above discussion, by considering the pencil in question as the limit of the family of 
  pencils through points $A_{\mu}$, $B_{\mu}$, $C_{\mu}$, $D_{\mu}$, 
  $A_{\mu}, B_{\mu}\to A$, $C_\mu, D_\mu\to C$, as $\mu\to0$ so that 
 $A_\mu B_\mu=L_A$, $C_\mu D_\mu=L_C$, and the lines $A_\mu C_\mu$, $B_\mu D_\mu$ are intersected at $M'$.  Similarly to the 
  above discussion, the involutions corresponding to $K_{A_{\mu}B_{\mu}}$ and 
  $K_{C_{\mu}D_{\mu}}$ converge to $\sigma_A$ and $\sigma_C$ respectively. 
  This implies the statement of the claim on the involutions $\sigma_M$, $\sigma_A$, 
  $\sigma_C$. 
  It remains to prove its statement  on the vertex $M'$. The intersection 
  point $M_2(\mu)$ of the lines $B_\mu C_\mu$ and $A_\mu D_\mu$, the point $M'$, 
  and the intersection points of the line $M_2(\mu)M'$ with lines $L_A$, $L_C$ form 
  a harmonic tuple of points on the line $M_2(\mu)M'$, as in  
  Claim 1. Hence, the involution $\sigma_{M',\mu}$ corresponding to the vertex 
  $M'$ and the perturbed pencil, with $\mu\neq0$, fixes $M'$ and each line through 
  $M'$  and permutes its intersection points with the lines $L_A$ and $L_C$. 
  Thus, it coincides with the involution $\sigma_{M'}$ corresponding to the nonperturbed 
  pencil. Hence, $\sigma_{M'}$ preserves  each conic of the nonperturbed pencil, as of  
  the perturbed one. 
  
  Consider the skew vertices in Cases c), d), e)  equipped with quasi-global dual billiard structures. 
  The corresponding involutions preserve the pencil and induce non-trivial conformal involutions of $\oc_\la$, by Proposition \ref{invpar}. 
  
  Consider now the vertices $C$ in Cases d) and e). In Case d) the involution $\sigma_C$ preserves 
  the singular conic $L\cup AB$ of the pencil and the conic tangent to $BC$ at $B$, as $\sigma_M$ in Claim 3. Hence, 
  it preserves the pencil. It does not preserve other conics, since their tangent lines at $B$ are not 
  $\sigma_C$-invariant. In Case e) $\sigma_C$ preserves each conic of the pencil. This can be 
  seen  in the affine chart $(z,w)$ for which $C$, $A$ are the intersection points of the infinity line with the 
  $z$- and $w$-axes respectively, and the conics  are the parabolas $w=z^2+\la$: 
  $\sigma_C(z,w)=\sigma_C(-z,w)$.   Proposition  \ref{glopro} is proved.
  \end{proof}
  \begin{proposition} \label{progroup} Let in a complex dual multibilliard all the curves be conics lying in a pencil, and 
their dual billiard structures  be defined by the same pencil. Let all its vertices be admissible for the pencil. 
Let $G\subset PSL_2(\cc)=Aut(\oc)$ denote the group generated by conformal transformations of the 
parameter space $\oc$ of the pencil induced by the involutions assotiated to the vertices, see Proposition \ref{glopro}. 
 Then the following statements are equivalent:

(i) The group $G$ is finite. 

(ii) The vertex collection satisfies Conditions 3)--5) of Definition \ref{multipb}. 

\noindent If  the group $G$ is finite, then it is either trivial (if and only if the multibilliard contains no skew vertex), 
or  isomorphic to $\zz_2$ or $S_3$.  One has $G=S_3$, if and only if the pencil has type a) 
and the multibilliard contains a pair of neighbor skew vertices $K_{EX}$, $K_{EY}$. 
\end{proposition}
 \begin{proof} 
 Case of pencil of type a). Then Conditions 3)--5) of Definition \ref{multip} impose no restriction 
 on admissible vertex collection. 
 The involution defining the dual billiard structure  
 at  each admissible vertex preserves the triple of 
 the singular conics $\Gamma_1$, $\Gamma_2$, $\Gamma_3$ of the pencil, see (\ref{gammy}). Therefore, 
 the conformal involutions $\oc\to\oc$ of the parameter space defined by the skew vertices 
 permute the corresponding parameter 
 values $\la_1$, $\la_2$, $\la_3$. The above conformal involutions (and their compositions) 
  are uniquely determined by the corresponding 
 permutations.  Hence, they generate a finite group $G\subset PSL_2(\cc)$ 
 isomorphic to a subgroup of $S_3$. 
 The conformal involution corresponding to a standard vertex is trivial. 
Each one of the involutions $\sigma_{K_{BC}}$, 
   $\sigma_{K_{AD}}$  fixes  
   $\Gamma_2$ and permutes $\Gamma_1$, $\Gamma_3$ (Corollary \ref {permut1}). Similarly, 
   each of $\sigma_{K_{AB}}$, $\sigma_{K_{CD}}$ fixes $\Gamma_1$ 
   and permutes $\Gamma_2$, $\Gamma_3$. 
   The two latter permutations generate all of $S_3$. This implies the statements of Proposition \ref{progroup}.

Case of pencil of type b). For every skew vertex 
equipped with a projective angular symmetry the latter symmetry fixes only the parameter 
values $\la_1$, $\la_2$ corresponding to the singular conics $\Gamma_1$, $\Gamma_2$ from (\ref{gammy2}), 
see Claim 2 in the proof of Proposition \ref{glopro}. 

Suppose the multibilliard contains only vertices of the above type. 
 Then the group $G$ is either trivial (if the skew vertex subset is empty), 
or isomorphic to $\zz_2$ (if it is non-empty), by the  above statement.

Let now the multibilliard contain the skew vertex $C$ equipped with a degenerate $\mcs$-angular symmetry defined 
by a regular conic $\mcs$ of the pencil. Let $\la_\mcs$ denote the parameter value corresponding to $\mcs$. 
 The conformal involution corresponding to the vertex $C$ fixes only $\la_2$ and $\la_{\mcs}$. Therefore, 
 if the multibilliard contains no other skew vertices, then $G\simeq\zz_2$. If it contains another skew vertex, then 
 $G$ is generated by two  involutions having only one common fixed  point $\la_2$. Their 
 composition is a parabolic transformation with the unique fixed point $\la_2$. It has infinite order. 
 Hence, $G$ is infinite. 
 
 Case of pencil of type c) is treated analogously.
 
 Case of pencil of type d). The involution $\sigma_C$ corresponding to a skew vertex $C\in L\setminus\{ A\}$ 
 is a projective angular symmetry fixing two conics: the singular conic $L\cup AB$ and the regular conic $\mcs$ of the 
 pencil that is tangent to the line $CB$ at $B$. The correspondence $\mcs\mapsto C$ is bijective. This implies that 
 $G$ is finite, if and only if the involution corresponding to any 
 other skew vertex of the multibilliard fixes the same conic $\mcs$, as in the above discussion. This holds, if and only if 
 Conditions 3)--5) of Definition \ref{multip} hold. 
 
 Case of pencil of type e) is treated analogously, with the singular conic now being the double line $L$. Proposition 
 \ref{progroup} is proved.
 \end{proof}
 
 \begin{proposition} \label{pratint} Every pencil type multibilliard  has a rational integral of  
 degree $2|G|\in\{2, 4, 12\}$, 
  where $|G|$ is the cardinality of  the group $G$.
  \end{proposition}
 
 \begin{proof} The group $G$ is finite, by Proposition \ref{progroup} and since  the multibilliard 
 is of pencil type (hence, satisfying Conditions 3)--5) of Definition \ref{multip}). 
  Let $F$ be a quadratic first integral of the pencil: 
  the ratio of two quadratic polynomials defining two its conics. Its constant 
 value on each conic coincides 
 with the corresponding parameter $\la$ (after replacing $F$ by its post-composition with conformal automorphism 
 $\oc\to\oc$). The product $\prod_{g\in G}g\circ F$ is a rational first integral of the multibilliard, since it is 
 invariant under the vertex involutions  (by  
 definition) and   the dual billiard involution of each tangent line to a multibilliard conic  
 permutes its intersection points with each conic of the pencil. Proposition \ref{pratint} is proved.
 \end{proof}

 \subsection{Foliation by level curves of rational integral. Proof of Theorems \ref{thmd1} and 
 \ref{thdeg}}
 
 \begin{definition} Consider a rationally integrable dual billiard structure on a complex conic $\gamma$ 
 (which belongs  to the list  given by Theorem \ref{tgerm}). In the case, when  
it is  defined by a pencil of conics, its {\it canonical integral} is a quadratic rational function constant on each conic of the 
pencil that vanishes on $\gamma$. In the case, when it is exotic, its {\it canonical integral} is the one given by the Addendum to Theorem \ref{tgerm} (whose zero locus is $\gamma$). 
 \end{definition}
 \begin{proposition} \label{proconst}
  Every rational integral of a rationally integrable  dual billiard on a conic is constant on each 
 irreducible component of each level curve of its canonical integral. 
 \end{proposition}
  
  \begin{proof} Let  $\gamma$ be the conic in question,  $\Psi$ be a rational integral of 
  the dual billiard, and let $R$ be its canonical integral. 
   We have to show that $\Psi\equiv const$ along the leaves of the foliation $R=const$ (which are, 
  by definition, the  irreducible components of level curves of the function $R$ with its critical and indeterminacy 
  points deleted). It 
  suffices to prove the above statement in a small neighborhood of the conic $\gamma$. 
Fix a point $P\in\gamma$ such that it is a regular point for the foliation and  the dual billiard involution $\sigma_P$ 
is defined there. (In fact, $\sigma_P$ is well-defined whenever $P$ is regular for the foliation. But we will not use this.) 
  Let $U\subset\cp^2$ 
   be a small neighborhood   of the point $P$ that is a  flowbox 
   for the foliation $R=const$ and whose closure is disjoint from singular points of the foliation 
   and indeterminacy points for the 
   involution family $\sigma_t$, $t\in\gamma$. We equip it with biholomorphic 
   coordinates $(x,y)$, where the local leaves of the flowbox are the horizontal 
   fibers $y=const$. Fix a point $P_0\notin\gamma$ close to $P$.  Take a tangent line $\ell_0$ to $\gamma$ through $P_0$; 
   let $Q_0$ denote the tangency point. Set $P_1=\sigma_{Q_0}(P_0)$. 
   Let $\ell_1$ be the tangent line to $\gamma$ through $P_1$ distinct from $\ell_0$, and let $Q_1$ be their 
   tangency point. Set 
  $$P_2=\sigma_{Q_1}(P_1),\ \text{ etc. } \ P_N=\sigma_{Q_{N-1}}(P_{N-1}); \ \ x_j=x(P_j).$$ 
  Here $N$ is the biggest number such that $P_1,\dots,P_N, Q_0,\dots,Q_{N-1}\in U$. 
   We claim that  as $P_0\to P$, the cardinality $N=N(P_0)$ of the above sequence tends to infinity. 
   This follows from the fact the involutions $\sigma_Q|_{L_Q}$, $Q\in\gamma\cap U$ 
   are uniformly asymptotic to the central symmetries $x\mapsto2x(Q)-x$ with respect to the 
   points $x(Q)$, as $x-x(Q)\to0$ and $Q\in U$: they are non-trivial 
   conformal involutions of the lines  $L_Q$ with fixed points $Q$.   Therefore, 
   $N>\deg\Psi\deg R$, whenever $P_0$ is close enough to $P$. One has 
   $$\Psi(P_0)=\dots=\Psi(P_N), \ \ R(P_0)=\dots=R(P_N),$$
   since both $\Psi$ and $R$ are integrals. This together with Bezout Theorem and the above inequality 
   implies that $\Psi\equiv const$ along each leaf of the foliation $R=const$. Proposition \ref{proconst} is proved.
   \end{proof}

   \begin{lemma} \label{lirr} Let $R$ be a  rational first integral  of an exotic 
   dual billiard structure from Theorem \ref{tgerm} given by the corresponding formula in its addendum. 
   
   1) For all but a finite number of values of  
   $\lambda\in\cc$  the complex level curve 
   $$\Gamma_{\la}:=\{ R=\lambda\}$$
    is irreducible of degree $d=\deg R$.

    2) The (punctured) curve $\gamma=\{ w=z^2\}$ is a {\bf multiplicity $\frac d2$ leaf} 
   of the foliation $R=const$, which means  that  each small  transversal cross-section 
   to $\gamma$ intersects each leaf   close enough to $\gamma$ (depending on  cross-section) 
    transversely at $\frac d2$ distinct points;  moreover, $[\Gamma_\la]\to\frac d2[\gamma]$ as divisors, as $\la\to0$.  
    
    3) The   curve $\gamma$ is the unique nonlinear multiplicity $\frac d2$ leaf.
   \end{lemma}
   \begin{proof}   For the proof of Statement 1)   it suffices to prove irreducibility of the level curve $\Gamma_\la$  for an open subset of values $\la$. We  
    will prove this for generic small $\la$: for an open set of values $\la$ accumulating to zero 
    (for all $\la\neq0$ small enough in Cases 2a1), 2a2).) 
  Indeed, it is well-known that if the level curve $\{ R=\la\}$ of a rational function is irreducible for an open subset 
  of values $\la$, then it is irreducible for all but a finite number of $\la$. This is implied by  the two following statements: 
  
-  each indeterminacy point can be resolved by a sequence of blow-ups, so that the function in question becomes a 
  well-defined $\oc$-valued holomorphic funciton on a new connected compact manifold, a blown-up $\cp^2$;  
  
  -  every non-constant holomorphic $\oc$-valued function on a connected compact complex manifold 
  has finite number of critical values.

  Let us first consider Case 2a1), when 
$R(z,w)=\frac{(w-z^2)^{2N+1}}{\prod_{j=1}^N(w-c_jz^2)^2}$, see (\ref{exot1}). 
   
   {\bf Claim 4.} {\it The germ of the curve $\Gamma_\la$ at the point $Q=[0:1:0]\in\cp^2$ (i.e., at the intersection point 
   of the infinity line with the $w$-axis) is irreducible, whenever $\la\neq 0, 1, \infty$. Its intersection index with each line through $Q$ distinct from the infinity line $L_Q=T_Q\gamma$ is equal to $2N+1$.}
   
   \begin{proof} Let $\la\neq0,1,\infty$. In the affine chart $(\wt z,\wt w)=(\frac zw, \frac1w)$ centered at $Q$ one has 
   $$\Gamma_\la=\{(\wt w-\wt z^2)^{2N+1}-\la\wt w^2\prod_{j=1}^N(\wt w-c_j\wt z^2)^2=0\}.$$
   In the coordinates $(\wt z, u)$, $u:=\wt w-\wt z^2$, $\Gamma_\la$ is the zero locus of the polynomial 
   \begin{equation} \mathcal P_\la(\wt z,u):=u^{2N+1}-\la(u+\wt z^2)^2\prod_{j=1}^N(u+(1-c_j)\wt z^2)^2.\label{pla}
   \end{equation}
   To prove irreducibility of the germ $(\Gamma_\la,Q)$, it suffices to show that the germ of the polynomial $\mathcal P_\la$ at the origin is irreducible. To do this, we 
    will deal with its Newton diagram. Namely, 
   consider the bidegrees $(m,n)\in(\rr_{\geq0}^2)_{x,y}$ of all the 
   monomials $\wt z^m u^n$ entering $P_\la$. Consider the convex hull of the union of the corresponding quadrants $(m,n)+\rr_{\geq0}^2$. 
   The union $\mathcal{ND}$ of its  boundary edges except for the coordinate axes  is called the {\it Newton diagram.} 
   We claim that the Newton diagram of the polynomial $P_\la$ is one edge 
   $E=[(4N+4,0), (0, 2N+1)]$. Indeed, for $\la\neq0,1,\infty$ the bidegrees of the monomials entering $P_\la$ are $(0,2N+1)$ and 
   a collection of points of the line $\{ 2y+x=4N+4\}$, since the  multiplier at $\la$ in (\ref{pla}) 
   is a $(2,1)$-quasihomogeneous polynomial. But the bidegrees  in the latter line lie above the edge 
    $E$, except for its 
   vertex $(4N+4,0)$. This 
   proves that $\mathcal{ND}=E$. 
   
   Suppose the contrary: the germ of the polynomial $\mathcal P_\la$ is not irreducible. Then it is the product of 
   two germs of analytic functions with Newton diagrams being edges parallel to $E$ 
   whose endpoints lie in  $\zz^2$. The latter edges should be  closer to the origin 
   than $E$ and have smaller lengths. But $E$ is the edge of smallest length among all the above 
   edges, since $E$ contains no integer points in its interior: the numbers $4N+4$ 
   and $2N+1$ are coprime. The contradiction thus obtained proves irreducibility of the germ of the polynomial 
   $\mathcal P_\la$. The intersection index statement of Claim 4 follows from irreducibility,  the Newton diagram statement $\mathcal{ND}=E$ 
    and the fact that $T_Q\gamma=\{ du=0\}\subset T_Q\cp^2$ (since $L_Q=\{\wt w=0\}=\{ u+\wt z^2=0\}$).
   \end{proof}
   
 Let us prove irreducibility of the curve $\Gamma_\la$ with $\la\neq0$ small enough. Fix two 
 distinct points $P_1, P_2\in\gamma^o:=\gamma\setminus\{ Q, O\}$, a path 
 $\alpha\subset\gamma^o$ connecting them and a path $\beta\subset\gamma\setminus\{ O, P_2\}$ connecting $P_1$ to $Q$. Fix small neighborhoods $U_j=U_j(P_j)\subset\cp^2$.  
The line  $P_1P_2$ intersects $\Gamma_\la$ with small 
 $\la\neq0$ at two subsets $\Sigma_j\subset U_j$, each consisting of $2N+1$ points, 
 by (\ref{pla}) and since $\deg\Gamma_\la=4N+2$. 
 It suffices to show that the whole intersection $\Sigma_1\cup\Sigma_2$ lies in one 
 irreducible component of the curve $\Gamma_\la$. 
 As $P_1\to Q$ along the path $\beta$, $\Sigma_2$ remains in $U_2$ (if $\la$ is small enough). 
 Some $2N+1$ points in $\Sigma_1$ should converge to $Q$, by the intersection index 
 statement of Claim 4. But the set $\Sigma_1$ consists of exactly $2N+1$ points, and hence, 
all its points converges to $Q$. (Modifying slightly the path $\beta$, one can achieve that 
 the trajectories of points in $\Sigma_1$ going to $Q$ avoid  possible singularities 
 of the curve $\Gamma_\la$.) Thus, $\Sigma_1$ lies in the irreducible component 
 of the curve $\Gamma_\la$ containing the germ $(\Gamma_\la,Q)$. Points of the set 
 $\Sigma_2$ are connected to points in $\Sigma_1$ by paths close to $\alpha$ going along 
 leaves of the foliation $R=const$, thus, along the regular part of the curve $\Gamma_\la$. 
  (The latter paths define the holonomy map between  neighborhoods 
 of the points $P_2$, $P_1$ in the line $P_1P_2$.) Hence, $\Sigma_2$ lies in the same 
 irreducible component. This component has degree  no less than 
 the sum $4N+2$ of cardinalities of the sets $\Sigma_j$, which is equal to 
 $\deg\Gamma_\la$. Hence, the component in question is the whole curve 
 $\Gamma_\la$, and thus, $\Gamma_\la$ is irreducible.

  Case of  integral $R$ given by (\ref{exot2}) is treated analogously with the following modification: 
  in the above coordinates $(\wt z,u)$ the Newton diagram of the new polynomial $\mathcal P_\la$ 
  is $[(2N+3,0),(0,N+1)]$;  $2N+3$, $N+1$ are again coprime.

  The other canonical rational integrals have degrees 4 or 6 and the type  
  \begin{equation}R(z,w)=\frac{(w-z^2)^m}{\Phi(z,w)}, \ \Phi \text{ is a polynomial, } \deg\Phi=2m, \ m\in\{2,3\}.
  \label{rat2m}\end{equation}
   \begin{proposition} \label{prat2m} 
   Let $R$ be as in (\ref{rat2m}). Let there exist a sequence of values $\la$ converging to zero for 
   which the  curve $\Ga_\la:=\{ R=\la\}$ is not irreducible. Then the foliation $R=const$ is a pencil of conics. 
   \end{proposition}
   \begin{proof} Passing to a subsequence we can and will consider that one of the following statements holds for all 
   above $\la$:

  (i) $m=2$ and $\Gamma_\la$ is a union of two regular conics $C_{1,\la}$, $C_{2,\la}$; 
  
  (ii) $\Gamma_\la$ contains a line;
  
  (iii) $m=3$ and $\Ga_\la$ is a union of two regular cubics $C_{1,\la}$, $C_{2,\la}$;
  
  (iv) $m=3$ and $\Ga_\la$ is a union of three regular conics. 
  
 Statement (ii) cannot hold: the contrary would imply that 
  the limit conic $\Gamma_0=\gamma=\{ w=z^2\}=\lim_{\la\to0}\Gamma_\la$ contains a line, which is not true. 
  Suppose  (iii) holds. Then each cubic considered as a divisor of degree three converges to an integer multiple 
  of the divisor $[\gamma]$ of degree two: thus, to a divisor of even degree. This is obviously impossible. 
  Therefore, the only possible cases are (i) and (iv). The a priori possible intersection points of 
  the conics from (i), (iv) lie in the  finite set of indeterminacy and critical points of the rational function $R$. 
  Therefore, passing to a subsequence one can and will achieve that a family of conics $C_\la\subset\Ga_\la$ 
  lies in a pencil. The function $R$ is constant on them for infinite number of values of $\la$. 
    Therefore, it is constant on each conic of the pencil, since the set of those parameters of the pencil for which 
    $R=const$ on the corresponding conics is finite (being algebraic). Finally, 
 the foliation $R=const$ is a pencil of conics. 
 \end{proof}
 
Let $R$ be a degree four integral given by (\ref{exo2bnew}) or (\ref{exo2bnew2}). We treat only  case 
  (\ref{exo2bnew}), since the integrals (\ref{exo2bnew}) and (\ref{exo2bnew2}) are obtained one from the other 
  (up to constant factor) by complex projective transformation fixing the conic $\gamma=\{ w=z^2\}$. Thus, 
  $$R=R_{b1}(z,w)=\frac{(w-z^2)^2}{(w+3z^2)(z-1)(z-w)}.$$
  Suppose the contrary: the curve $\Gamma_\la:=\{ R=\la\}$ is not irreducible 
  for a sequence of numbers $\la$ converging to zero. Then the foliation $R=const$ is a pencil of conics, 
  by Proposition \ref{prat2m}. It contains the 
  conics $\gamma$ and $\{ w+3z^2=0\}$, which are tangent to each other at the origin and at infinity. Therefore, the pencil consists 
  of conics tangent to them at these points. On the other hand, the line $\{z=1\}$ lies in the polar locus $\{R=\infty\}$. 
 Hence, it should lie in a conic from the pencil. But this is obviously impossible, -- a contradiction.
  
  Let now $R$ be a degree 6 integral from the Addendum to Theorem \ref{tgerm}, Cases 2c) or 2d). 
  Supposing the contrary to irreducibility, we similarly get that the foliation $R=const$ is a pencil of conics. 
  But in both Cases 2c) and 2d) the polar locus $\{ R=\infty\}$  
  contains an irreducible cubic, see \cite[subsections 7.5, 7.6]{grat}. This contradiction  proves Statement 1) of 
  Lemma \ref{lirr}.
  
  Statement 2) of Lemma \ref{lirr} follows from Statement 1) and the fact that $\gamma$ is a multiplicity $\frac d2$ 
  zero curve of the integral $R$. 
  
  Let us prove Statement 3). Suppose the contrary: there exists another leaf $\alpha$ of multiplicity $\frac d2$  
  and degree $\mu\geq2$. Then for every given line $L$ that is transversal to $\alpha$ 
  and does not pass through singularities of the foliation 
 each leaf close enough to $\alpha$ intersects $L$ in at least $\mu\frac d2\geq d$ points. 
 The number of intersection points cannot be greater than $d$. Hence, $\mu=2$ and $\alpha$ is a conic. 
 Let us renormalize the integral $R$ by postcomposition with M\"obius transformation $\nu$ to an integral 
 $\wt R=\nu\circ R$ so that $\wt R|_\gamma=0$, $\wt R|_\alpha=\infty$. Let $Y(z,w)$ be a quadratic polynomial 
 vanishing on $\alpha$. Then 
 $$\wt R=\left(\frac{z-w^2}{Y(z,w)}\right)^{\frac d2},$$
 up to constant factor, by construction and multiplicity assumption. Therefore, the foliation $\wt R=const$ is a pencil 
 of conics containing $\gamma$ and $\alpha$, and so is $R=const$. But its generic leaves 
 are punctured irreducible algebraic curves $\Gamma_\la$ of degree $d\geq4$. This contradiction   
 proves Lemma \ref{lirr}. 
  \end{proof}

   \begin{proof} {\bf  of Theorem \ref{thmd1}.}  Consider a rationally integrable 
   dual multibilliard  with integral $\Psi\not\equiv const$. 
   Then the dual billiard on each its curve $\gamma_j$ is rationally 
   integrable with  integral $\Psi$. Hence, each $\gamma_j$ is a conic equipped with either 
   pencil type, or exotic dual billiard structure, by Theorem \ref{tgerm},  and 
   $\Psi|_{\gamma_j}\equiv const$,  by \cite[proposition 1.35]{grat} (or by Proposition \ref{proconst}). 
   
   Case 1). Let some two conics $\gamma_1$, $\gamma_2$ be the same conic $\gamma$ 
   equipped with two distinct dual billiard structures, given by projective involution families 
   $\sigma_{P,j}:L_P\to L_P$, $j=1,2$. Here $P$ lies outside a finite set: the union of the indeterminacy loci of 
   families $\sigma_{P,j}$, which are finite by Theorem \ref{tgerm}. The product 
   $g:=\sigma_{P,1}\circ\sigma_{P,2}$ is a parabolic transformation $L_P\to L_P$, having unique fixed point $P$. 
   The integral $\Psi$ is $g$-invariant: $\Psi\circ g=\Psi$ along each line $L_P$. But each non-fixed 
   point of a parabolic transformation has infinite orbit. Therefore, $\Psi\equiv const$ along each line tangent to 
   $\gamma$. But we know that $\Psi$ is constant along the curve $\gamma$, as noted above. 
   Therefore, $\Psi\equiv const$, by the two latter statements and since the union of 
  lines tangent to $\gamma$ at points lying in an open subset in $\gamma$ contains an open subset in $\cp^2$. 
  The contradiction thus obtained proves that Case 1) is impossible.
  
  Case 2):  there are at least two geometrically distinct conics, say, $\gamma_1$, $\gamma_2$. 
  For every $j=1, 2$ let $R_j$ denote the 
  canonical integral of the corresponding dual billiard structure.   We have to prove the two following statements: 
  
  1) the  dual billiard structure on each $\gamma_j$ is defined by a pencil of conics, that is, the degree 
  $d_j:=\deg R_j$ is equal to 2; 
  
  2) the latter pencil is  the same for $j=1, 2$, and it contains both $\gamma_j$.

  Let $\mcf$ denote the foliation $\Psi=const$.  
   For every $j$ for all but a finite number of  values $\la\in\cc$ the complex level curve 
  $\{ R_j=\la\}$ is irreducible of degree $d_j$, by Lemma \ref{lirr}, and $\Psi\equiv const$ along it (Proposition \ref{proconst}). Hence, each foliation $R_j=const$ coincides with $\mcf$. This together with the previous 
 statement implies that  the degrees $d_j$ are equal, set $d=d_j$, and both (punctured) 
 conics $\gamma_1$, $\gamma_2$ 
are leaves of the same multiplicity $\frac d2$ for the foliation $\mcf$. Therefore, the foliation $\mcf$ is 
a pencil of conics containing $\gamma_1$ and $\gamma_2$, by Statement 3) of Lemma \ref{lirr}.
 Hence, all the conics of the multibilliard lie in this pencil, and $d=2$ 
  (since  $d$ is  the degree of irreducible level curve of the function $R_1$).  
  Thus, each $R_j$ is a ratio of two quadratic polynomials, and its level curves are conics from the 
 pencil. Hence, the dual billiard structure on $\gamma_j$ is given by the same pencil.  
  Theorem \ref{thmd1} is proved.
  \end{proof}
  
  \begin{proof} {\bf of Theorem \ref{thdeg}.} A rational first integral of a pencil type multibilliard is constant on 
  each conic of the pencil (Proposition \ref{proconst}). Moreover, it is constant on every union of those conics 
 whose parameter values $\la$ lie in the same $G$-orbit. Here $G$ is the group from Proposition \ref{progroup}. 
 The cardinality of a generic $G$-orbit is equal to the cardinality $|G|$ of the group $G$, since a generic point 
 in $\oc$  has trivial stabilizer in $G$. Thus, the minimal degree of the integral (which is achieved, by 
  Proposition \ref{pratint}) is $2|G|\in\{2,4,12\}$. This together with Proposition \ref{progroup} 
   implies the statement of Theorem \ref{thdeg}. 
  \end{proof}
   
   \subsection{Dual billiard structures at vertices. Birationality and types of involutions}
   
   \begin{proposition} \label{altpr} Let $A$ be a point in $\rp^2$ ($\cp^2$) equipped with real (complex) dual 
   billiard structure given by involution family $\sigma_{A,\ell}$ that has a real (complex) rational first integral $\Psi\not\equiv const$: 
   $\Psi\circ\sigma_{A,\ell}=\Psi$ on each line $\ell$ through $A$ on which the involution is defined. 
      Let the foliation $\Psi=const$ be not the family of lines through $A$.  Then $\sigma_{A,\ell}$  
 coincide (up to correction at a finite number of lines $\ell$) 
 with a  birational involution $\sigma_A:\cp^2\to\cp^2$ fixing all but a finite number of lines 
  through $A$ 
 and holomorphic and bijective on the complement to a finite number of lines through $A$. 
 (Thus, it is a de Jonqui\`eres involution, see Footnote 2 and \cite[p.422]{blanc}.)  The rational integral 
 $\Psi$ and the corresponding foliation $\Psi=const$ are $\sigma_A$-invariant\footnote{Invariance of a rational function $\Psi$ under a birational transformation 
 $\sigma:\cp^2\to\cp^2$ means that $\Psi\circ\sigma=\Psi$ on the open and dense subset 
 where both $\Psi$ and $\Psi\circ\sigma$ are well-defined. If $\sigma$ is a de Jonqui\`eres involution $\sigma_A$ fixing all but a finite of number of lines through $A$, and $\Psi$ is 
 $\sigma_A$-invariant, then the $\sigma_A$-image (strict transform) of each level curve 
 $\Gamma_\la=\{\Psi=\la\}$ coincides with $\Gamma_\la$ except for the following case:  if 
 $\Gamma_\la$ contains some lines through $A$, then $\sigma_A$ may 
 contract some of its lines to $A$, and its other irreducible components (if any) 
 are permuted by 
 $\sigma_A$.}
 \end{proposition} 
\begin{proof}  Let $\sigma_{A,\ell}:\ell\to\ell$ be the 
   corresponding projective involution family acting on lines $\ell$ through $A$. They are defined on lines 
   $\ell$ through $A$ from an open subset $U\in\cp^1$ in complex case ($U\subset\rp^1$ in real case). 
   
   Fix a  nonlinear complex level curve $X:=\{\Psi=\la\}$. Fix an $\ell_0\in U$ (consider it as a complex line) such that the 
  points of intersection $X\cap\ell_0$ distinct from $A$ are regular points of the curve $X$, the intersections are 
  transversal, and the multiplicity of the intersection $X\cap\ell_0$ at $A$ is minimal. 
  There exists a simply connected neighborhood $V=V(\ell_0)\subset\cp^1$ such that for every 
  $\ell\in V$ the number of  geometrically distinct points of the set $X_\ell:=(X\cap\ell)\setminus\{ A\}\subset\cp^2$ 
  is the same 
  (let us denote their number by $d$),  
  and they depend holomorphically on $\ell$ (Implicit Function Theorem). We numerate these holomorphic 
  intersection point families by indices $1,\dots,d$. For every $\ell\in V$ the involution 
  $\sigma_{A,\ell}$ makes a ($\ell$-dependent) permutation of the latter intersection  points, which is identified 
  with a permutation  of indices $1,\dots,d$: an element in $S_d$. There exists a 
  permutation $\alpha\in S_d$ realized by $\sigma_{A,\ell}$ for a continuum cardinality subset $Y\subset V$ 
   of lines.  Let us fix it. 
  
  {\bf Claim 6.} {\it There exists a projective involution family $\wt\sigma_{A,\ell}:\ell\to\ell$ 
 depending holomorphically on the parameter $\ell\in V$ 
  that makes the permutation $\alpha$ on $X_\ell$ for every $\ell\in V$. The rational function $\Psi|_\ell$ is 
  $\wt\sigma_{A,\ell}$-invariant: $\Psi\circ\wt\sigma_{A,\ell}=\Psi$ on every $\ell\in V$.} 
  
  \begin{proof} Consider first the case, when $X_\ell$  is just one point. 
  For every $\ell\in V$ set  $\wt\sigma_{A,\ell}:\ell\to\ell$ to be the non-trivial conformal involution 
  fixing the points $X_\ell$ and $A$. It depends holomorphically on $\ell\in V$. It preserves 
  $\Psi|_\ell$:  $\Psi\circ\wt\sigma_{A,\ell}=\Psi$ on  every $\ell\in Y$, and the latter relation 
  holds for every $\ell\in V$, since $Y$ is of cardinality continuum and by uniqueness of analytic extension. 
  
  Let now $X_\ell$ consists of at least two points. Let us define $\wt\sigma_{A,\ell}$ to be the unique projective 
  transformation $\ell\to\ell$ fixing $A$ and sending
  the points in $X_\ell$ with indices $1$, $2$ to the points with indices $\alpha(1)$,  
  $\alpha(2)$ respectively. For every $\ell\in V$ this is an involution preserving $\Psi|_\ell$, since this is true for 
  every $\ell\in Y$ and by uniqueness of analytic extension. The claim is proved.
  \end{proof}
  
  {\bf Claim 7.} {\it The  involution family $\wt\sigma_{A,\ell}$ 
  extends holomorphically to a finitely punctured space $\cp^1$ 
  of lines through $A$. It coincides with $\sigma_{A,\ell}$ on all the lines $\ell\in U$ except maybe for a 
   finite number of them, on which  $\Psi=const$.}
  
  \begin{proof} We can extend the involution family 
  $\wt\sigma_{A,\ell}$ analytically in the parameter $\ell$ along each path avoiding a finite number of lines $\ell$ 
  for which either some of the points in $X_\ell$ are not transversal intersections, or  the index of intersection 
  $\ell\cap X$ at $A$ is not the minimal possible. This follows from the previous claim and its proof. 
  Extension along a closed path does not change holomorphic branch. Indeed, otherwise  
  there would exist  its another holomorphic branch over a domain $W\subset V$: an 
  involution family $H_{A,\ell}:\ell\to\ell$ depending holomorphically on $\ell\in W$, 
  $H_{A,\ell}\neq\wt\sigma_{A,\ell}$, which preserves the integral $\Psi$. The product 
  $F_\ell:=\wt\sigma_{A,\ell}\circ H_{A,\ell}:\ell\to\ell$ is a parabolic projective transformation, with $A$ being its 
  unique fixed point, for every $\ell\in W$. Its orbits are infinite, and $\Psi$ should be constant along each of them. 
  This implies that $\Psi=const$ along each line $\ell\in W$. Hence, the foliation $\Psi=const$ is the family of lines 
  though $A$, which is forbidden by our assumption. The contradiction thus obtained proves singlevaluedness 
  of analytic extensions of the involution family $\wt\sigma_{A,\ell}$ along paths and the first statement of the claim. 
  Its second statement follows from the fact that for those $\ell\in U$ for which $\wt\sigma_{A,\ell}\neq\sigma_{A,\ell}$, 
  one has $\Psi\equiv const$ along $\ell$: see the above argument, now with the parabolic transformation 
  $\wt\sigma_{A,\ell}\circ\sigma_{A,\ell}$. The claim is proved.
  \end{proof} 
  
  Without loss of generality we consider that $\sigma_{A,\ell}=\wt\sigma_{A,\ell}$, correcting $\sigma_{A,\ell}$ 
  at a finite number of lines. The latter equality defines  analytic extension of the involution family $\sigma_{A,\ell}$ 
  to all but a finite number of lines $\ell$ through $A$. 
   The invariance condition $\Psi\circ\sigma|_\ell=\Psi|_\ell$ is a system of algebraic 
  equations on the pairs $(\ell,\sigma)$, where $\ell$ is a projective line through $A$ and $\sigma:\ell\to\ell$ 
  is a non-trivial projective involution fixing $A$. For every line $\ell$ through $A$ (except for a finite set of lines,  along which $\Psi\equiv const$)  its solution $\sigma$ is at most unique, 
  by the above argument.  This implies that the family $\sigma_{A,\ell}$ is a connected open subset in an algebraic subset of a smooth algebraic manifold and all $\sigma_{A,\ell}$ paste together to a global birational automorphism $\cp^2\to\cp^2$ acting 
  as a holomorphic involution on the complement to a finite number of lines through $A$.  It preserves $\Psi$, and hence, 
  the foliation $\Psi=const$, by construction. Proposition \ref{altpr} is proved.
  \end{proof}
  
  \begin{lemma} \label{alt2} 
  The dual billiard structure at each  vertex $A$ of any rationally integrable dual multibilliard 
  is either global, or quasi-global. In the  case of quasi-global structure the foliation by level curves of 
  rational integral is a pencil of conics, and the conic of 
  fixed points of the corresponding involution $\sigma_A$ lies in the same pencil. 
  \end{lemma}
   
\begin{proof}  Consider first the case, when the foliation by level curves of integral is a pencil of conics. 
It is invariant under the birational involution $\sigma_A$ from Proposition \ref{altpr}. Thus, the involution $\sigma_A$ permutes conics of the pencil (except for at most a finite number of singular conics containing some lines through $A$, see Footnote 3). Therefore, $\sigma_A$ acts on  parameter space $\oc$ of the pencil 
as a conformal involution with at least two fixed points (by Erasing Singularity Theorem). Thus, $\sigma_A$ fixes pencil 
parameters of at least two distinct conics of the pencil. 
Fix one of them that is not a pair of lines intersecting at $A$, let us denote it by $\Gamma$. It is possible, since 
a pencil cannot contain two singular conics, each of them being a pair of lines, so that all the four lines forming 
them pass through the same point $A$. Indeed, otherwise $A$ would be the unique base point of the pencil, 
and the pencil would have type e). Thus, its only singular conic would be the double line tangent to all its 
regular conics at $A$, -- a contradiction. 

Subcase 1.1): $\Gamma$ is disjoint from $A$. Then $\sigma_A$ is a projective involution, by  Example \ref{extypes3}. 

Subcase 1.2): $\Gamma$ passes through $A$. If $\Gamma$ is a union of lines, then  
some of its lines, let us denote it by $L$, does not pass through $A$. Then $\sigma_A$ is the projective involution 
that fixes each point of the line $L$, by definition. Similarly, if $\Gamma$ is a regular conic, then $\sigma_A$ 
fixes each its point. Hence, it defines a quasi-global dual billiard structure. 

Consider now the  case, when the foliation by level curves of the integral is not a pencil. Then 
the multibilliard contains just one conic, let us denote it by $\gamma$, 
 equipped with an exotic dual billiard structure. Let $R$ be its canonical integral, $d=\deg R$. 
 The foliation $\Psi=const$ 
 coincides with  $R=const$, by Proposition \ref{proconst}. The (punctured) curve $\gamma$ being a leaf 
 of multiplicity $\frac d2$, its (punctured) image $\gamma':=\sigma_A(\gamma)$ is also 
a multiplicity $\frac d2$ leaf, since $\Psi\circ\sigma_A=\Psi$ and 
 multiplicity  is invariant under birational automorphism of foliation. Hence, $\gamma'=\gamma$, if $\gamma'$ is not a line (Lemma \ref{lirr}, Statement 3)). 

Subcase 2.1): $A\notin\gamma$. Then a generic line through $A$ intersects $\gamma$ at two points 
distinct from $A$. Hence, the same holds for the image $\gamma'$. Thus, it is not a line, 
and $\sigma_A(\gamma)=\gamma$. 
Therefore, $\sigma_A$ is a global projective transformation, the  
$\gamma$-angular symmetry.

Subcase 2.2): $A\in\gamma$ and $\gamma'\neq\gamma$. Let us show that this case is impossible. 
Indeed, then $\gamma'$ is a line, see 
the above discussion, and $R|_{\gamma'}\equiv const\neq0$.  
Therefore, the points of  intersection $\gamma'\cap\gamma$ 
are  indeterminacy points for the  function $R$. 
In Case 2a) the only indeterminacy points are the origin $O$ and the infinity. Therefore, 
$\gamma'$ is some of the following lines: the $Ow$-axis (which passes through both latter points), the $Oz$-axis 
or the infinity line (which are tangent to $\gamma$ at $O$ and at infinity respectively). But each of the latter lines 
satisfies at least one of the following statements: 

- either $R$ is non-constant there;

- or $R$ has a  pole of multiplicity less than $\frac d2$ there. 

See the two first formulas for the integrals in the addendum to Theorem \ref{tgerm}. 
Therefore, the (punctured) line $\gamma'$ cannot be a multiplicity $\frac d2$ leaf of the foliation $R=const$. 
This contradiction  proves that the case under consideration is impossible. 
Cases 2b), 2c), 2d) are treated analogously. 

Subcase 2.3): $A\in\gamma=\sigma_A(\gamma)$. Therefore, for every  line $\ell$ through $A$ distinct 
from the line $L$ tangent to $\gamma$ at $A$ the involution 
$\sigma_A$ fixes the point of intersection $\ell\cap\gamma$ distinct from $A$. 
Thus, it is the degenerate $\gamma$-angular symmetry. In the chart $(x,y)$ where $\gamma=\{ y=x^2\}$, 
$A$ is the point of the parabola $\gamma$ at infinity, and  $L$  is the infinity 
line,  $\sigma_A$ acts as 
$$\sigma_A:(x,y)\mapsto(x,2x^2-y).$$
 In the coordinates $(x,y)$ one has 
\begin{equation}R=R(x,y)=\frac{(y-x^2)^m}{F(x,y)}, \ F(x,y) \text{ is a polynomial, } \deg F\leq 2m.\label{rqxy}
\end{equation}
To treate the case in question, we use the following claim.

{\bf Claim 8.} {\it The point $A$ is an indeterminacy point of the function $R$.}

\begin{proof} Let us first consider the case, when   $L$ lies in a level curve $S_\la:=\{ R=\la\}$. Then $\la\neq0$, since  $S_0=\gamma$. Thus, $A$ lies in two distinct level curves, and hence, is 
an indeterminacy point. Let us now suppose that $R|_L\not\equiv const$. 
As a line $\ell$ through $A$ tends to the tangent line $L$ to $\gamma$, 
its only intersection point $B(\ell)$ with $\gamma$ distinct from $A$ tends to $A$. Therefore, the involution 
$(\sigma_A)|_\ell$, which fixes the confluenting points $A$ and $B(\ell)$,  tends to the constant map $L\to A$ uniformly on compact subsets in $L\setminus\{ A\}$. Suppose the contrary: $A$ is not an indeterminacy point. Fix a generic $\la\neq0$, for which $S_\la=\{ R=\la\}$ is irreducible 
(see Lemma \ref{lirr}, Statement 1)). 
Its image (strict transform) $S_{\la'}:=\sigma_A(S_\la)$ is an irreducible component 
of  a level curve $\{ R=\la'\}$, $\la'\neq0$. Hence, $A\notin S_\la,S_{\la'}$, since $A\in S_0$. 
 Therefore, the points of intersections 
$\ell\cap S_\la$ do not accumulate to $A$, as $\ell\to L$. Hence, the  points in $\sigma_A(\ell\cap S_\la)\subset S_{\la'}$ converge  to $A$, and thus, $A\in S_{\la'}$. This contradiction proves the claim.
\end{proof}

{\bf Claim 9.} {\it Let $F$ be the same, as in (\ref{rqxy}). Then 
\begin{equation}F\circ\sigma_A(x,y)=(-1)^{m+1}F(x,y) +a(y-x^2)^m, \ \ a=const\in\cc.\label{sigmaqu}\end{equation}}

\begin{proof} The involution $\sigma_A$ is birational, and it permutes leaves of the foliation $R=const$ (except for possible contraction of leaves that are  (punctured) lines through $A$, see Footnote 3). 
All but a finite number of  leaves are punctured level curves of the function $R$, since all but a finite number of 
level curves are irreducible (Lemma \ref{lirr}). 
Therefore $\sigma_A$ permutes level curves of the function $R$ (except for at most a finite number of non-irreducible level curves and those containing lines through $A$). Hence, it acts on its values by conformal involution 
$\oc\to\oc$ (by Erasing Singularity Theorem). The latter involution preserves zero, which corresponds to the $\sigma_A$-invariant curve $\gamma$. 
Hence, its action on values of the function $\frac1R$ is either identity, or an affine involution $\mu\mapsto-\mu+b$, $b=const$. This together with the fact that $\sigma_A$ changes sign of the polynomial $y-x^2$ 
 implies the statement of the claim.
\end{proof}

We consider the rational integrals $R$ from the addendum to Theorem \ref{tgerm}, all their indeterminacy points $A$ 
and the corresponding involutions $\sigma_A$ fixing points of $\gamma$. For every pair $(R,A)$,
 assuming $\sigma_A$-invariance of the foliation $R=const$, we will arrive to contradiction.

Everywhere below by $F$ we denote the denominator of the rational function $R$  (written in the 
coordinates $(z,w)$ or $(x,y)$ under consideration). 

2a1) The integral $R$ given by (\ref{exot1}). Let $A$ be the infinity point of the parabola $\gamma=\{ w=z^2\}$. 
Then $\sigma_A(z,w)=(z,2z^2-w)$. The functions $R$ and $R\circ\sigma_A$ have the same foliation by level curves. Therefore, their ratio, which is equal to $\pm\frac{F\circ\sigma_A}{F}$, is constant along each leaf. But the 
latter ratio is constant on the $w$-axis, since the coordinate $z$ is $\sigma_A$-invariant,  
 the polynomial $F$ contains a unique purely  $w$-term $w^k$, and this  
  remains valid for its pullback $F\circ\sigma_A$ with the same $k$. 
On the other hand, the $w$-axis is not a leaf, since $R(0,w)=w^2$. Thus, 
the above ratio is globally constant,  
\begin{equation} F\circ\sigma_A=F\ \text{ up to constant factor.}\label{uptofact}\end{equation}
But  $F(z,w)=\prod_{j=1}^N(w-c_jz^2)^2$, \  $F\circ\sigma_A(z,w)=\prod(w-(2-c_j)z^2)^2$,  
the coefficients $c_j$ in $F$ are negative,  see (\ref{exot1}), while the latter coefficients $2-c_j$ are positive. Therefore, 
equality (\ref{uptofact}) cannot hold, -- a contradiction. 

Let  $A=O=(0,0)$. Consider the chart $(x,y)=(\frac zw, \frac1w)$, in which $A=\infty$, 
$$R(x,y)=\frac{(y-x^2)^{2N+1}}{F(x,y)}, \ F(x,y)=y^2\prod_{j=1}^N(y-c_jx^2)^2, \ c_j<0.$$
Equality (\ref{uptofact}) is proved analogously. The coefficients $c_j$ at $x^2$ in $F$ are negative. 
 But $F\circ\sigma_A(x,y)=(y-2x^2)^2\prod(y-(2-c_j)x^2)^2$, and the coefficients at $x^2$ are positive there. 
 Hence, equality (\ref{uptofact}) cannot hold, -- a contradiction. 

2a2) Case of integral (\ref{exot2}). Treated analogously to the above case. But now 
$R\equiv\infty$ along the $w$-axis, and the above argument does not work as it is. In this case we replace the $w$-axis in the above argument by a parabola $\Gamma_\eta:=\{ w=\eta z^2\}$ 
with a generic $\eta\neq0$. One has 
$\frac{F\circ\sigma_A}{F}|_{\Gamma_\eta}\equiv const$, which follows by substitution 
$w=\eta z^2$. On the other hand, $R|_{\Gamma_\eta}\equiv 
cz$, $c=c(\eta)\neq0$ for a generic $\eta$. Afterwards repeating the above argument 
we get (\ref{uptofact}) and arrive to contradiction as above.
 
 Cases 2b1) and 2b2) have the same complexification; thus we treat only Case 2b2), when the integral $R$  is given 
 by (\ref{exo2bnew2}). There are three indeterminacy points: the infinity and $(\pm i, -1)$. Let $A\in\gamma$ 
 be the infinite point, hence $\sigma_A(z,w)=(z,2z^2-w)$. The denominator $F$ is the product of
  the $\sigma_A$-invariant quadratic polynomial $z^2+1$ and another quadratic polynomial $\Phi(z,w)=z^2+w^2+w+1$. 
  Therefore, the polynomial $F\circ\sigma_A$ is divisible by $z^2+1$, and hence is equal to $\pm F$, 
  by (\ref{sigmaqu}).  This implies that $\Phi\circ\sigma_A=\pm\Phi$ is a quadratic polynomial, while it has obviously degree four, -- 
  a contradiction. Let now $A=(\pm i,-1)$. Let $B$ denote 
 the infinite point of the parabola $\gamma$. Let us choose an affine 
 chart $(x,y)$ centered at $B$ so that its complement (i.e., the corresponding infinity line) 
 is the line tangent to $\gamma$ at $A$, the $y$-axis is the line 
 $\{ z=z(A)\}$, and $\gamma=\{ y=x^2\}$. In the new coordinates one has  
 $R(x,y)=\frac{(y-x^2)^2}{x\Phi(x,y)}$, where $\Phi$ is a cubic polynomial coprime with $y-x^2$. Analogously we 
 get that $\Phi\circ\sigma_A=\pm\Phi$. If $\Phi$ contains a monomial divisible by $y^2$, then 
 $\deg(\Phi\circ\sigma_A)\geq4$, and we get a contradiction. Otherwise $\Phi(x,y)=c(y+\Psi(x))$, where $\Psi$ is a polynomial; $\Phi$ is coprime with $y-x^2$, hence, $\Psi(x)\neq-x^2$. But then 
 $\Phi\circ\sigma_A\neq\pm\Phi$, --  a contradiction. 
 
 Cases 2c1) and 2c2). Note that the integral $R$ from 2c1) is invariant under the order 3 group 
 generated by the symmetry $(z,w)\mapsto(\var z, 
 \var^2 w)$, where $\var$ is a cubic root of unity. 
 This group acts transitively on the set of three indeterminacy points. 
 Thus, it suffices to treat the case of just one indeterminacy point $A$.   Again it suffices to treat only Case 2c2), which has the same complexification, as Case 2c1), with $A$ being the infinite point of the parabola $\gamma$. 
 To do this, 
 let us first recall that  the {\it $(2,1)$-quasihomogeneous degree} of a monomial $z^mw^n$ is the number 
 $m+2n$. A polynomial is {\it $(2,1)$-quasihomogeneous,} if all its monomials have the same 
  $(2,1)$-quasihomogeneous degree.  Each polynomial in two variables is uniquely presented as a finite sum of  $(2,1)$-quasihomogeneous polynomials of distinct quasihomogeneous degrees. The indeterminacy points of the integral $R$ 
  given by (\ref{exoc2}) are $O=(0,0)$, $(1,1)$ and the infinity point of the parabola $\gamma$.

 Taking composition with $\sigma_A(z,w)=(z,2z^2-w)$ preserves the quasihomogeneous 
 degrees. The lower quasihomogeneous part  of the denominator $F$ in (\ref{exoc2}) 
  is the polynomial $(w+8z^2)^2$  of quasihomogeneous degree 4. The numerator 
 $(w-z^2)^3$ is quasihomogeneous of degree 6. Therefore, the lower quasihomogeneous part of the 
 polynomial $F\circ\sigma_A$ is the polynomial $(w-10z^2)^2\neq\pm(w+8z^2)^2$. The two latter statements together 
 imply that formula (\ref{sigmaqu}) cannot hold, -- a contradiction. 
 
 Case 2d). Then we have three indeterminacy points: the origin, the infinity point and the point $(1,1)$. 
 The case, when $A$ is the infinity point of the parabola $\gamma$, is treated analogously to Case 2b2).  
  Let us consider the case, when 
 $A$ is the origin. In the coordinates $(x,y)=(\frac zw,\frac1w)$ one has 
 $\sigma_A(x,y)=(x,2x^2-y)$, and  the function $R$ takes the form 
 $R(x,y)=\frac{(y-x^2)^3}{F(x,y)}$, 
 $$F(x,y)=(y+8x^2)(x-y)(y^2+8x^2y+4y+5x^2-14xy-4x^3).$$
 The lower $(2,1)$-quasihomogeneous part of the polynomial $F$ is $V(x,y):=x(y+8x^2)(4y+5x^2)$. It has 
 quasihomogeneous degree 5, while the numerator in $R$ has quasihomogeneous degree 6. This together with 
 (\ref{sigmaqu}) implies that $\sigma_A$ multiplies  the lower quasihomogeneous part by $\pm1$. But $V\circ\sigma_A(x,y)=x(y-10x^2)(4y-13x^2)\neq\pm V(x,y)$, -- a contradiction. 
 
 Let us now consider the case, when 
 $A=(1,1)$. Take the affine coordinates $(x,y)$ centered at the infinite point of the parabola $\gamma$ such that 
  the complement 
 to the affine chart $(x,y)$ is the tangent line $L$ to $\gamma$ at $A$, the $y$-axis is the zero line $\{ z=1\}$ of the 
 denominator (which passes through $A$), and $\gamma=\{ y=x^2\}$. The rational function $R$ takes 
 the form 
 $$R(x,y)= \frac{(y-x^2)^3}{F(x,y)}, \ F(x,y)=xG_2(x,y)G_3(x,y),  \ G_j(0,y)\not\equiv0, \ \deg G_j=j.$$
 In the chart $(x,y)$ one has $\sigma_A(x,y)=(x,2x^2-y)$. The  factor $x$ in $F$ is $\sigma_A$-invariant. Therefore, 
 $F\circ\sigma_A=\pm F$, by (\ref{sigmaqu}), and $\sigma_A$ leaves invariant the zero locus 
 $Z=\{ G_2G_3=0\}$. The latter zero locus is a union of the conic $\{ G_2=0\}$ disjoint from $A$ 
 and a cubic $\{ G_3=0\}$. The latter  cubic intersects a generic line $\ell$ through $A$ at a unique point distinct from $A$, 
 since it  has a cusp at $A$, see \cite[subsection 7.6,  
 claim 9]{grat}. Thus, for a generic line $\ell$ through $A$,  the intersection $Z\cap(\ell\setminus\{ A\})$  
 consists of three distinct points disjoint from $\gamma$,  
 and it is invariant under the involution $\sigma_A$. This implies that one of them  is fixed (let us denote it by $B$). Hence, for a generic $\ell$ 
 the projective involution $(\sigma_A)|_\ell$ has three distinct fixed points: $A$, $B$ and the unique point 
 of the intersection   $\gamma\cap(\ell\setminus\{ A\})$. Thus, it is identity, -- a contradiction. 
 We have checked that a rationally integrable dual multibilliard with exotic foliation $R=const$ cannot 
 have a vertex $A$ whose involution $\sigma_A$ is not a global projective transformation. This proves Lemma \ref{alt2}.
 \end{proof}
  
\subsection{Pencil case. Proof of Theorem \ref{thmd12}}
  
 We already know that if a dual multibilliard is of pencil type, then 
  it is rationally integrable (Proposition \ref{pratint}). Consider now an arbitrary rationally integrable 
  dual multibilliard where all the curves are conics equipped with a dual billiard structure defined by the same pencil 
  (containing each conic of the multibilliard). Let us show that the multibilliard is of pencil type, that is, its vertices 
  (if any) are from the list given by Definition \ref{multip}, and their collection satisfies the 
   conditions of Definition \ref{multipb}. 
    
\begin{proposition} \label{provert} 
 Let a rationally integrable multibilliard consist of conics lying in a pencil, equipped 
with dual billiard structures defined by the same pencil, and some vertices. 
Then each its vertex is admissible for the pencil. 
\end{proposition}
\begin{proof} Let $\Psi$ be a rational integral of the multibilliard, $\Psi\not\equiv const$. The foliation $\Psi=const$ 
coincides with the pencil under consideration, by Proposition \ref{proconst}. 
 Let $K$ be a vertex of the multibilliard. Then its dual billiard structure is given by an involution 
 $\sigma_K:\cp^2\to\cp^2$ preserving the pencil that is either a global projective involution, 
 or an involution fixing points of a regular conic $\alpha$ 
 passing through $K$ and lying in the pencil (Lemma \ref{alt2}). 
  
 Let us first treat the second case, when $\sigma_K$ fixes points of a regular conic $\alpha$ from the pencil, $K\in\alpha$. 
 Let $L$ denote the tangent line to $\alpha$ at $K$. 
Consider the affine chart $(z,w)$ on the complement $\cp^2\setminus L$ in which $\alpha=\{ w=z^2\}$; the point  
$K$ is the intersection of the infinity line with the $w$-axis. One has $\sigma_K(z,w)=(z,2z^2-w)$. 
Each regular conic $\beta$ of the pencil is given by a quadratic equation $\{w+\Phi_2(z)=0\}$, where $\Phi_2$ is a quadratic polynomial. Indeed, the quadratic equation on $\beta$ contain  neither $w^2$, nor $wz$ terms, since $\sigma_K$ transforms them to polynomials 
of degrees four and three respectively, while it should send $\beta$ to a conic of the pencil. This implies that all 
the regular conics of the pencil are tangent to each other at $K$, and we get that $K$ is an admissible vertex 
from Definition \ref{multip}. 

Let us now treat the first case, when  $\sigma_K$ is a global projective involution. 

{\bf Claim 10.} {\it Let $K$ be not a base point of the pencil, and let $\sigma_K$ is a global projective involution. Then the  conic $\mcs$ through $K$ is a pair of lines.}

\begin{proof} The conic $\mcs$ is fixed by $\sigma_K$, as is $K$. If it were regular, then it would 
intersect a generic line $\ell$ through $K$ at a point distinct from $K$,  and  the involution 
$\sigma_K$ would fix each  point in $\mcs$. Thus, it would not be a projective transformation, -- a contradiction.
\end{proof}

Subcase 1): $K$ is not a base point of the pencil, and the conic 
$\mcs$ through $K$ is a pair of distinct lines $L_1$, $L_2$, both passing through $K$.  
In the case, when the pencil consists of 
conics passing through four distinct base points, there are two base points $V_{j1}$, $V_{j2}$ in each line $L_j$. Therefore, $K$ is a standard vertex $M_s$ from Definition \ref{multip}, and  the projective involution 
$\sigma_K$ permutes $V_{j1}$, $V_{j2}$ for every 
$j=1,2$. Hence, $\sigma_K$ coincides with the involution $\sigma_{M_s}$ from 
Definition \ref{multip}, Case a). 
In the case, when the pencil has three distinct base points, two of them lie in one of the lines, 
say $L_1$,  the third one (denoted by $C$) lies in $L_2$, and $L_2$ is tangent at $C$ to all the regular conics  
of the pencil. We get analogously that $K=M$ and $\sigma_K=\sigma_M$, see Definition \ref{multip}, Case b). 
Case c), when the conics of the pencil are tangent to each other at two base points, is treated analogously. 
Case e) is impossible, since then the pencil contains no distinct line pair. In Case d) it contains the unique pair of distinct lines. They intersect at a base point: the common tangency point of  conics. Hence, this case is impossible. 

Subcase 2): $K$ is not a base point, the conic 
$\mcs$ consists of a line $L_1$ through $K$ and a line $L_2$ that does not pass through $K$. Then 
$\sigma_K$ fixes $K$ and each point of the line $L_2$, and hence, the intersection point $M=L_1\cap L_2$. 
In Case a) there are two base points in each line $L_j$. Those lying in $L_1$ should be permuted 
by  $\sigma_K$. Therefore, the points $K$, $M$  and 
the two  base points in $L_1$ form a harmonic quadruple. Hence $K$ is one of the skew vertices 
from Definition \ref{multip}, Case a). In Case b) there are three distinct base points. 
The case, when two of them lie in $L_1$, is treated as above: then $K=K_{AB}$, see Definition \ref{multip}, b2). In the case, when $L_1$ contains only one 
base point $C$, it should be fixed by $\sigma_K$, as is  $M$. Therefore, $C=M$, 
since the projective involution $(\sigma_K)|_{L_1}$ cannot have three distinct fixed points $C$, $M$, $K$.
Finally, all the three base points are contained in the line $L_2$, which is obviously impossible. 
Similarly in Case c) we get that $K$ is a point of the line through the two base points $A$, $C$. 
Indeed,  the other a priori possible subcase is when  $K$ lies in a line tangent to the conics at a base point (say, $A$). 
 Then   $\sigma_K$ would fix three points of the latter line: $K$, $A$ and the intersection 
point of the tangent lines to conics at $A$ and $C$. The contradiction thus obtained shows that this subcase is impossible. 

In Case d) (conics tangent at a point $A$ with triple contact and passing through another point $B$) 
the point $K$ should lie in the line $L$ tangent to the conics at $A$. (This realizes the vertex $C$ from 
Subcase d2) in Definition \ref{multip}.)  Indeed, otherwise $K$ would lie in $AB\setminus\{ A,B\}$, 
and the involution $\sigma_K$ would fix three distinct points $A,B,K\in AB$, -- a contradiction. Case e) is impossible, since 
then the pencil contains no distinct line pair. 

Subcase 3):  $K$ is not a base point, and the conic $\mcs$  is a double line $L$ through $K$. Then 

- either the pencil is of type c) and $L$ passes through the two base points; the involution $\sigma_K$ should 
permute them and the tangent lines at them to the conics of the pencil; we get $K=M'$, see 
Definition \ref{multip}, c1); 

- or the pencil is of type e) and $L$ is tangent to its regular conics at the unique base point; the involution should 
fix $K$ and those points where the tangent lines to conics pass through $K$; the latter points form a line through 
the base point; we get $K=C$, see 
Definition \ref{multip}, e2).

Subcase 4): $K$ is a base point.  Then each line $L$ passing through $K$ 
and another base point $A$ contains no more base points. Thus, the involution $(\sigma_K)|_L$ 
should fix both $K$ and $A$. 
Finally, $\sigma_K$ fixes each base point. Therefore, the base points different from $K$ lie on the 
fixed point line $\La$ of the projective involution $\sigma_K$. 
Let the pencil be of type a). Then $\La$ contains  three base points, -- a contradiction.
Let now the pencil have type b).  Then $\La$ contains two base points. Hence, they are 
points of transversal intersection with $\La$ of a generic conic of the pencil. Thus, its  
conics are  tangent to each other at $K$, and the vertex $K$ has type b3) from  Definition 
\ref{multip}.  Case of pencil of type c) is treated analogously: $K$ is a vertex of type c2). 

Consider the case of pencil of type d). Let $K$ be the base point $B$ 
of transversal intersection of conics. Then the involution 
$\sigma_K$, which preserves the pencil, should fix $K$ and 
each point of the common tangent line $L$ to the conics at the other base point $A$. Thus, 
$\La=L$. In an affine chart $(z,w)$ centered at $A$, in which 
$B$ is the intersection point of the infinity line with the $Ow$-axis and $L$ is the $z$-axis, 
the involution $\sigma_K$ is the symmetry with respect to the $z$-axis. The latter symmetry changes the 2-jet 
of a regular conic of the pencil at $A$. But its conics should have the same 2nd jet at $A$. Hence, $\sigma_K$ cannot 
preserve the pencil, -- a contradiction. Therefore, $K=A$,  $B\in\La$. Let us choose an affine chart $(z,w)$ centered at $B$ for which $L$ is the infinity line, $K$ is its intersection with the $Ow$-axis, and $\La$ is the $z$-axis. 
Then $\sigma_K$ is again the symmetry as above, and it changes 2nd jets of conics (which are parabolas) 
at their infinity point $K$. Therefore, Case d)  is impossible. Case  e)  is treated analogously.

Finally, all the possible  vertices listed above belong to the list of admissible vertices from Definition \ref{multip}. 
Proposition \ref{provert} is proved.
\end{proof}

The rational integral of the multibilliard is constant on each union of conics of the pencil whose parameter values 
form a $G$-orbit, and the double cardinality $2|G|$ is no greater than the degree of the integral: see  
 the proof of Theorem \ref{thdeg} at the end of Subsection 2.2. Therefore, $G$ is finite. 
 Hence, the multibilliard is of pencil type, by Proposition \ref{progroup}. 
Theorem \ref{thmd12} is proved.

 \subsection{Exotic multibilliards. Proof of Theorem \ref{thmd2}}
 
 \begin{proposition} \label{exeach} Each multibilliard from Theorem \ref{thmd2}   is rationally integrable. 
 The corresponding rational function $R$ from the addendum to Theorem \ref{tgerm} is 
 its integral of minimal degree, except for the subcase in Case (i), when $\rho=2-\frac1{N+1}$; 
 in this subcase $R^2$ is a first integral of minimal degree. 
 \end{proposition}
 \begin{proof} In the case, when there are no vertices, rational integrability follows by Theorem \ref{tgerm}. Let 
 us consider that the multibilliard contains at least one admissible vertex.
 
 Case (i). The function $R$ (respectively, $R^2$) is a rational integral of minimal degree, 
 since $R$ is even (odd) in $z$. 
 
 Case (ii) is treated analogously to Case (i). It suffices to treat the subcase 2b2), since the multibilliards of types 
 2b1), 2b2) containing the unique admissible vertex are projectively isomorphic (Proposition \ref{pthmd2}). In 
 subcase 2b2) the function $R$ is even in $z$, and hence, invariant under the corresponding admissible vertex 
 involution $(z,w)\mapsto(-z,w)$.

Case  (iii). Let us show that the $\gamma$-angular symmetry $\sigma_Q$ centered at each admissible vertex $Q$ 
preserves the integral $R$. 
Cases 2c1) and 2c2) being complex-projectively isomorphic and invariant under order three symmetry 
cyclically permuting the three singular points,  we treat Case 2c2), with $Q=(0,-1)$ being the intersection point of 
 the $w$-axis (i.e., the line through two indeterminacy points: $O$ and $\infty$) and 
the line tangent to $\gamma$ at the indeterminacy point $(1,1)$.  We use the following two claims and proposition. 

\smallskip

{\bf Claim 11.} {\it The polar locus 
$$S:=\{ P(z,w)=8z^3-8z^2w-8z^2-w^2-w+10zw=0\}$$
passes through $Q=(0,-1)$ and has an inflection point there.}

\begin{proof} One has $Q=(0,-1)\in S$ (straighforward calculation). To show that $Q$ is an inflection point, 
it suffices to show that $\nabla P(Q)\neq0$ and the Hessian form of the function $P$ evaluated on the 
skew gradient $(\frac{\partial P}{\partial w}, -\frac{\partial P}{\partial z})$ (which is a function of $Q$ denoted 
by $H(P)(Q)$) vanishes at $Q$. (The latter Hessian form  function  $H(P)$ was introduced in \cite{tab08}.) 
 Indeed, 
$$\frac{\partial P}{\partial z}(Q)=-10, \ \frac{\partial P}{\partial w}(Q)=2-1=1,$$ 
$$\frac{\partial^2 P}{\partial z^2}(Q)=0, \ 
\frac{\partial^2 P}{\partial w^2}(Q)=-2, \ \frac{\partial^2 P}{\partial z\partial w}(Q)=10,$$
$$H(P)=\frac{\partial^2 P}{\partial z^2}\left(\frac{\partial P}{\partial w}\right)^2\hspace{-0.2cm}+\frac{\partial^2 P}{\partial w^2}
\left(\frac{\partial P}{\partial z}\right)^2\hspace{-0.2cm}-2\frac{\partial^2 P}{\partial z\partial w}
\frac{\partial P}{\partial w}\frac{\partial P}{\partial z}=0-200+200=0$$
 at the point $Q$. The claim is proved.
 \end{proof}
 
 \begin{proposition} \label{propinfl} 
  Let a cubic $S\subset\cp^2$ have an inflection point $Q$. Then there exists a projective infolution 
 $\sigma_{Q,S}:\cp^2\to\cp^2$ that fixes each line  through $Q$ and permutes its intersection points with $S$ distinct 
 from $Q$ (for $\ell$ being not the tangent line to $S$ at $Q$). 
 \end{proposition}
 \begin{proof} The above involution  is well-defined on each line $\ell$ through $Q$ distinct from the tangent line 
 $\La$ to $S$ at $Q$ as a projective involution $\sigma_{Q,S,\ell}:\ell\to\ell$ depending holomorphically on $\ell\neq\La$. It suffices to show that the involution family thus obtained extends holomorphically to $\ell=\La$. This will imply that 
 $\sigma_{Q,S}$ is a well-defined global holomorphic involution $\cp^2\to\cp^2$, and hence, a projective transformation. 
 Indeed, let us take an affine chart $(x,y)$ centered at $Q$ and adapted to $S$, so that $\La$ is the $Ox$-axis 
 and the germ of the cubic $S$ is the graph of a germ of holomorphic function:
 $$y=f(x), \ f(x)=ax^3+(b+o(1))x^4; \ \ S=\{ y=f(x)\}.$$
 A line $\ell_\delta:=\{ y=\delta x\}$ with small  slope $\delta$ intersects $S$ at two points distinct from $Q=(0,0)$ with $x$-coordinates 
 $x_0$, $x_1$ satisfying the equation 
 \begin{equation} ax_0^2+(b+o(1))x_0^3=ax_1^2+(b+o(1))x_1^3=\delta.\label{eqx01}\end{equation}
 Taking square root and expressing $x_1$ as an implicit function $-x_0(1+o(1))$ of $x_0$ yields 
 $$x_1=-x_0+(c+o(1))x_0^2, \ \ c=-\frac ba.$$
 Writing the projective involution $\sigma_{Q,S}:\ell_\delta\to\ell_\delta$ fixing the origin and permuting the above intersection points as a fractional-linear transformation in the coordinate $x$,  we get a transformation 
 \begin{equation}x\mapsto-\frac x{1+\nu(\delta)x}.\label{eqnuu}\end{equation}
Substituting $x_0$ to (\ref{eqnuu}) yields 
$$-\frac{x_0}{1+\nu(\delta)x_0}=-x_0+(c+o(1))x_0^2;$$ 
 $$(1+\nu(\delta)x_0)(1-(c+o(1))x_0)=1, \ x_0=x_0(\delta)\to0,$$ 
as $\delta\to0$. Therefore, $\nu(\delta)=c+o(1)\to c$, and the one-parametric holomorphic family 
of  transformation (\ref{eqnuu}) extends holomorphically to $\delta=0$ as the projective 
transfomation $x\mapsto-\frac x{1+cx}$ (continuity and Erasing Singularity Theorem). The proposition is proved.
\end{proof}

{\bf Claim 12.} {\it The polar cubic $S$ is $\sigma_Q$-invariant.}

\begin{proof} The projective involution $\sigma_{Q,S}$ from Proposition \ref{propinfl} fixes $S$ and each 
line through $Q$. 
 It preserves the conic $\gamma$, which is the unique regular conic  tangent to $S$ at the 
three indeterminacy points of the integral $R$. Indeed, if there were 
two such distinct conics, then their total intersection index at the three latter points would be no less than 6, -- 
a contradiction to B\'ezout Theorem. Therefore, $\sigma_{Q,S}$ is the $\gamma$-angular symmetry, and hence, 
it coincides with $\sigma_Q$. This implies that $\sigma_Q(S)=S$. The claim is proved. 
\end{proof}

Claim 12 together with $\sigma_Q$-invariance of the zero locus $\gamma$ of the function $R$ 
  yields $R\circ\sigma_Q=\pm R$. The  restriction of the integral $R$ to the line $\ell$ through $Q$ 
tangent to the conic $\gamma$ at the point $(1,1)$ is holomorphic and nonconstant at $(1,1)$: 
its numerator  restricted to $\ell$ has order 6 zero at $(1,1)$, and  its denominator 
has order 4 zero there. The point 
$(1,1)$ is fixed by  $\sigma_Q$. Therefore,  the above equality holds with sign "$+$" near the point $(1,1)$, 
and hence, everywhere. Case iii) is treated. 
Proposition \ref{exeach} is proved.
\end{proof}

Recall that in our case of a multibilliard containing a conic with an exotic dual billiard structure, 
the involution associated to each vertex is  a projective angular symmetry (Lemma \ref{alt2}). As it is 
shown below, this together with the next proposition implies Theorem \ref{thmd2}.

\begin{proposition} Consider an exotic rationally integrable dual billiard on a conic $\gamma$. Let $R$ be its 
canonical integral. Let $A\in\cp^2$, and 
let $\sigma_A:\cp^2\to\cp^2$ be a projective angular symmetry centered at $A$ that preserves the foliation $R=const$. 
Then  $\sigma_A$ is the $\gamma$-angular symmetry centered at $A$, and $A$ 
belongs to the list from Theorem \ref{thmd2}. 
\end{proposition}
\begin{proof} Set $d=\deg R\in2\nn$. The (punctured) conic $\gamma$ is the only conical leaf of multiplicity $\frac d2$ 
of the foliation $R=const$, see Lemma \ref{lirr}, Statement 3).  Therefore, $\sigma_A(\gamma)=\gamma$. Hence, 
$A\notin\gamma$, since $\sigma_A$ is a projective involution. Thus, its restriction to each line $\ell$ through $A$ 
permutes its intersection points with $\gamma$. Hence, $\sigma_A$ is the $\gamma$-angular symmetry centered at $A$. 
It preserves the set of indeterminacy points of the integral $R$, whose number is either two, or three, since it 
preserves the foliation $R=const$. 
Hence, it preserves the union of  lines tangent to $\gamma$ at the indeterminacy points. 

Consider first Case 2a), where there are 2 intederminacy points: in the affine chart $(z,w)$ these are 
 the origin $O$ and the infinity point 
$B$ of the parabola $\gamma$.  Let  $L_O$, $L_B$ denote the lines tangent to $\gamma$ at them, and let $Q$ be their intersection point. Let us show that $A=Q$: this is Case (i) from Theorem \ref{thmd2}. 
The infinity line $L_B$ is a leaf of the foliation $R=const$, while $L_O$ isn't. Therefore, the lines $L_B$ and $L_O$ 
are $\sigma_A$-invariant. Suppose the contrary: $A\neq Q$. Say, $A\notin L_O$; the case $A\notin L_B$ is 
treated analogously. Then the restriction of the involution $\sigma_A$ to each line $\ell\neq L_B, OB$ through $A$ 
fixes $A$ and its intersection points distinct from $A$ with the lines $L_O$, $L_B$ and $OB$. The number of the latter 
points is at least two, unless $A$ is one of their pairwise intersection points $Q$, $O,B\in\gamma$. But 
$A\neq O,B$, since $A\notin\gamma$, and $A\neq Q$ by assumption, -- a contradiction. Thus, $A=Q$. 

Case 2b). Then there are three indeterminacy points, and we can name them by $X$, $Y$, $Z$ so that 
 $R|_{YZ}\equiv const\neq\infty$, while $R|_{XY}, R|_{XZ}\equiv\infty$.  In Subcase 2b1) $X=(1,1)$, and $Y$, $Z$ are the origin and the infinity point of 
the parabola $\gamma$.   The involution $\sigma_A$ should fix one of the indeterminacy points and permute 
two other ones. We claim that it fixes $X$ and permutes $Y$ and $Z$. Indeed, suppose the contrary: 
$\sigma_A$ fixes, say, $Y$ and permutes $X$ and $Z$. Then $\sigma_A(XY)=YZ$ and $\sigma_A(XZ)=XZ$.  
On the other hand, $\sigma_A$ should send level sets of the 
integral $R$ to its level sets, since this is true for generic, irreducible level sets of degree $\deg R$, and remains 
valid after passing to limit. Thus,  
$\sigma_A$ should preserve the infinity level set, since $\sigma_A(XZ)=XZ$. On  
the other hand, it should permute it with a finite level set containing the line $YZ$, since $\sigma_A(XY)=YZ$.  
The contradiction thus obtained proves that $\sigma_A(X)=X$ and $\sigma_A(Y)=Z$. This implies that $A$ is the 
intersection point of the line $YZ$ with the tangent line to $\gamma$ at $X$, and the corresponding involution $\sigma_A$ 
permutes the intersection point of each line through $A$ with the lines $XY$ and $XZ$. Thus, the pair 
$(A,\sigma_A)$ is the same, as in the Cases (ii) of Theorem \ref{thmd2}. These cases are obtained one from 
the other by complex projective transformation, as in Theorem \ref{tcompl}. 

Case 2c1). (Case 2c2) is obtained from it by projective transformation.) The integral $R$ has three indeterminacy points 
lying on the conic $\gamma$. Let us denote them by $X$, $Y$, $Z$. 
 Then one of them, say, $X$,  is fixed by the $\gamma$-angular symmetry $\sigma_A$, 
and the two other ones are permuted. Indeed, if all of them were fixed, then the three  distinct tangent lines to $\gamma$ 
at them would intersect at the point $Q$, which is impossible: through every point lying outside the conic $\gamma$ 
 there are only two tangent lines to $\gamma$.  This implies that  $A$ is the intersection point of the line tangent to 
 $\gamma$ at $X$ and the line $YZ$. Thus, it is admissible in the sense of Theorem \ref{thmd2}, case (iii). 
 
 In Case 2c2) all the admissible vertices are real, since so are the indeterminacy points. 
 In Case 2c1) $X=(1,1)$ is the unique real indeterminacy point; the other ones are $Y=(\var,\bar\var)$ and $Z=(\bar\var,\var)$, where $\var=e^{\frac{2\pi i}3}$. 
 The intersection of the line $\{ w=2z-1\}$ tangent to $\gamma$ at $(1,1)$ and the line $YZ$ is the admissible 
 vertex $(0,-1)$. 
 Each one of the complex lines $XZ$, $XY$ has non-real slope, and hence,  $X$ is its unique real point. 
 Therefore, the other admissible vertex lying there is not real. Thus, in Case 2c1) the point $(0,-1)$ is the unique real 
 admissible vertex.

 Case 2d). The corresponding integral $R=R_d$ has three indeterminacy points: the origin, the point $(1,1)$ and 
 the infinity point of the conic $\gamma$. The line  $\{ z=1\}$ through the two latter indeterminacy points 
 lies in a level curve of the integral $R$: namely, in its polar locus. On the other hand, $R$ is non-constant 
 on the lines $\{ z=0\}$ and $\{ z=w\}$ passing through the origin and the other indeterminacy points. 
 This implies that every projective transformation preserving the foliation $R=const$ 
 should fix the origin, and hence, the $Oz$-axis: the corresponding tangent line to $\gamma$. Let us show that 
 it cannot be a $\gamma$-angular symmetry. Suppose the contrary: it is the $\gamma$-angular symmetry $\sigma_A$ 
 centered at a point $A$. Then $\sigma_A$ has to fix the origin and to permute the two other indeterminacy points, 
 as in the case discussed above. 
 Therefore, $A$ is the intersection point of the line $\{ z=1\}$ through them and the $Oz$-axis: thus, $A=(1,0)$. 
 Thus, the involution $\sigma_A$ fixes the line $\{ z=1\}$, which lies in the polar locus of the integral $R$. 
 Hence, it preserves the whole polar locus, as in the above Case 2b). The polar locus consists of the above line, 
 the regular conic $\alpha:=\{ w=-8z^2\}$ and an irreducible rational cubic, see \cite[proposition 7.15]{grat}. Hence, 
 $\sigma_A$ fixes the conic $\alpha$.  Thus, it permutes its infinite point (coinciding with that of $\gamma$) 
 and its other, finite intersection point $(1,-8)$ with the line $\{ z=1\}$. On the other hand, it should send the infinite point 
 to the other point $(1,1)\in\gamma\cap\{ z=1\}$, since $\sigma_A$ is the $\gamma$-angular symmetry. 
 The contradiction thus obtained proves that if a multibilliard contains a conic with exotic dual billiard structure of 
 type 2d), then it contains no vertices. The proof of Theorem \ref{thmd2} is complete.  
 \end{proof}
 
 \begin{proof} {\bf of Proposition \ref{pthmd2}.} 
The  complex projective equivalence of billiards of type (ii) is obvious. 
Let us prove the analogous second statement of Proposition \ref{pthmd2} on billiards of type (iii).   
The dual billiard structure on $\gamma$ of type 2c1) 
admits the order 3 symmetry $(z,w)\mapsto(\var z,\bar\var w)$ 
cyclically permuting the indeterminacy points of the integral. Therefore, it also permutes cyclically 
admissible vertices and hence, acts transitively on them  and on their unordered pairs. The same statement holds for 
type 2c2), since the dual billiards 2c1) and 2c2) are complex-projectively isomorphic.  
This implies the second statement of Proposition \ref{pthmd2}. 
In the case of type 2c2) the indeterminacy points are real, and hence, so are the admissible vertices, and 
the order 3 symmetry is a real projective transformation. This together with the above discussion implies 
the third statement of Proposition \ref{pthmd2}.
\end{proof}

\subsection{Real admissible vertices (lines) of  (dual) pencils of conics. Proof of Propositions \ref{preal} 
and \ref{prorelines}} 
\begin{proof} {\bf of Proposition \ref{preal}.} 
The ambient projective 
plane $\cp^2$ is the projectivization of a three-dimensional complex space $\cc^3$. 
The complex conjugation involution acting on $\cc^3$ induces its action on 
$\cp^2$, which will be also called conjugation. It sends projective lines to 
projective lines and preserves the complexification of every real pencil of conics. 
An admissible vertex is uniquely determined by its involution. Thus, 
if the latter is real, then so is the former and so is the corresponding fixed point line (conic). 
Conversely, if both the vertex and the fixed point line (conic) are real, then the involution commutes with complex conjugation, and hence, is real. 
Let us show, case by case, that if the involution is projective, then reality of the vertex 
automatically implies reality of the fixed point line (and hence, of the involution). Simultaneously 
we describe those vertices that are always real.

 Case a): pencil of real conics through 
four distinct (may be complex) points $A$, $B$, $C$, $D$.  The conjugation permutes the points $A$, $B$, $C$, $D$. Therefore, it permutes 
vertices   $M_1$, $M_2$ and $M_3$. 
Hence, at least one of them is fixed (say, $M_1$), or equivalently, real. The line through 
the other points $M_2$, $M_3$ should be fixed, and thus, real, since 
the union $\{ M_2, M_3\}$ is invariant. Finally, the involution 
$\sigma_{M_1}$ is real. 

Let now the point $K_{AB}$ be real. Then the ambient line $AB$ is invariant under 
the conjugation, since the collection of complex lines through pairs of permuted 
base points of the pencil is invariant and the line $AB$ is uniquely determined by $K_{AB}$. 
Indeed, otherwise $K_{AB}$ would be an intersection point of the line $AB$ with 
a line $EF\neq AB$ through some  distinct base points $E$ and $F$. Hence, it would be 
either some of  base points of the pencil, or some of $M_j$'s, which is clearly impossible. 
  Therefore, $AB$ is a real line, and  the set 
$\{ A, B\}$ is conjugation invariant. Hence, so is $\{ C, D\}$, and 
the fixed point line $CD$ of the involution $\sigma_{K_{AB}}$ is real. 
Thus,  $\sigma_{K_{AB}}$ is also real. 

The case, when $C$, $D$ are real and 
$A$, $B$ aren't is possible. In this case $A$ and $B$ are permuted by the conjugation. 
Hence,  the points $M_2$ and $M_3$ are also permuted,  and thus, they are not real. 
Similarly,  $K_{BC}$ and $K_{AC}$ are permuted, $K_{BD}$ and $K_{AD}$ 
are permuted, and hence, they are not real. 

Case b):  real pencil of conics through 3 points $A$, $B$, $C$  tangent at the point $C$ to the same line $L$. See Fig. \ref{fig3}. The point $C$ and the line $L$ should be obviously fixed by conjugation, and hence, real. 
Therefore, the points $A$, $B$ are either both fixed, or permuted. Hence, the line $AB$ is real, and so is the 
intersection point $M=AB\cap L$. The point $K_{AB}\in AB$ is also real, since complex conjugation acting 
on  complex projective line sends harmonic quadruples to harmonic quadruples and the harmonicity property of 
a quadruple of points is invariant under two transpositions: one permuting  its two first points; 
the other one permuting its two last points. Therefore, the line $CK_{AB}$ is real, and so is $\sigma_M$. 
The global projective involutions $\sigma_C$ and $\sigma_{K_{AB}}$ are both real, since so are the lines $AB$ and $L$.

Case c): real pencil of conics through two distinct points $A$, $C$ tangent at them to two given lines $L_A$ and 
$L_C$. See Fig. \ref{fig4}.  A priori, the points $A$ and $C$ need not be real. For example, a pencil of concentric 
circles satisfies the above statements with $A=[1:i:0]$, $C=[1:-i:0]$: the so-called isotropic points at infinity. 
The line $AC$ is real, and so is $M=L_A\cap L_C$, since the complex conjugation either permutes $A$ and $C$ 
(and hence, the lines $L_A$ and $L_C$), or fixes them. Therefore, the involution $\sigma_M$ is real. The point $M'$, 
which is an arbitrary point of the complex line $AC$,  needs not  be real. But if it is real, then so is the 
involution $\sigma_{M'}$. Indeed, $\sigma_{M'}$ can be equivalently defined to fix $M'$ and the line through 
 $M$ and the point $K\in AC$ for which the quadruple $M', K, A, C$ is harmonic. If $M'$ is real, then so is 
$K$, as in the above case. Hence, the line $MK$ is real and so is $\sigma_{M'}$. 

 Cases d) and e). Reality of the vertex $A$ is obvious. Let now $C\in L\setminus\{ A\}$ 
 be a vertex from d2) or e2). Clearly $\sigma_C$ can be real only if $C$ is. Let 
  $C$ be real. In Case d2) the fixed point line $AB$ of the involution $\sigma_C$ is  real. Hence, $\sigma_C$ is real. In Case e2) take a real conic $\alpha$ of the pencil. 
 There exists a unique real tangent line to $\alpha$ through $C$ different from $L$.  Let $E$ 
 denote their tangency point. Then $AE$ is
 the fixed point line of the involution $\sigma_C$, by definition, and it is real. 
 Hence, $\sigma_C$ is real. 
 Proposition \ref{preal} is proved.
 \end{proof}
 
 \begin{proof} {\bf of Proposition \ref{prorelines}.} It suffices to prove its  dual version: 
 a  real pencil of conics through four distinct complex base points $A$, $B$, $C$, $D$ has 
 two {\it real} neighbor skew vertices $K_{EF}$, $K_{FS}$, if and only if all the base points are 
 real. See Remark \ref{remdu}. 
 If the base points are real, then clearly so are all the skew vertices. Let us now suppose that   two neighbor vertices, say, $K_{AB}$ and $K_{BC}$ are real, and prove that all the base points 
 are real. Indeed, then the corresponding lines 
 $AB$ and $BC$ are real, see the above proof of Proposition \ref{preal}. Thus, their intersection point $B$ is real. Hence, the other base 
 points $A\in AB$ and $C\in BC$ are real. Hence,  the remaining base point $D$ is real, since the complex conjugation  cannot permute it with another base point.  
 Thus, each base point is real. Proposition \ref{prorelines} is proved.
 \end{proof}

\section{Rationally $0$-homogeneously integrable piecewise smooth projective billiards. 
Proof of Theorems \ref{thmd1pr}, \ref{thmd12pr}, \ref{thmd2pr}}

The first step of their classification is the following lemma.

\begin{lemma} \label{lemcon}
 Let a planar projective billiard with piecewise $C^4$-smooth boundary containing a nonlinear arc 
 be rationally $0$-homogeneously 
integrable. Then the  boundary consists of conical arcs and straightline segments. 
The projective billiard structure of each conical arc is  either of dual pencil type, or an exotic one from Statement 2) of Theorem \ref{tgermpr}. 
\end{lemma}
\begin{proof} A rational $0$-homogeneous integral of the billiard is automatically such an integral for the projective 
billiard on each nonlinear  arc. Hence, by Theorem \ref{tgermpr}, the nonlinear arcs 
are conics, and each of them is equipped with a projective billiard structure either of dual pencil type, or exotic.
\end{proof}

Below we describe the possible combinations of conical arcs and straightline segments equipped with projective billiard 
structures that yield altogether a rationally integrable projective billiard. We reduce this description to the classification 
of rationally integrable dual multibilliards. To do this, we use the projective duality given by orthogonal polarity, 
see Subsection 1.1, which transforms a projective billiard to a dual multibilliard. We show that the former is 
rationally $0$-homogeneously integrable, if and only if the latter is rationally integrable. We present a one-to-one correspondence 
between  integrals of the former  and  of the latter. Afterwards the 
main results on classification of rationally $0$-homogeneously integrable projective billiards follow immediately by duality 
from those on dual multibilliards. 

\subsection{Duality and correspondence between integrals}


Consider a projective billiard in $\rr^2_{x_1,x_2}$ with piecewise $C^4$-smooth boundary where each $C^4$-smooth arc is either strictly convex, or a straightline segment. This holds automatically, 
if all the nonlinear boundary arcs are conical, as in Lemma \ref{lemcon}. 
Consider the ambient plane $\rr^2_{x_1,x_2}$   as the horizontal plane $\{ x_3=1\}\subset\rr^3_{x_1,x_2,x_3}$. 
We identify it with the standard affine chart $\{ x_3=1\}\in\rp^2_{[x_1:x_2:x_3]}$ by tautological projection. Consider the projective duality given by the  orthogonal polarity, see Subsection 1.1. 
Recall that it sends each  line $L\subset\rr^2\subset\rp^2$ to a point in $\rp^2$. In more detail, 
for every $r=(x_1,x_2,1)\in\rr^2$ and $v=(v_1,v_2,0)\in T_r\rr^2$ the line through $r$  
tangent to $v$ is sent to the point $[\mcm_1:\mcm_2:\mcm_3]$, $\mcm=[r,v]$, see Remark 
\ref{removect}. The duality transforms a piecewise smooth projective billiard to a dual multibilliard, 
as explained in Subsection 1.1 and  formalized by the following definition. 
%
%
%
%

\begin{definition} Consider a projective billiard as above. 
Its {\it dual multibilliard} is 
the collection of curves $\alpha^*$ in $\rp^2$ dual\footnote{The curves $\alpha^*$ are also $C^4$-smoothly immersed. In general, for every $m\geq2$ the dual to a $C^m$-smoothly immersed  curve $\alpha$  is a $C^m$-smoothly immersed curve $\alpha^*$. Note that in general, the straightforward parametrization of the dual curve $\alpha^*$ induced by a $C^m$-smooth parametrization of the curve $\alpha$ is not  $C^m$-smooth. But one can choose  
the parameter of the  dual curve to make its parametrization $C^m$-smooth.} to its $C^4$-smooth nonlinear boundary arcs 
$\alpha$, and the points $A$ (called {\it vertices}) dual to the ambient lines of the straightline segments contained in its boundary segments. Each curve $\alpha^*$ and each point $A$ are  equipped with the dual billiard structures dual to the projective billiard structures on the corresponding boundary arcs. See Definitions  \ref{defdual}, \ref{defdp} and the discussions just before them. (Note that a  priori  a vertex in a multibilliard may appear several times, with different dual billiard structures. They correspond  to projective billiard structures on different boundary segments lying in one and the same line.) \end{definition}

We can and will 
consider that the dual multibilliard lies in the projective plane equipped with 
homogeneous coordinates $[\mcm_1:\mcm_2:\mcm_3]$, see the above discussion.

\begin{proposition} \label{proint}  1) A projective billiard is rationally $0$-homogeneously integrable, if and only if 
its dual multibilliard is rationally integrable.

2) Each rational $0$-homogeneous integral of degree $n$ (if any) 
of the projective billiard is a rational $0$-homogeneous function of 
the moment vector $\mcm=(\mcm_1,\mcm_2,\mcm_3)$, $\mcm=[r,v]$, of the same degree $n$. As a function of $\mcm$, it  is an integral of the dual multibilliard.

3) Conversely, let $R$ be a rational integral of the dual multibilliard written 
as a ratio of two homogeneous polynomials of degree $n$ in $(\mcm_1,\mcm_2,\mcm_3)$.  
Then $R([r,v])$ is a rational $0$-homogeneous integral of the projective billiard of the same degree $n$. 
\end{proposition}
\begin{proof} The  statements of the proposition extend \cite[propositions 1.23, 1.24]{grat} (formulated 
for a projective billiard on a connected curve) to projective billiards with piecewise $C^4$-smooth boundary.  
The proofs given in \cite[subsections 9.1, 9.2]{grat} remain valid in this more general case. 
The second part of Statement 2) 
states  that a function of  $\mcm=[r,v]$ 
that is an integral of the projective billiard flow  is an integral of the dual billiard (considered 
as a function of $\mcm$). It follows immediately from the fact that each  integral of the projective billiard flow is reflection invariant, and for a function of $\mcm$ reflection invariance in the 
projective billiard is equivalent to being an integral of the multibilliard. 
See \cite[the end of subsection 9.2]{grat}.
\end{proof}
\begin{corollary} \label{cordegd} 
The minimal degree of rational $0$-homogeneous integral of a projective billiard is equal to 
the minimal degree of rational integral of its dual multibilliard.
\end{corollary}

\subsection{Rationally $0$-homogeneously integrable billiards on surfaces of constant curvature. Proof of Proposition \ref{proeucl}} 

Polynomial integrability implies rational $0$-homogeneous integrability with 
rational integral of degree 2 or 4 (the minimal  degree of a nonlinear polynomial integral). 
Indeed, in this case there always exists  a nontrivial polynomial homogeneous in $v$ first integral $I(x,v)$ of degree $n\in\{2, 4\}$, see Subsection 1.1. Then the ratio $\frac{I(x,v)}{||v||^{n}}$ 
is a nonconstant rational $0$-homogeneous integral of the billiard flow. Let us prove the converse: rational $0$-homogeneous integrability implies polynomial integrability.
 
Let $\Sigma$ be a surface of constant curvature realized as a surface in the space $\rr^3$ equipped with appropriate quadratic form $<\mca x,x>$, $x=(x_1,x_2,x_3)$, $< , >$ is the 
Euclidean scalar product, $\mca$ is a real symmetric matrix:

Case, when $\Sigma=\rr^2$: $\mca=\diag(1,1,0)$, $\Sigma=\{ x_3=1\}$.

Case, when $\Sigma=S^2$: $\mca=Id$, $\Sigma=\{ x_1^2+x_2^2+x_3^2=1\}$.

Case of $\Sigma=\hh^2$: $\mca=\diag(1,1,-1)$, $\Sigma=\{ <\mca x,x>=-1\}\cap\{ x_3>0\}$. 

Consider an arbitrary rationally $0$-homogeneously integrable billiard with piecewise $C^2$-smooth boundary on $\Sigma$. We show that it admits a special rational $0$-homogeneous 
integral, whose denominator is a power of $<\mca v,v>$. Then we deduce that its numerator 
is a non-trivial polynomial integral. 

For every $r=(r_1,r_2,r_3)\in\Sigma$ and $v\in T_r\Sigma$ (treated as a vector in $\rr^3$) set  
\begin{equation}\mcm=(\mcm_1,\mcm_2,\mcm_3)=\mcm(r,v)=[r,v].\label{momcurv}\end{equation}
 Let $\Psi$ be a nonconstant rational $0$-homogeneous integral. Then it is a rational $0$-homogeneous function of $\mcm$. Indeed, consider the tautological projection 
$\pi:\Sigma\to\rp^2$. For simplicity we consider that the billiard is projected to the 
affine chart $\rr^2=\{ x_3=1\}$: the opposite case is treated analogously. Then its image is a projective billiard, and the function $\Psi$ is its rational $0$-homogeneous integral. 
Set $r'=\frac r{r_3}=(x_1,x_2,1)\in\{ x_3=1\}=\rr^2\subset\rp^2$, $v':=d\pi(v)\in T_{r'}\rr^2$, 
$\mcm'=[r', v']$. The  function $\Psi$ is a rational $0$-homogeneous function  $\Psi(\mcm')$ of 
 $\mcm'$, or equivalently, of $\mcm$, by Proposition \ref{proint}, 
$0$-homogeneity  and since the vectors $\mcm'$, $\mcm$ are proportional. 

Let us write $\Psi(\mcm)=\frac{F(\mcm)}{G(\mcm)}$ as an irreducible fraction. Then the complex zero  divisors of the polynomials $F$ and $G$ are 
essentially distinct in the sense that their intersection has complex codimension at least two 
in $\cp^2$. The divisor of one of them, say $F$, is not a multiple of the light conic divisor 
$$\ii:=\{ <\mca x,x>=0\}\subset\cp^2_{[x_1:x_2:x_3]}.$$  
Consider the rational $0$-homogeneous function 
 \begin{equation}R(\mcm):=\frac{F^{2}(\mcm)}{<\mca\mcm,\mcm>^{n}}, \ \ n:=\deg F.\label{rmcm}
 \end{equation}
 We claim that $R^2$ is an integral of the projective billiard. Or equivalently, 
 the same function considered as a 
 function of  $[\mcm_1:\mcm_2:\mcm_3]\in\cp^2$ is an integral of the multibilliard dual to the 
 projective billiard. Indeed, let us show that for every (complex) 
 line $L$ tangent to a (complexified) curve of the 
 multibilliard (or a line through its vertex) the corresponding involution leaves invariant the 
 restriction  $R^2|_L$. Indeed, it leaves invariant the intersection of the line $L$ with the zero 
 divisor of the polynomial $F$ by invariance of the restriction $\Psi|_L$ (since $\Psi$ is an 
 integral of the multibilliard). It also leaves invariant the intersection $L\cap\ii$, since the involution 
 in question is a $\ii$-angular symmetry restricted to $L$, being dual to reflection in a billiard on constant curvature surface, see \cite[subsection 2.1, corollary 2.3 and proposition 2.5]{gl2}. 
 Therefore, the restriction $R|_L$ and its image under the involution have the same divisors, 
 and hence, are proportional. The proportionality coefficient is equal to $\pm1$, since the 
 transformation in question is an involution acting linearly on the space of rational functions on 
 $L$. Hence, $R^2|_L$ is invariant under the involution. Thus, $R^2$ is a nonconstant 
  integral of the 
 multibilliard. Therefore, the function 
 $R^2(\mcm)$, $\mcm=[r,v]$, 
 is a nonconstant rational $0$-homogeneous integral of the flow of the projective billiard, and hence, of the flow of the initial billiard on $\Sigma$. 
 Its denominator $<\mca\mcm,\mcm>^{2n}$ is equal to $<\mca v,v>^{2n}$, by Bolotin's result stating that 
  for every fixed  $r\in\Sigma$ the restricted 
 moment map $v\mapsto[r,v]$ is an isometry of two-dimensional spaces $T_r\Sigma$ 
 and $r^\perp$ equipped with the quadratic form $<\mca x,x>$, see \cite[formula (15), p. 23]{bolotin2}, \cite[formula (3.12), p. 140]{kozlov}. Therefore, the integral $R^2(\mcm)$ of the 
 billiard flow is a ratio of a homogeneous polynomial in $v$ and a trivial polynomial 
 integral $<\mca v,v>^{2n}$. Hence, its numerator is a polynomial integral. It is 
 non-constant along the unit velocity hypersurface $\{<\mca v,v>=1\}$, since 
 $R^2$ is $0$-homogeneous and nonconstant. Therefore, the billiard in question is polynomially 
 integrable. 
 
 Now for the proof of Proposition \ref{proeucl} it remains to show that if a billiard on a constant 
 curvature surface is polynomially integrable, then the minimal degree of rational integral is 
 equal to the minimal degree $d$ of a nonlinear non-trivial polynomial integral $F(\mcm)$: $d\in\{2,4\}$. Indeed, the ratio 
 $$\Phi(\mcm):=\frac{F(\mcm)}{<\mca v,v>^{\frac d2}}=\frac{F(\mcm)}{<\mca \mcm,\mcm>^{\frac d2}}$$ 
 is a rational $0$-homogeneous integral of  degree $d$. 
  Suppose the contrary to the above minimal degree equality. 
 Then the minimal degree of a rational integral is less than $d$. Therefore, $d=4$ and 
 the minimal degree of a rational integral  is equal to 2, since  it is known to be 
  an even number. Let $R$ be a quadratic integral written as a function $R(\mcm)=R(\mcm')$. 
Its restriction to a curve dual to a nonlinear arc of the boundary of the projective billiard is 
constant. Hence, the foliation $R=const$ is a pencil $\mcp$ of conics whose generic conic is nonlinear. Let us show that the light conic $\mathbb I$ 
lies in $\mcp$. Indeed, the dual multibilliard to the projective billiard 
is of pencil type, and the corresponding pencil $\mcp$ is real, since $R$ is an integral and by Theorem \ref{tgerm}. 
 The foliations $\Phi=const$ and $R=const$ coincide, by Proposition \ref{proconst}. Hence, each level curve of the function $\Phi$, which has degree 4, 
  is a union of at most two conics 
 of the pencil $\mcp$. Its polar locus is the light conic $\mathbb I$ taken twice. 
 Therefore, either $\mathbb I\in\mcp$ (which holds, when $\mathbb I$ is irreducible, i.e., 
 $\Sigma\neq\rr^2$), or $\Sigma=\rr^2$, $\mathbb I$ is the pair of lines $\{ x_1=\pm i x_2\}$ 
 and each one of the latter lines taken twice is a conic of the pencil $\mcp$. 
 In the latter case  each conic of the pencil $\mcp$ is a double line. This is impossible, since 
 $\mcp$ should contain a nonlinear conic. Finally, $\mathbb I\in\mcp$. Then 
one can write $R=\frac{Q(\mcm)}{<\mca\mcm,\mcm>}$, 
replacing $R$ by its postcomposition with a M\"obius transformation. Then the polynomial 
$Q(\mcm)$ is a  non-trivial quadratic integral of the  billiard on $\Sigma$. This contradicts our assumption that the minimal degree of  polynomial integral is equal to 4. 
 Proposition \ref{proeucl} is proved.

\subsection{Case of dual pencil. Proof of  Theorems \ref{thmd1pr}, \ref{thmd12pr}, \ref{thdegpr}}
We use the following proposition.

\begin{proposition} \label{produal} Let $\mcp$ be a pencil of conics, $\mcp^*$ be its dual pencil. 

1) Let $\alpha\subset\rr^2\subset\rp^2$ be a conical arc whose ambient conic lies in $\mcp^*$, 
equipped with the  projective billiard structure defined by  $\mcp^*$. 
Then its dual is the dual conical arc $\alpha^*$ equipped with the 
 dual billiard structure of pencil type, defined by the pencil $\mcp$. The converse statement also holds.
 
 2) A planar projective billiard is of dual pencil type, defined by the dual pencil $\mcp^*$, if and only if its 
 dual multibilliard is of pencil type, defined by  $\mcp$. 
 \end{proposition}
 \begin{proof} Statement 1) of the proposition follows from definition. The definitions of dual multibilliard of 
 pencil type and projective billiard of dual pencil type are dual to each other: the standard (skew) admissible lines 
 for the dual pencil $\mcp^*$ are dual to the standard (skew) admissible vertices for the pencil $\mcp$. See Remark \ref{remdu}. 
 This implies Statement 2). 
 \end{proof}

\begin{proof} {\bf of Theorem \ref{thmd1pr}.} Let a projective billiard with piecewise $C^4$-smooth boundary 
containing a nonlinear  arc be rationally $0$-homogeneously 
integrable. Then  its dual multibilliard is rationally 
integrable (Proposition \ref{proint}).  Hence, each its curve is a conic equipped with either   pencil type, or exotic dual billiard structure (Theorem \ref{thmd1}). Thus,  the nonlinear arcs of projective 
billiard boundary are conical,  and each of them is equipped with either a dual pencil type, 
or exotic projective billiard structure.  Let the projective billiard contain at least two arcs of two distinct conics. Then the multibilliard contains their dual conics, which are also distinct. Hence, they are equipped with the dual billiard structure defined by the pencil $\mcp$ 
containing them, each conic of the multibilliard lies in the same pencil $\mcp$ and is equipped with 
the dual billiard structure defined by $\mcp$ (Theorem \ref{thmd1}). Thus, all the conical arcs of the projective billiard 
boundary lie in the dual pencil $\mcp^*$ and are equipped with the projective billiard structures defined by $\mcp^*$, 
by Proposition \ref{produal}, Statement 1). 
Theorem \ref{thmd1pr} is proved.
\end{proof} 

\begin{proof} {\bf of Theorem \ref{thmd12pr}.} Let in a projective billiard all the nonlinear boundary arcs be 
conics lying in a dual pencil $\mcp^*$, 
 with  the projective billiard structures defined by $\mcp^*$. 
Then the curves of the dual multibilliard are conics lying in the pencil $\mcp$,
 equipped with the dual billiard structure defined by $\mcp$. 
The projective billiard is rationally $0$-homogeneously integrable, if and only if the dual multibilliard 
is rationally integrable, by Proposition \ref{proint}, Statement 1). The latter holds,  if and only if 
the multibilliard is of pencil type (Theorem \ref{thmd12}); or equivalently,  the projective billiard is of 
dual pencil type (Proposition \ref{produal}, Statement 2)). Theorem \ref{thmd12pr} is proved.
\end{proof}

Theorem \ref{thdegpr} follows immediately from Theorem \ref{thdeg} and Corollary \ref{cordegd} by duality, since 
(neighbor) skew admissible lines for a dual pencil $\mcp^*$ are dual to  (neighbor) skew admissible vertices for the 
pencil $\mcp$ and vice versa.

\subsection{Exotic projective billiards.   Proof of Theorem \ref{thmd2pr}}

Let a projective billiard have boundary that consists of conical arcs of one and the same conic equipped with an 
exotic dual billiard structure from Theorem \ref{tgermpr}, Case 2), and maybe some straightline segments. 
It is rationally $0$-homogeneously integrable, if and only if  the corresponding dual multibilliard  is rationally 
integrable, by Proposition \ref{proint}, Statement 1). In appropriate coordinates the dual multibilliard 
consists of one conic $\gamma=\{ w=z^2\}\subset\rr^2_{z,w}=\{ t=1\}\subset\rp^2_{[z:w:t]}$ equipped with the corresponding exotic dual billiard structure from Theorem \ref{tgerm} and maybe some vertices. 
It is rationally integrable, if and only if either it has no vertices, or each its vertex is admissible in the sense of 
Theorem \ref{thmd2}. This holds if and only if the ambient lines of the projective billiard 
boundary segments are dual to the admissible vertices, and their corresponding projective billiard 
structures are dual to the dual billiard structures at the vertices. The lines dual to the admissible vertices, equipped 
with the corresponding dual projective billiard structures, will be called 
admissible.  Let us find the admissible lines  
case by case. To do this, we use the following proposition.

\begin{proposition} \label{produf}
 Consider the above parabola $\gamma$ equipped with an exotic dual billiard structure from 
Theorem \ref{tgerm}, Case 2). 
Let $C\subset\rp^2_{[z:w:t]}$ denote the conic orthogonal-polar-dual to $\gamma$. 

1) The projectivization 
$[F]:\rp^2_{[z:w:t]}\to\rp^2_{[x_1:x_2:x_3]}$ of the linear map 
\begin{equation} F:(z,w,t)\mapsto (x_1,x_2,x_3):=(\frac z2, t, w)\label{mapf}\end{equation}
sends  $C$ to the parabola $\{ x_2x_3=x_1^2\}$, which will be now referred to, as $C$,
 $$C\cap\{ x_3=1\}=\{ x_2=x_1^2\},$$
 equipped with the corresponding projective billiard structure from Theorem \ref{tgermpr}, Case 2). 
 
 2) For every point $(z_0,z_0^2)\in\gamma$ the corresponding point of the dual curve $C$ has 
$[x_1:x_2:x_3]$-coordinates $[-z_0:z_0^2:1]$. 

3) The points in $C$ corresponding to  $O=(0,0),  (1,1), \infty\in\gamma$ are  respectively 
$[0:0:1]=(0,0)$, $[-1:1:1]=(-1,1)$, $[0:1:0]=\infty$ in the coordinates $[x_1:x_2:x_3]$ and in the coordinates 
$(x_1,x_2)$ in the affine chart $\rr^2_{x_1,x_2}=\{ x_3=1\}$. 
\end{proposition} 
\begin{proof}
Statements 1) and 2)  follow from \cite[claim 14, subsection 9.4]{grat} and the discussion after it. Statement 3) follows from Statement 2). \end{proof}

Case 2a). The only admissible vertex of the dual billiard on $\gamma$ is the intersection point $Q=[1:0:0]$ 
of the tangent lines to $\gamma$ at the origin and the infinity. It is equipped with the projective involution 
$[z:w:t]\mapsto[-z:w:t]$ fixing the points of the line 
$Ow$ through the origin and the infinity. The duality sends the above tangent lines to the origin and to the infinity 
respectively in the coordinates $(x_1,x_2)$. 
Thus,   the dual line $Q^*$ is the line through the origin and 
the infinity, equipped with the field of lines through the point $[1:0:0]$: the horizontal line field orthogonal to it. 
 
 Case 2b1) (Case 2b2) is treated analogously). The only admissible vertex $Q=(0,-1)$ is the intersection point of the 
 tangent line to $\gamma$ at the point $(1,1)$ and the line $Ow$ through the origin and the infinity. 
 The dual point to the above tangent line is $(-1,1)$, by Proposition \ref{produf}, Statement 3). The dual to the $Ow$-axis 
 is the infinity point $[1:0:0]$: the intersection point of the tangent lines at the origin and at the infinity. 
 Therefore, the  line $Q^*$ dual to $Q$ is the line $\{ x_2=1\}$ through the 
points $(-1,1)$ and $[1:0:0]$. Let us find the corresponding dual projective billiard structure on it. 
 The fixed point line of the involution $\sigma_Q$ is the line $L=\{ w=1\}$. Indeed, 
 $\sigma_Q$ fixes $\gamma$, and hence, the tangency points of the lines through $Q$ tangent to $\gamma$. 
 The latter tangency points are $(\pm1,1)$. Hence, the fixed point line is the line $L$ through them. The line $L$ 
 intersects $\gamma$ at the points $(\pm1,1)$. Their dual lines are tangent to $C$ at the points $(\pm1,1)$, by 
 Proposition \ref{produf}, Statement 3). 
 Therefore, the dual point $L^*$ is the intersection point $(0,-1)$ of the latter tangent lines. 
 Finally, the admissible line $Q^*=\{ x_2=1\}$  is equipped with the field of lines 
 through the point $(0,-1)$. 

Case 2c2). The dual billiard structure on the conic $\gamma$ 
has three base (indeterminacy) points: $(0,0)$, $(1,1)$, $\infty$. Each admissible vertex 
is the intersection point of a line tangent to $\gamma$ at one of them and the line through two other ones. The admissible 
vertices are $(1,0)$, $(0,-1)$ and $[1:1:0]$. Let us find their dual lines and the projective billiard structures on them. 
The point $Q=(1,0)$ is the intersection point of the $Oz$-axis (which is tangent to $\gamma$ at $(0,0)$) 
 and the line $\{ z=1\}$ (which is the line through the points $(1,1)$ and $\infty$). 
 The dual point to the $Oz$-axis is  the origin $(0,0)$. The dual point to the line $\{ z=1\}$ is 
 the point $[-1:2:0]$. Indeed, it is  the point of intersection of the lines tangent to $C$ at the 
points $(-1,1)$ and infinity,  by Proposition \ref{produf}, Statement 3). The latter intersection point 
is  $[-1:2:0]$, since the line tangent to $C$ at $(-1,1)$ has slope $-2$. Finally, the admissible line $Q^*$ dual to 
$Q$ is the line $\{ x_2=-2x_1\}$ through the origin and the point $[-1:2:0]$. Let us find its projective billiard structure. 
The lines through $Q$ tangent to $\gamma$ are the $Oz$-axis and the line $\{ w=4(z-1)\}$, with tangency points 
$(0,0)$ and $(2,4)$ respectively. Therefore, the fixed point line of the involution $\sigma_Q$ is the line $L=\{ w=2z\}$ 
through them. Its dual point $L^*$ is the intersection of the lines dual to  $(0,0), (2,4)\in\gamma$. 
The latter lines  are tangent to $C$ at the points $(0,0)$ and $(-2,4)$, by Proposition 
\ref{produf}, Statement 2), and they intersect at $(-1,0)$. Thus,  $L^*=(-1,0)$, and the admissible line 
$Q^*=\{ x_2=-2x_1\}$ is equipped with the field of lines through  $L^*=(-1,0)$. 

The line dual to the admissible vertex $(0,-1)$ is the line $\{ x_2=1\}$ equipped with the field of lines through the 
point $(0,-1)$, as in Case 2b1). 

The admissible vertex $[1:1:0]$ is the intersection point of the line tangent to $\gamma$ at infinity and the line 
through the points $(0,0)$ and $(1,1)$. Therefore, its dual line passes through infinity (i.e., is parallel to the $Ox_2$-axis) 
and  the intersection point of the tangent lines to $C$ at the corresponding points $(0,0)$ and $(-1,1)$. 
The latter intersection point is $(-\frac12,0)$. Hence, the line dual to $[1:1:0]$ is $\{ x_1=-\frac12\}$. It intersects 
the conic $C$ at two points: the infinity and the point $(-\frac12,\frac14)$. 
Its projective billiard structure is the field of lines through the intersection point of the tangent lines to $C$ at the 
two latter points, as in the above cases. Their intersection point is $[-1:1:0]$, since the slope of the 
 tangent line to $C$ at $(-\frac12,\frac14)$ is equal to $-1$. Finally, the admissible line $[1:1:0]^*=\{ x_1=-\frac12\}$ 
 is equipped with the line field parallel to the vector $(-1,1)$. 

Case 2c1). Then $(0,-1)$ is the unique admissible vertex for the dual billiard. The only admissible 
line is its dual line $\{ x_2=1\}$ equipped with the field of lines through the 
point $(0,-1)$, as in Cases 2c2) and 2b1). 

Case 2d). There are no admissible lines, since the dual billiard has no admissible 
vertices (Theorem \ref{thmd2}). Theorem \ref{thmd2pr} is proved.

\section{Integrals of dual pencil type billiards: examples of degrees 4 and 12. Proof of Theorems \ref{proformint}, \ref{tintpr} and Lemma \ref{lintpr}}
First we prove Theorems \ref{proformint}, \ref{tintpr} and Lemma \ref{lintpr}. Then we provide examples of dual pencil type 
projective billiards with integrals  of degrees 4 and 12. Afterwards we discuss their realization by 
the so-called semi-(pseudo-) Euclidean billiards, with nonlinear part of boundary 
 being equipped with normal line field for the standard (pseudo-) Euclidean form. 

\subsection{Multibilliards of pencil type. Proof of Theorem \ref{proformint}} 
For every admissible vertex $V$ from Definition \ref{multip}, Case a) ($M_j$ or $K_{EL}$) 
 equipped with the corresponding projective involution 
$\sigma_{V}$,  let $\wh V:\rr^3\to\rr^3$ denote 
the linear involution whose projectivization is $\sigma_{V}$. We normalize it to fix the points of the two-dimensional 
subspace projected to the fixed point line of $\sigma_V$ and to act as the central symmetry $\alpha\mapsto-\alpha$ 
on the one-dimensional subspace projected to $V$. Let $\wh V^*$ denote its conjugate, acting on the 
space $\rr^{3*}$ of linear functionals on $\rr^*$. The symmetric square $\sym^2(\rr^{3*})$ is identified with 
the space of homogeneous quadratic polynomials on $\rr^3$. The operators $\wh V^*$ lifted to $\sym^2(\rr^{3*})$ 
will be also denoted by $\wh V^*$. In the proof of Theorem \ref{proformint} we use the two 
following propositions.
\begin{proposition} Let a pencil have type 2a): conics through four different points $A$, $B$, $C$, $D$. 
One has 
\begin{equation} (\wh K_{EL}\wh K_{LF})^3=Id \ \  \text{ for every three distinct } E,L, F\in\{ A, B, C, D\}.
\label{cyclic3}\end{equation}
\end{proposition}
\begin{proof} Let $N\in\{ A, B, C, D\}$ be the point distinct from $E$, $L$, $F$. 
The involutions $\sigma_{K_{EL}}$, $\sigma_{K_{LF}}$ fix $N$; $\sigma_{K_{EL}}$ fixes $F$ and permutes $E$, $L$; 
$\sigma_{K_{LF}}$ fixes $E$ and permutes $L$, $F$. Hence, their product fixes $N$ and makes an order three cyclic 
permutation of the points $E$, $L$, $F$. Thus, $\Pi:=(\sigma_{K_{EL}}\circ\sigma_{K_{LF}})^3$ fixes all the four 
points $A,B, C,D\in\rp^2$, hence $\Pi=Id$. Thus, $(\wh K_{EL}\wh K_{LF})^3=Id$ up to constant factor. 
The latter constant factor should be equal to one, since the operator in question has unit determinant, being a product 
of six involutions $\wh K_{ST}$, each with determinant $-1$. This proves (\ref{cyclic3}).
\end{proof}

Recall that for every line $X\subset\rp^2$ by $\xi_{X}\in\rr^{3*}$ we denote a linear functional vanishing on the two-dimensional subspace in $\rr^3$ projected to $X$.
\begin{proposition} \label{pminusrel}
 1) The subspace  $W\subset \sym^2(\rr^{3*})$ generated by the products $\xi_{EL}(Y)\xi_{(EL)'}(Y)$ 
with $(EL)'$ being the line through the pair of points $\{ E', L'\}:=\{ A, B, C, D\} \setminus\{ E, L\}$, is two-dimensional 
and $\wh V^*$-invariant for every admissible vertex $V$. Each operator $\wh V^*$ corresponding to a standard 
admissible vertex  acts on $W$ as the identity up to constant factor. 

2) The above functionals $\xi_{EL}$ can be normalized so that 
\begin{equation} \wh K_{AB}^*(\xi_{AB}\xi_{CD})=-\xi_{AB}\xi_{CD}, \ 
\wh K_{AB}^*(\xi_{BC}\xi_{AD})=-\xi_{AC}\xi_{BD},\label{minusrel}\end{equation}
and so that analogous formulas hold for the other operators $\wh K_{EL}^*$. 
\end{proposition}
\begin{proof} The zero conics of the  polynomials $\xi_{EL}(Y)\xi_{(EL)'}(Y)$ are the singular 
conics $AB\cup CD$, $BC\cup AD$, $AC\cup BD$ of the pencil of conics through $A$, $B$, $C$, $D$. 
Hence the space $W$ spanned by these polynomials is two-dimensional. 
Its $\wh V^*$-invariance  follows from $\sigma_V$-invariance of the pencil. 
For every  $V\in\{ M_1, M_2, M_3\}$ the involution $\sigma_V$ fixes the three above conics, and 
hence, each conic of the pencil. Thus $\wh V^*|_W=Id$  up to constant factor.  

Let us prove the first formula in (\ref{minusrel}) for arbitrary normalization of the functionals $\xi_{AB}$ and 
$\xi_{CD}$. Every vector $v\in\rr^3\setminus\{0\}$ with $\pi(v)\notin AB\cup CD$ 
is sent by $\wh K_{AB}$ to the opposite side from the hyperplane $\pi^{-1}(CD)$ and to the same side from the 
hyperplane $\pi^{-1}(AB)$, by definition: $\wh K_{AB}$ fixes the points of the former hyperplane 
and acts as central symmetry 
on its complementary invariant subspace $\pi^{-1}(K_{AB})$, which lies in the latter hyperplane. 
Therefore, it keeps the sign of the functional $\xi_{AB}$ and changes the sign of $\xi_{CD}$. Hence, it multiplies 
their product by $-1$, being an involution. 

The involution $\sigma_{K_{AB}}$ permutes the conics $BC\cup AD$ and $AC\cup BD$. Therefore, the functionals 
$\xi_{BC}$, $\xi_{AD}$, $\xi_{AC}$, $\xi_{BD}$ can be normalized so that the corresponding 
products $\xi_{BC}\xi_{AD}$ and $\xi_{AC}\xi_{BD}$ be permuted 
by $\wh K_{AB}^*$ with change of sign. Formula (\ref{minusrel}) is proved. Let us normalize 
$\xi_{AB}$ and $\xi_{CD}$ by constant factors (this does not change formula (\ref{minusrel})) so that the analogue of 
the second formula in (\ref{minusrel}) holds for $\wh K_{BC}^*$:
\begin{equation} \wh K_{BC}^*(\xi_{AC}\xi_{BD})=-\xi_{AB}\xi_{CD}.\label{kbcac}\end{equation}
This together with the second formula in (\ref{minusrel}) and involutivity of the operators $\wh V^*$ imply that 
\begin{equation}\wh K_{BC}^*\wh K_{AB}^*(\xi_{AC}\xi_{BD})=-\wh K_{BC}^*(\xi_{BC}\xi_{AD})=\xi_{BC}\xi_{AD}.
\label{cycle3}\end{equation}
Replacing the right-hand side in (\ref{kbcac}) by $\wh K_{AB}^*(\xi_{AB}\xi_{CD})$, applying $\wh K_{BC}^*$ 
to both sides, denoting $H:=\wh K_{BC}^*\wh K_{AB}^*$, together with (\ref{cycle3}) yield 
\begin{equation} H(\xi_{AB}\xi_{CD})=\xi_{AC}\xi_{BD}, \  H(\xi_{AC}\xi_{BD})=\xi_{BC}\xi_{AD}.\label{hcycl}\end{equation} 
One also has 
\begin{equation} H(\xi_{BC}\xi_{AD})=\xi_{AB}\xi_{CD},\label{cycleclose}\end{equation}
 since $H^3=Id$, by (\ref{cyclic3}). Therefore, 
  \begin{equation} \xi_{AC}\xi_{BD}+\xi_{BC}\xi_{AD}+\xi_{AB}\xi_{CD}=0\label{sum0}\end{equation}
  since the terms in the latter sum form an orbit of order three linear operator acting on a 
   two-dimensional space $W$. 
   Let us now prove the analogues of formula (\ref{minusrel}) for the other 
 $K_{EL}$. To this end,  let us show that 
 \begin{equation} \wh K_{CD}^*=\wh K_{AB}^* \ \text{ on the space } \  W.\label{restreq}\end{equation}
 Indeed, the composition $\sigma_{K_{CD}}\circ\sigma_{K_{AB}}$ fixes the three singular conics 
 (and hence, each conic of the pencil), 
 by definition. Therefore, $\wh K_{CD}^*=\wh K_{AB}^*$ on $W$, up to constant factor. The latter constant factor 
 is equal to one, since the operators in question take equal value at $\xi_{AB}\xi_{CD}$, by the first formula 
 in (\ref{minusrel}) (which holds for $K_{AB}$ replaced by $K_{CD}$). This proves (\ref{restreq}). Formula 
 (\ref{restreq}) together with the other similar formulas, and already proved formula (\ref{minusrel}) for the operators 
 $\wh K_{AB}^*$, $\wh K_{BC}^*$ imply the 
 analogues of (\ref{minusrel}) for $\wh K_{CD}^*$, $\wh K_{DA}^*$. Let us prove its analogue for $\wh K_{AC}^*$. 
 One has 
 \begin{equation} \wh K_{BC}^*\wh K_{AC}^*(\xi_{AC}\xi_{BD})=-\wh K_{BC}^*(\xi_{AC}\xi_{BD})=\xi_{AB}\xi_{CD}, 
 \label{bcac}\end{equation}
 by formula (\ref{minusrel}) for $K_{BC}$. Therefore, $\xi_{AC}\xi_{BD}$, $\xi_{AB}\xi_{CD}$ together with a third 
 vector $\wh K_{BC}^*\wh K_{AC}^*(\xi_{AB}\xi_{CD})$ form an orbit of order three linear operator 
 $\wh K_{BC}^*\wh K_{AC}^*$ acting on $W$, see (\ref{cyclic3}). The sum of 
 vectors in the orbit should be equal to zero. This together with (\ref{sum0}) implies that 
 $$\wh K_{BC}^*\wh K_{AC}^*(\xi_{AB}\xi_{CD})=\xi_{BC}\xi_{AD}.$$
 Applying  $\wh K_{BC}^*$ to this equality yields $\wh K_{AC}^*(\xi_{AB}\xi_{CD})=-\xi_{BC}\xi_{AD}$. 
 Formula (\ref{minusrel}) for $K_{AC}$ is proved. For $K_{BD}$ if follows from its version for $K_{AC}$ 
 as above. Formula (\ref{minusrel}) is proved 
 for all $K_{EL}$. Proposition  \ref{pminusrel} is proved.
 \end{proof}

 {\bf Claim.} {\it Formula (\ref{minusrel}) holds, if and only if formula (\ref{normfunct}), i.e., 
 (\ref{sum0}) holds. 
 Relation (\ref{sum0}) determines the collection of products $\xi_{EL}\xi_{FN}$ uniquely up to common constant 
 factor.}

\begin{proof} Formula (\ref{minusrel}) implies  (\ref{sum0}), as was shown above. The converse follows from the uniqueness statement of the claim, which in its turn 
follows from two-dimensionality of the ambient space $W$. 
\end{proof}

\begin{proof} {\bf of Theorem \ref{proformint}.} Let the linear functionals $\xi_{EL}$ be normalized to satisfy 
(\ref{normfunct}), which is possible by Proposition \ref{pminusrel} and the above claim. Then they satisfy (\ref{minusrel}), 
by the claim. Projective transformations of $\rp^2$ act on rational functions on 
$\rp^2$ (which can be represented as rational $0$-homogeneous functions of $Y=(y_1,y_2,y_3)$).  
The ratio $\frac{\xi_{AB}\xi_{CD}}{\xi_{BC}\xi_{AD}}$ is sent by $\sigma_{K_{AB}}$ to 
$\frac{\xi_{AB}\xi_{CD}}{\xi_{AC}\xi_{BD}}$ etc., by (\ref{minusrel}). This  implies  invariance of the 
degree 12 rational function (\ref{integr12}) under all the involutions $\sigma_{K_{EL}}$. Its invariance under 
the involutions  corresponding to the standard  vertices follows from Proposition \ref{pminusrel}, Statement 1). Thus, the integral in question is invariant under the involutions of all the admissible vertices. 
It is invariant under the involution of tangent line to a conic of the multibilliard, since so are its factors, 
which are constant on the conics of the pencil. Theorem \ref{proformint} is proved.
\end{proof}

\subsection{Dual pencil  billiards. Proof of Theorem \ref{tintpr} and Lemma \ref{lintpr}}
\begin{proof} {\bf of Theorem \ref{tintpr}.} 
Consider a dual pencil type projective 
billiard given by a dual pencil of conics tangent to four distinct lines $a$, $b$, $c$, $d$. 
 Consider  the corresponding dual multibilliard in  $\rp^2_{[\mcm_1:\mcm_2:\mcm_3]}$ obtained by 
 orthogonal polarity duality. It is  of pencil type, 
 defined by the pencil of conics through the points $A$, $B$, $C$, $D$ dual to the latter lines. 
 Expression (\ref{intpr12}) considered as a rational $0$-homogeneous function of $\mcm$ is a well-defined function on 
 $\rp^2$. It is an integral  of the multibilliard, provided that relation (\ref{sum0pr}) holds 
 (Theorem \ref{proformint}). There indeed exist 
 $\chi_{em;fn}$ satisfying (\ref{sum0pr}),  by Theorem \ref{proformint} and since 
 each scalar product $<r(en),\mcm>$ is a linear functional vanishing on the two-dimensional 
 subspace $\pi^{-1}(EN)\subset\rr^3$ corresponding to the line $EN$: $r(en)\perp\pi^{-1}(EN)$,  
  by orthogonal polarity duality. Thus, substituting $\mcm=[r,v]$ to (\ref{intpr12}) yields an integral of the 
   projective billiard, by Proposition \ref{proint}, Statement 3). Theorem \ref{tintpr} is proved.
  \end{proof}
  
  \begin{proof} {\bf of Lemma \ref{lintpr}.} Let us find $\chi_{em;fn}$ from linear equations  implied by (\ref{sum0pr}). 
  Substituting $\mcm=(0,0,1)$ to (\ref{sum0pr}) yields 
  \begin{equation}\sum\chi_{em;fn}=0.\label{sum0l}\end{equation}
  For every intersection point $em$ set 
  $$x(em):=(x_1(em),x_2(em))\in\rr^2.$$
  Let now $\mcm\perp r(ab)$. Then there exists a $v=(v_1,v_2)=(v_1,v_2,0)$ such that 
  \begin{equation}\mcm=[r(ab),v]=(-v_2, v_1,\Delta_{ab}), \ \Delta_{ab}:=[x(ab),v]=x_1(ab)v_2-x_2(ab)v_1,\label{mab}
  \end{equation}
  since $v\mapsto[r,v]$ is a linear isomorphism 
  $\rr^2\to r^\perp$. Substituting (\ref{mab}) to (\ref{sum0pr}) yields 
  $$\chi_{bc;ad}[x(bc)-x(ab),v][x(ad)-x(ab),v]$$
  \begin{equation}+\chi_{ac;bd}[x(ac)-x(ab),v][x(bd)-x(ab),v]=0.\label{chi23}
  \end{equation}
  The vector differences $x(bc)-x(ab)$, $x(bd)-x(ab)$ in  (\ref{chi23}) are  proportional (being parallel 
  to the line $b$), 
  and so are the other vector differences (parallel to $a$). Therefore, equality (\ref{chi23}) is equivalent to the relation 
  \begin{equation}\chi_{bc;ad}\rho\tau+\chi_{ac; bd}st=0.\label{bcac}\end{equation}
  Combining (\ref{bcac}) with (\ref{sum0l}) and normalizing the $\chi_{em;fn}$ so that $\chi_{ac;bd}=1$ yields (\ref{chi-pr}). 
  Lemma \ref{lintpr} is proved.
  \end{proof}
  \begin{remark} The $\chi_{em;fn}$ satisfying (\ref{sum0pr}) can be also found by all the possible substitutions 
  $\mcm\perp r(em)$ for $em=ab, bc, cd, ac, bd, ad$. This yields a system of linear equations. It appears that 
  their matrix has rank two, so that there exists a unique common non-zero solution. Namely, all the 3x3-minors vanish. 
  This follows from two-dimensionality of the subspace $W$ generated by the three quadratic forms 
  $<r(em),\mcm><r(fn),\mcm>$, which in its turn follows from the fact that {\it the three singular conics $AB\cup CD$, 
  $AC\cup BD$, $AD\cup BC$ formed by the lines $EM$ dual to the points $em$  lie in the same pencil of conics 
  through the points $A$, $B$, $C$, $D$.} On the other hand, a direct calculation of 3x3-minors of the matrix of 
  linear equations and equating these minors to zero 
  yields relations on the (oriented) lengths $s$, $\tau$, $\rho$, $t$,  $p=|cd-bc|$, $u=|bc-ac|$, $q=|cd-ad|$, $h=|ad-bd|$, 
  see Fig. \ref{fig-a-int-form}. These relations are given by 
  the following  geometric theorem,  which can also be deduced from Sine Theorem. The author believes that this theorem  
  is well-known, but he did not found a reference to it.
  \end{remark}
 \begin{theorem} In a triangle $XZT$, see Fig. \ref{fig-proj-triang}, let us take arbitrary points $Y$, $I$ 
  on its sides $ZT$ and $XZ$ respectively. Let $V$ denote the intersection point 
  of the lines $XY$ and $TI$. Set 
  $$\rho:=|YV|, \ t:=|VX|, \ s:=|TV|, \ \tau:=|VI|,$$ 
 $$ u:=|TY|, \ p:=|YZ|, \ q:=|IZ|, \ h:=|XI|.$$
Then
$$ \frac{pq}{(p+u)(q+h)}=\frac{\rho\tau}{st}, \ \frac{tp}{(\rho+t)(p+u)}= \frac\tau{s+\tau}, 
\ \frac{sq}{(s+\tau)(q+h)}= \frac\rho{\rho+t}.$$
\end{theorem}
 \begin{SCfigure}
  \centering
 \caption{}\label{fig-proj-triang}
\epsfig{file=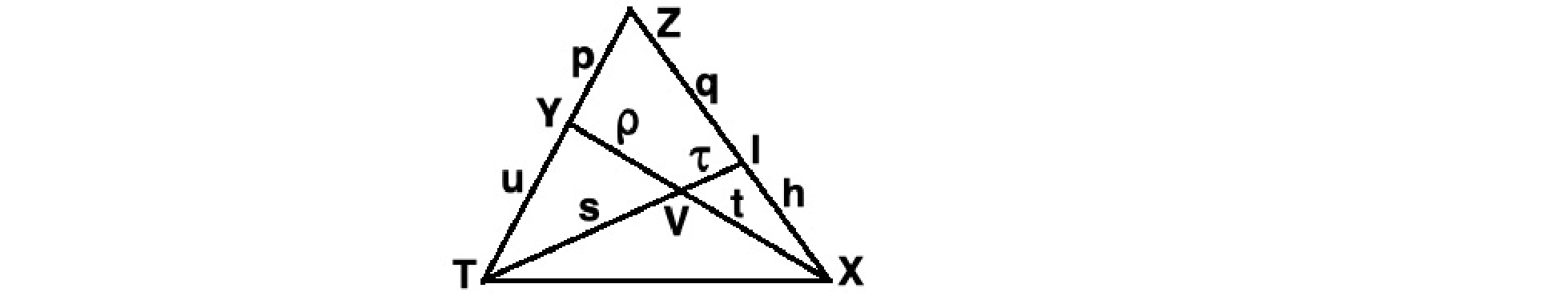, width=30em} 
\hspace{-2cm}
\end{SCfigure}

\subsection{Generic dual pencil billiards  with integrals of degrees 4 and 12 }
Let us construct explicit examples of dual pencil type projective billiards with minimal degree of 
integral being equal to 4 and 12, with non-degenerate dual pencil. 
Consider a dual pencil of conics tangent to four given distinct lines: $a$, $b$, $c$, $d$. 
Fix some its conic $\gamma$. We consider that it is a  closed curve in $\rr^2$. 
Let us equip it with the projective billiard structure defined by the pencil: 
the conics of the pencil are its complex caustics. Let us construct the corresponding admissible 
lines $m_1$, $m_2$, $m_3$ and 
$k_{ef}$, $e,f\in\{ a, b, c, d\}$, $e\neq f$, equipped with their central projective billiard 
structures. We consider that the intersection points $ab$, $bc$, $cd$, $da$ of the 
tangent lines  form a convex quadrilateral 
in which $\gamma$ is inscribed. Then the lines $k_{ac}$ and $k_{bd}$ both intersect the convex domain bounded by 
$\gamma$, see Figures \ref{figd4} and  \ref{fig12}. 

\begin{example} {\bf  of projective billiards with integral of degree 4.} The line $k_{ac}$ cuts the domain bounded by 
$\gamma$ into two pieces. Each of them is a projective billiard bounded by an arc of the curve $\gamma$ and a segment of the line $k_{ac}$, both equipped with the corresponding projective billiard structures.  
Both these projective billiards are  rationally integrable with minimal degree of integral equal to four 
(Theorems \ref{thmd12pr} and \ref{thdegpr}). 
See Fig. \ref{figd4}. 
However the tangency points (marked in bold) of the curve $\gamma$ with the lines $a$, $b$, $c$, $d$ 
are indeterminacy points of the projective billiard structure on $\gamma$. But in each piece we can split 
its boundary arc lying in $\gamma$  into open subarcs separated by the tangency points.  
The projective billiard structure is well-defined on the latter subarcs, and we can consider them  
as smaller smooth pieces of the boundary. A way to exclude the  indeterminacy points from the 
boundary is to cut by the line $m_2$ and to consider a smaller domain, bounded by segments of the lines $k_{ac}$, 
$m_2$ and an arc of the curve $\gamma$, now without indeterminacies  on smooth 
boundary arcs (except for the "corners").  
This yields a  curvilinear triangle equipped with a projective billiard structure, also admitting a rational integral 
of minimal degree 4. 
 \begin{figure}[ht]
  \begin{center}
   \epsfig{file=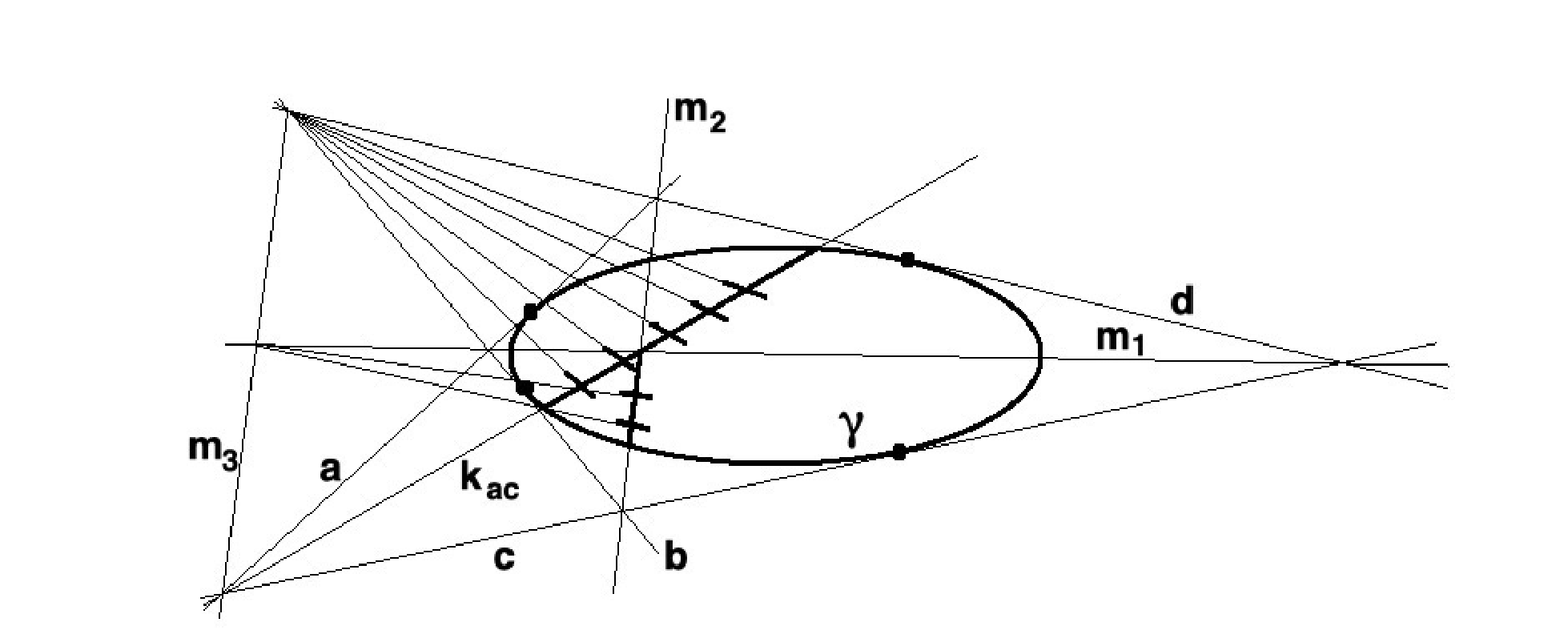, width=35em}
   \caption{Three projective billiards (with boundaries marked in bold) with rational integral of minimal degree 4. 
   The indeterminacies of the projective billiard structure on $\gamma$ are marked in bold.}
        \label{figd4}
  \end{center}
\end{figure}
\end{example}

\begin{example} {\bf  with integral of degree 12.} Consider the two curvilinear quadrilaterals with boundaries 
marked in bold at Fig. \ref{fig12}  as projective billiards. The first  one  is 
bounded by segments of the lines $k_{ad}$, $k_{bd}$, $m_2$ and an arc of the conic $\gamma$. The second one is 
bounded by segments of the lines $k_{bd}$, $k_{ab}$, $k_{ac}$ and an arc of the conic $\gamma$. (We need to note 
that the boundary arc in $\gamma$ in at least some of them contains (at least one) tangency point, which is an 
indeterminacy point of the projective billiard structure.) 
Both quadrilaterals considered as projective billiards are rationally integrable with minimal degree of integral 
being equal to 12, by Theorems \ref{thmd12pr} and \ref{thdegpr}. 
 \begin{figure}[ht]
  \begin{center}
   \epsfig{file=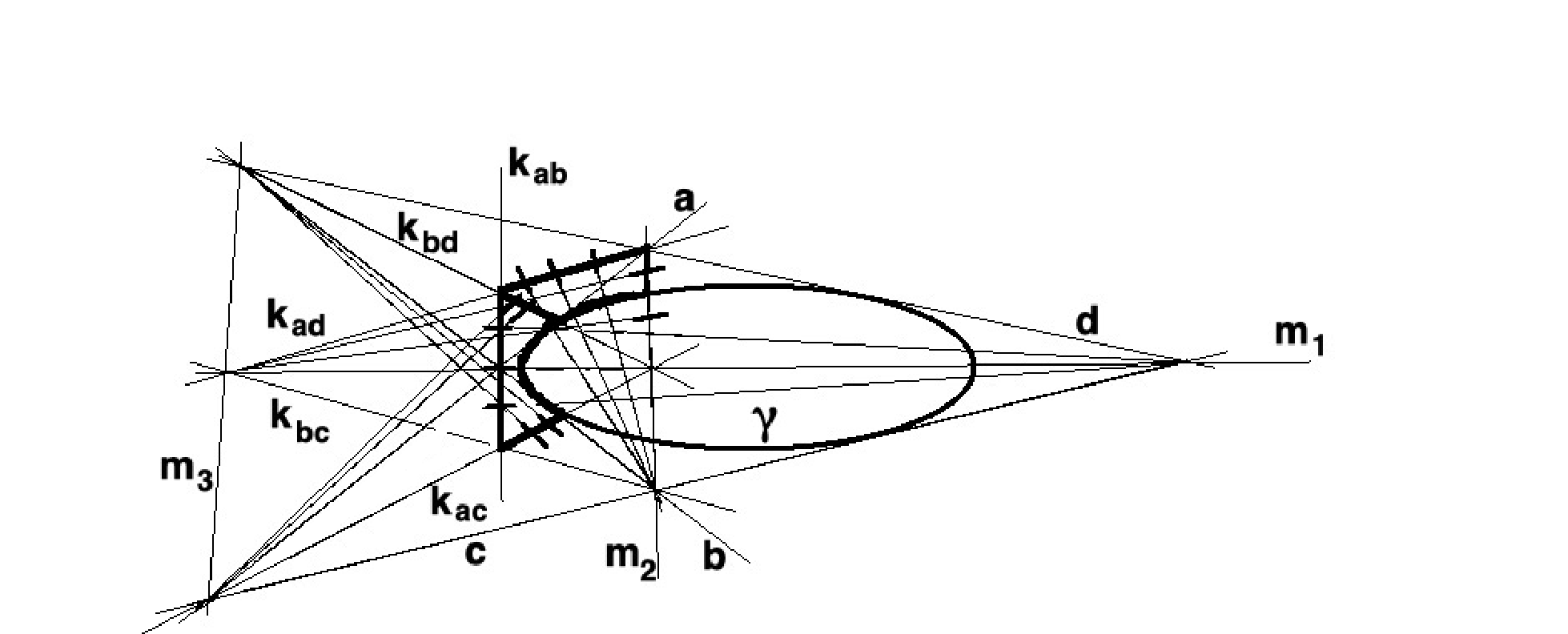, width=35em}
   \caption{Two projective billiards (with boundaries marked in bold) with  integral of minimal degree 12.}
        \label{fig12}
  \end{center}
\end{figure}
\end{example}

\subsection{Semi-(pseudo-) Euclidean billiards with integrals of different degrees}
\begin{definition} A projective billiard in $\rr^2_{x_1,x_2}$ with piecewise smooth boundary 
 is called {\it semi-Euclidean (semi-pseudo-Euclidean)}, if the nonlinear part of the boundary, i.e., its complement 
 to the union of straightline intervals contained there, is equipped 
 with normal line field for the  Euclidean metric $dx_1^2+dx_2^2$ (respectively, for the  
 pseudo-Euclidean metric $dx_1^2-dx_2^2$). 
 \end{definition}
 
 \begin{theorem} \label{thdegsemi} 
 A semi-Euclidean billiard is rationally $0$-homogeneously integrable, if and only if  
 the nonlinear part of its boundary is a finite union of confocal conical arcs and segments of some of the  
 {\bf admissible real lines} (listed below) for the corresponding confocal pencil of conics: 
 
 Case 1), pencil of confocal ellipses and hyperbolas: 
 
 - the two symmetry axes of the ellipses, equipped with normal line field;
 
 - the lines $L_1$, $L_2$ through the foci $F_1$, $F_2$,  orthogonal to the  line $F_1F_2$, 
 each $L_j$ is equipped with the field of lines through the other focus $F_{2-j}$. 
 
 \noindent The billiard has quadratic integral, if and only if its boundary contains no segments of lines $L_{1,2}$; 
 otherwise the minimal degree of integral is  four. 
 
 Case 2), pencil of confocal parabolas:
 
 - the common axis of the parabolas;
 
 - the line $L$ through the focus that is orthogonal to the axis. 

\noindent Both lines are equipped with the normal line field. 
The billiard has quadratic integral, if and only if its boundary contains no segments of 
the  line $L$;  otherwise the minimal degree of integral is  four.
\end{theorem}
\begin{proof} The other admissible lines from Definition \ref{multipr} are not finite real lines. 
For example, in Case 1) the dual pencil of confocal conics consists of conics tangent to two given pairs of lines  
 through the two isotropic points $[1:\pm i:0]$ at infinity. In this case the only 
real skew admissible lines are the lines $L_1$ and $L_2$, and they are opposite as skew admissible lines: they 
correspond to two opposite intersection points of the above tangent lines, namely, the foci $F_1$ and $F_2$. 
Similarly in Case 2) the only real skew admissible line is $L$. 
  This together with Theorem \ref{thdegpr} proves Theorem \ref{thdegsemi}.
\end{proof}

\begin{example} Consider an ellipse and a line $L_1$ through its left focus $F_1$ that is orthogonal 
to the foci line. See Fig. \ref{fig-semi}, the left part. 
Consider the dashed domain bounded by the  intersection segment of the line $L_1$ with the 
ellipse interior and the left elliptic arc; the latter arc is equipped with normal line field, and the segment 
with the field of lines through the other focus $F_2$. This projective billiard admits a rational $0$-homogeneous 
integral of minimal degree four. As the second focus $F_2$ tends to infinity so that the ellipse tends to 
a parabola with the focus $F=F_1$, the above billiard converges to a usual billiard (with normal line field) 
bounded by a segment of the line $L$ through $F$ orthogonal to the parabola axis and by a parabola arc. 
See the right part of Fig. \ref{fig-semi}.
The latter  parabolic billiard is known to have a polynomial integral of minimal degree four 
(and hence, a rational $0$-homogeneous integral of the same degree). It was first discovered in \cite{gram}. 
\end{example}
\begin{figure}[ht]
  \begin{center}
  \hskip-3.5cm
   \epsfig{file=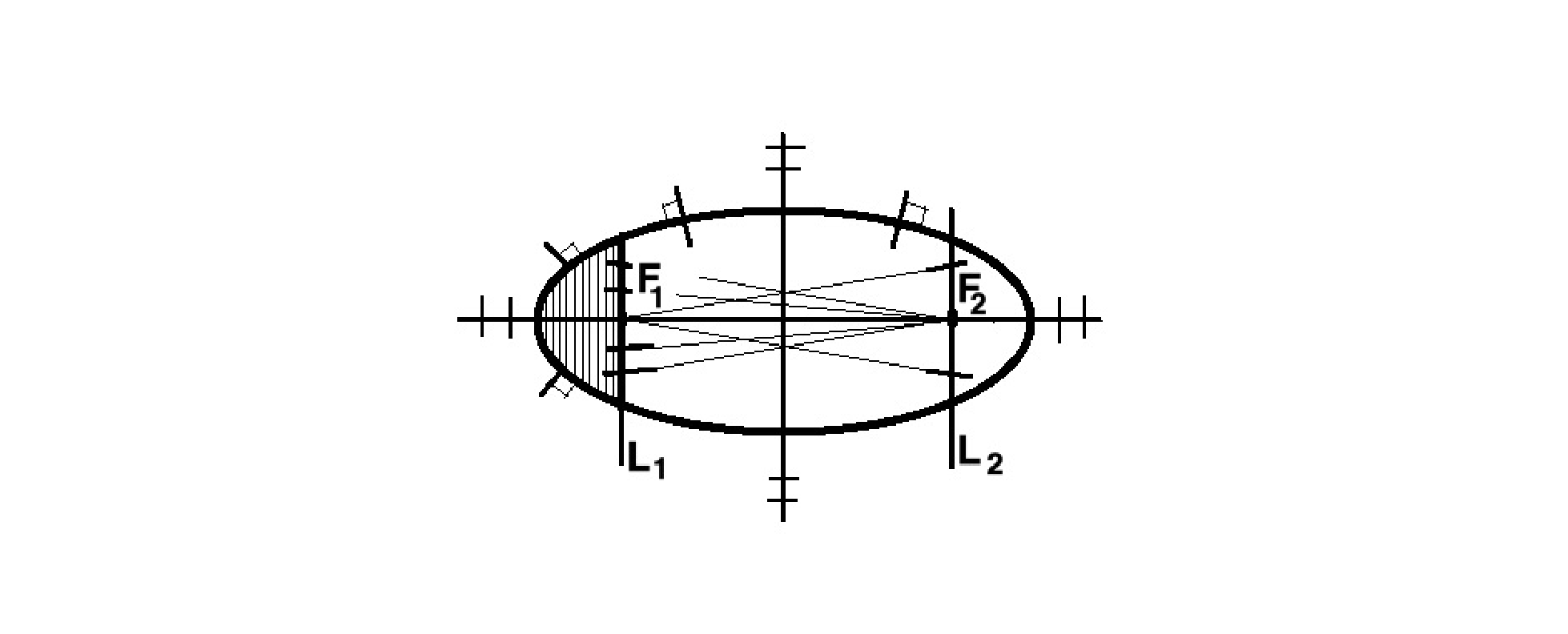, width=33em}
   \hspace{0cm}
    \epsfig{file=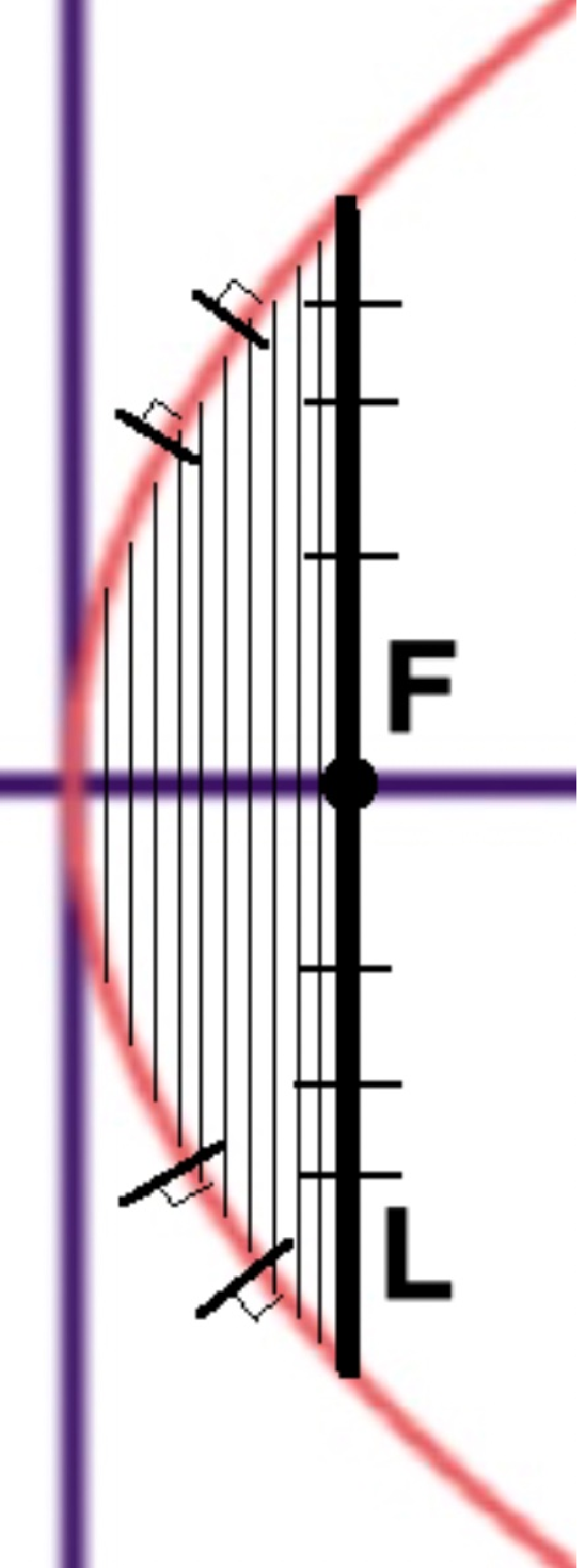, width=4em}
   \caption{Billiards (dashed) with degree 4 integrals. On the left: the semi-Euclidean billiard  
    bounded by a segment of the line $L_1$ and an elliptic arc. As the  ellipse degenerates to a horizontal 
    parabola, it tends to the Euclidean billiard on the right discovered in \cite{gram}, with degree 4 polynomial integral.} \label{fig-semi}
        \label{figsemi4}
  \end{center}
\end{figure}
\begin{proposition} The projective billiards with rational $0$-homogeneous integral of minimal degrees 4 and 12 
presented at Fig. \ref{figd4} and \ref{fig12} respectively are  
realized by  semi-pseudo-Euclidean billiards in $(\rr^2,dx_1^2-dx_2^2)$.
\end{proposition}
\begin{proof} Take the points $bd=b\cap d$ and $ac=a\cap c$ to be the 
isotropic points at infinity $[1:\pm1:0]$. 
\end{proof}

\section{Acklowledgments}

I wish to thank Sergei Tabachnikov for introducing me to projective billiards and helpful discussions. I wish to thank Sergei Bolotin, Valery Kozlov, Constantin Shramov, Dmitry Treschev, 
  Suzanna Zimmermann, Julie Deserti, Fr\'ed\'eric Mangolte   for helpful discussions.

 \end{document}